\documentclass[11pt, 3p]{elsarticle}

\usepackage{setspace}
\usepackage{microtype}
\usepackage{paralist}
\usepackage{parskip}
\usepackage{url}
\usepackage{xcolor}

\usepackage[reqno]{amsmath}
\usepackage{amssymb}
\usepackage{amsthm}
\usepackage{mathrsfs}

\usepackage{adjustbox}
\usepackage{array}
\usepackage{booktabs}
\usepackage{caption}
\usepackage{colortbl}
\usepackage{float}
\usepackage{graphicx}
\usepackage{longtable}
\usepackage{lscape}
\usepackage{rotating}
\usepackage{subcaption}

\usepackage[algoruled,linesnumbered]{algorithm2e}
\usepackage{lipsum}
\usepackage{todonotes}

\usepackage{tikz}
\usetikzlibrary{patterns,fit,arrows,shapes,positioning,decorations.markings}

\usepackage[colorlinks,
            citecolor=blue,
            urlcolor=blue,
            linkcolor=blue,
            bookmarks=false]{hyperref}
\usepackage{cleveref}

\newtheorem{proposition}{Proposition}

\definecolor{lightcyan}{rgb}{0.9, 1, 1}
\definecolor{pastelgreen}{RGB}{205, 249, 189}
\definecolor{pastelblue}{RGB}{201, 246, 255}
\definecolor{verdone}{RGB}{0,128,0}
\definecolor{mygreen}{RGB}{51,204,51}
\definecolor{mblue}{RGB}{0,112,192}
\definecolor{azzurro}{RGB}{1, 120, 229}

\newcommand{\Rev}[1]{#1}
\newcommand{\RevTwo}[1]{#1}

\newcommand{\UB}{\mathrm{UB}}

\newcommand{\pe}{c}

\newcommand{\pev}{\boldsymbol{c}}

\newcommand{\spacedvdots}{%
  \mathbin{%
    \vcenter{\baselineskip=8pt \lineskiplimit=0pt \hbox{.}\hbox{.}\hbox{.}}%
  }%
}

\newcommand{\spacedcdots}{%
  \mathbin{%
    \hbox{.\kern0.5em.\kern0.5em.}%
  }%
}

\tikzstyle arrowstyle=[scale=1.8]
\tikzstyle directed=[postaction={decorate,decoration={markings,
    mark=at position .65 with {\arrow[arrowstyle]{stealth}}}}]
\tikzstyle reverse directed=[postaction={decorate,decoration={markings,
    mark=at position .65 with {\arrowreversed[arrowstyle]{stealth};}}}]

\newcommand{\OSBoundGeneral}{\mathrm{UB}_1(G,\boldsymbol c, \tilde{\mathscr{I}})}
\newcommand{\OSBound}{\mathrm{UB}_1(G,\boldsymbol c,\mathscr{C})}

\newcommand{\OSBunon}{\mathrm{UB}_1(G^1_n,\boldsymbol c^1_n,\mathscr{C}^1_n)}
\newcommand{\OSBduen}{\mathrm{UB}_1(G^2_n,\boldsymbol c^2_n,\mathscr{C}^2_n)}
\newcommand{\OSBtren}{\mathrm{UB}_1(G^3_n,\boldsymbol c^3_n,\mathscr{C}^3_n)}

\newcommand{\OSBall}{\mathrm{UB}_1(G,\boldsymbol c,\mathscr{I})}

\newcommand{\JAPbound}{\mathrm{UB}_2(G,\boldsymbol c,\mathscr{C}, \Rev{\prec})}

\newcommand{\JAPunon}{\mathrm{UB}_2(G^1_n,\boldsymbol c^1_n,\mathscr{C}^1_n, \Rev{\prec^1_n})}
\newcommand{\JAPduen}{\mathrm{UB}_2(G^2_n,\boldsymbol c^2_n,\mathscr{C}^2_n, \Rev{\prec^2_n})}
\newcommand{\JAPtren}{\mathrm{UB}_2(G^3_n,\boldsymbol c^3_n,\mathscr{C}^3_n, \Rev{\prec^3_n})}
\newcommand{\JAPquattrok}{\mathrm{UB}_2(G^4_k,\boldsymbol c^4_k,\mathscr{C}^4_k, \Rev{\prec^4_k})}
\newcommand{\JAPcinquen}{\mathrm{UB}_2(G^5_n,\boldsymbol c^5_n,\mathscr{C}^5_n, \Rev{\prec^5_n})}

\newcommand{\optimum}{\omega(G,\boldsymbol c)}

\newcommand{\optunon}{\omega(G^1_n,\boldsymbol c^1_n)}

\newcommand{\HFBbound}{\mathrm{UB}_3(G,\boldsymbol c,k)}
\newcommand{\HFBquattrok}{\mathrm{UB}_3(G^4_k,\boldsymbol c^4_k,k)}
\newcommand{\HFBunon}{\mathrm{UB}_3(G^1_n,\boldsymbol c^1_n,n)}

\newcommand{\HFBcinquen}{\mathrm{UB}_3(G^5_n,\boldsymbol c^5_n,n+1)}

\newcommand{\Fone}{\mathrm{F1}}
\newcommand{\Feleven}{\mathrm{F11}}

\newcolumntype{T}{>{\ttfamily}l}
\journal{ }

\begin{document}

\begin{frontmatter}
  \title{\LARGE Strength of the Upper Bounds for the Edge-Weighted Maximum Clique Problem}

{
  \author[1]{Fabio Ciccarelli}
  \ead{f.ciccarelli@uniroma1.it}
  
  \author[1]{Valerio Dose}
  \ead{valerio.dose@uniroma1.it}

  \author[1]{Fabio Furini}
  \ead{fabio.furini@uniroma1.it}

  \author[2]{Marta Monaci}
  \ead{marta.monaci@unimercatorum.it}

  \address[1]{Department of Computer, Control and Management Engineering ``Antonio Ruberti'' - Sapienza University of Rome. Via Ariosto 25, 00185 Roma, Italy.  
  }
  \address[2]{Dipartimento di Ingegneria e Scienze - Universitas Mercatorum. Piazza Mattei 10, 00186 Roma, Italy.
  }
}

\begin{abstract}

We theoretically and computationally compare the strength of the \Rev{three} main upper bounds from the literature on the optimal value of the Edge-Weighted Maximum Clique Problem (EWMCP). We provide a set of instances for which the ratio between \Rev{any of the three} upper bounds and the optimal value of the EWMCP is unbounded, \Rev{showing that none of them can give a performance guarantee.} 
\Rev{We further analyze the relative strength among the three upper bounds by determining, for every choice of a ratio between any two of them, the largest values it can attain and providing families of instances for which such values can be reached. Our results show that, for each pair of upper bounds, there exist appropriately chosen instances on which either bound is tighter than the other.}
Our theoretical analysis is complemented by extensive computational experiments on two benchmark datasets: the standard DIMACS instances and randomly generated instances, providing practical insights into the empirical strength of the upper bounds.
\end{abstract}


\end{frontmatter}

\section{Introduction}
\label{sec:intro}

Let \( G = (V, E) \) be a \textit{simple undirected graph}, where \( V \) is the set of vertices, each labeled with a distinct natural number, and \( E \) is the set of edges connecting pairs of vertices. We denote an edge \( \{u,v\} \in E \), with \( u,v \in V \), by \( e \). We assume that each edge \( e \in E \) is associated with a nonnegative integer \emph{edge-weight} \( \pe_{e} \in \mathbb{Z}_{\ge 0} \).
Given a clique $C \subseteq V$ of $G$, namely a subset of pairwise connected vertices, its \emph{total edge-weight} is:
\begin{equation}
    \sum_{e\in E[C]} \pe_{e},  
\end{equation}
where $E[C] \subseteq E$ is the set of edges connecting the vertices in $C$.
The \emph{Edge-Weighted Maximum Clique Problem} (EWMCP) aims to determine a clique of the graph of maximum total edge-weight. The vector $\pev \in \mathbb{Z}_{\ge 0}^{|E|}$, containing all edge-weights, is called the \textit{edge-weight vector}. We denote an instance of the EWMCP by the pair $(G,\boldsymbol c)$. We denote by $\optimum$ the optimal value of the EWMCP and call it the \emph{edge-weighted clique number} of the graph. It is worth noticing that, since all edge-weights are nonnegative, there exists a maximum total edge-weight clique which is inclusion-wise maximal. Figure~\ref{fig:exampleSOL} introduces two EWMCP instances that are used throughout the remainder of the article and also illustrates their corresponding optimal solutions.

\begin{figure}
  \caption{
Two EWMCP instances with the same vertex set and edge set but different edge-weight vectors $\boldsymbol c_a$ and $\boldsymbol c_b$. 
The natural number labeling a vertex is represented inside the vertex, whereas the edge-weight of an edge is reported next to the edge.  In Part \subref{subfig:graph1}, the clique $\{4,5,6\}$, highlighted in green, is a maximum total edge-weight clique for the first instance. The edge-weighted clique number of the graph $\omega(G,\boldsymbol c_a)$ is $ 13$.
  In Part \subref{subfig:graph2}, the clique $\{1,2,3\}$, highlighted in green, is a maximum total edge-weight clique for the second instance. The edge-weighted clique number of the graph $\omega(G,\boldsymbol c_b)$ is $19$.
  \label{fig:exampleSOL}
  }
  \begin{subfigure}[b]{0.48\textwidth}
   \centering 
    \begin{tikzpicture}[scale=1.29]
        \footnotesize

        \pgfmathsetmacro{\a}{1.6}
        \pgfmathsetmacro{\result}{sqrt((\a)^2 - (0.5*\a)^2)}

        \draw (-3/2*\a,\result cm) coordinate (1);
        \draw (-0.5*\a,\result cm) coordinate (2);

        \draw (-\a,0cm) coordinate (3);
        \draw (\a,0cm) coordinate (6);

        \draw (0.5*\a,\result cm) coordinate (4);
        \draw (3/2*\a,\result cm) coordinate (5);

        \draw (-0.5*\a,-\result cm) coordinate (7);
        \draw (0.5*\a,-\result cm) coordinate (8);
        \draw (0,-2*\result cm) coordinate (9);

        \draw[-]  (1) -- node[label=above: $5$] {} (2);
        \draw[-]  (1) -- (3) node[midway, xshift = -2ex, yshift = -1ex] {$3$};
        \draw[-]  (2) -- (3) node[midway, xshift = 2ex, yshift = -1ex] {$2$};
        \draw[-]  (2) -- node[label=above: $2$] {} (4);
        \draw[-]  (3) -- (7) node[midway, xshift = -2ex, yshift = -1ex] {$5$};
        \draw[-]  (4) -- node[label=above: $3$] {} (5);
        \draw[-]  (4) -- (6) node[midway, xshift = -2ex, yshift = -1ex] {$3$};
        \draw[-]  (5) -- (6) node[midway, xshift = 2ex, yshift = -1ex] {$7$};
        \draw[-]  (6) -- (8) node[midway, xshift = 2ex, yshift = -1ex] {$9$};
        \draw[-]  (7) --  node[label=above: $4$] {} (8);
        \draw[-]  (7) -- (9) node[midway, xshift = -2ex, yshift = -1ex] {$3$};
        \draw[-]  (8) -- (9) node[midway, xshift = 2ex, yshift = -1ex] {$2$};

        \draw[circle,fill=white!25,minimum size=17pt,inner sep=0pt] (1) circle (9pt) node[label=center: $1$] {};
        \draw[circle,fill=white!25,minimum size=17pt,inner sep=0pt] (2) circle (9pt) node[label=center: $2$] {};
        \draw[circle,fill=white!25,minimum size=17pt,inner sep=0pt] (3) circle (9pt) node[label=center: $3$] {};
        \draw[circle,fill=green!25,minimum size=17pt,inner sep=0pt] (4) circle (9pt) node[label=center: $4$] {};
        \draw[circle,fill=green!25,minimum size=17pt,inner sep=0pt] (5) circle (9pt) node[label=center: $5$] {};          
        \draw[circle,fill=green!25,minimum size=17pt,inner sep=0pt] (6) circle (9pt) node[label=center: $6$] {};
        \draw[circle,fill=white!25,minimum size=17pt,inner sep=0pt] (7) circle (9pt) node[label=center: $7$] {};
        \draw[circle,fill=white!25,minimum size=17pt,inner sep=0pt] (8) circle (9pt) node[label=center: $8$] {};
        \draw[circle,fill=white!25,minimum size=17pt,inner sep=0pt] (9) circle (9pt) node[label=center: $9$] {};
      \end{tikzpicture}
\subcaption{$\omega(G,\boldsymbol c_a)=13$}
\label{subfig:graph1}
\end{subfigure}
\hfill
\begin{subfigure}[b]{0.48\textwidth}
   \centering 
    \begin{tikzpicture}[scale=1.29]
        \footnotesize

        \pgfmathsetmacro{\a}{1.6}
        \pgfmathsetmacro{\result}{sqrt((\a)^2 - (0.5*\a)^2)}

        \draw (-3/2*\a,\result cm) coordinate (1);
        \draw (-0.5*\a,\result cm) coordinate (2);

        \draw (-\a,0cm) coordinate (3);
        \draw (\a,0cm) coordinate (6);

        \draw (0.5*\a,\result cm) coordinate (4);
        \draw (3/2*\a,\result cm) coordinate (5);

        \draw (-0.5*\a,-\result cm) coordinate (7);
        \draw (0.5*\a,-\result cm) coordinate (8);
        \draw (0,-2*\result cm) coordinate (9);

        \draw[-]  (1) -- (2) node[midway, yshift = 3ex] {$9$};
        \draw[-]  (1) -- (3) node[midway, xshift = -2ex, yshift = -1ex] {$9$};
        \draw[-]  (2) -- (3) node[midway, xshift = 2ex, yshift = -1ex] {$1$};
        \draw[-]  (2) -- (4) node[midway, yshift = 3ex] {$4$};
        \draw[-]  (3) -- (7) node[midway, xshift = -2ex, yshift = -1ex] {$1$};
        \draw[-]  (4) -- (5) node[midway, yshift = 3ex] {$1$};
        \draw[-]  (4) -- (6) node[midway, xshift = -2ex, yshift = -1ex] {$7$};
        \draw[-]  (5) -- (6) node[midway, xshift = 2ex, yshift = -1ex] {$9$};
        \draw[-]  (6) -- (8) node[midway, xshift = 2ex, yshift = -1ex] {$9$};
        \draw[-]  (7) -- (8) node[midway, yshift = 3ex]{$3$};
        \draw[-]  (7) -- (9) node[midway, xshift = -2ex, yshift = -1ex] {$6$};
        \draw[-]  (8) -- (9) node[midway, xshift = 2ex, yshift = -1ex] {$7$};

        \draw[circle,fill=green!25,minimum size=17pt,inner sep=0pt] (1) circle (9pt) node[label=center: $1$] {};
        \draw[circle,fill=green!25,minimum size=17pt,inner sep=0pt] (2) circle (9pt) node[label=center: $2$] {};
        \draw[circle,fill=green!25,minimum size=17pt,inner sep=0pt] (3) circle (9pt) node[label=center: $3$] {};
        \draw[circle,fill=white!25,minimum size=17pt,inner sep=0pt] (4) circle (9pt) node[label=center: $4$] {};
        \draw[circle,fill=white!25,minimum size=17pt,inner sep=0pt] (5) circle (9pt) node[label=center: $5$] {};          
        \draw[circle,fill=white!25,minimum size=17pt,inner sep=0pt] (6) circle (9pt) node[label=center: $6$] {};
        \draw[circle,fill=white!25,minimum size=17pt,inner sep=0pt] (7) circle (9pt) node[label=center: $7$] {};
        \draw[circle,fill=white!25,minimum size=17pt,inner sep=0pt] (8) circle (9pt) node[label=center: $8$] {};
        \draw[circle,fill=white!25,minimum size=17pt,inner sep=0pt] (9) circle (9pt) node[label=center: $9$] {};
      \end{tikzpicture}
\subcaption{$\omega(G,\boldsymbol c_b)=19$}
\label{subfig:graph2}
\end{subfigure}
\end{figure}

The EWMCP is an important generalization of the classical \textit{Maximum Clique Problem} (MCP), which aims to determine a clique in a graph with the largest possible number of vertices. We refer the interested reader to \cite{SANSEGUNDO20231008} for the state-of-the-art exact algorithm for the MCP. We denote by $\omega(G)$ the optimal value of the MCP and refer to it as the \emph{clique number} of the graph. Since the MCP is strongly \(\mathcal{NP}\)-hard, the same holds for the EWMCP \citep{gary1979computers, karp2010reducibility}.

The EWMCP has several real-world applications across many domains, such as robotics, computer vision and materials science \citep{agapito2016accurate, ma2012maximum, san2015novel, san2013improved}. In many of these applications, edge weights encode the strength, similarity, or reliability of pairwise relationships, so that the goal is to identify a clique that is not only mutually compatible, but also maximizes the overall quality of the selected configuration.
We refer the interested reader to other relevant weighted clique-related problems; see, for instance, \cite{CONIGLIO2021435,CORNAZ2017151,DELLEDONNE202139,FURINI2019112,FURINI202154,MATTIA202448,SANSEGUNDO201918}.
 
 Over the past decades, numerous exact approaches have been proposed in the literature to solve the EWMCP to proven optimality \citep{lei2015unified, prokopyev2009equitable, san2019new, shimizu2019branch, shimizu2020maximum}. The main state-of-the-art techniques are combinatorial branch-and-bound algorithms that efficiently explore the search space through an implicit enumeration \Rev{\citep{san2019new, shimizu2019branch, shimizu2020maximum, hosseinian2020lagrangian}}. These algorithms rely on effective bounding techniques to limit the search space. 
 
There are \Rev{three} main upper bounds for the optimal value of the EWMCP, which constitute the basis of the bounding procedures used in the exact algorithms proposed by \citet{san2019new}, \citet{shimizu2020maximum}\Rev{, and \citet{hosseinian2020lagrangian}}. \Rev{Recall that an \emph{independent set} is a subset of pairwise non-adjacent vertices, and that a \emph{coloring} of the vertices is an assignment of colors to the vertices such that no two adjacent vertices receive the same color. Equivalently, it can be viewed as a partition of the vertex set into independent sets, each corresponding to the set of vertices receiving the same color. Since any feasible solution of the EWMCP is a clique, it can contain at most one vertex from each independent set in a coloring of the vertices. This observation plays a key role in the construction of all three upper bounds, although in different ways. The first upper bound is derived from the \emph{Linear Programming} (LP) relaxation of an \emph{Integer Linear Programming} (ILP) formulation of the EWMCP and exploits a coloring of the vertices to define the variables of the model. The second and third upper bounds are purely combinatorial: the former exploits a coloring of the vertices together with an ordering of the corresponding independent sets, while the latter relies on an upper bound on the clique number of the graph, which in practice can also be obtained from a coloring, since the number of colors provides an upper bound on the clique number.}
\Rev{The first two upper bounds are formally defined in Section~\ref{sec:UB}, while the third is defined in Section~\ref{sec:UB3}.} To the best of our knowledge, these \Rev{three} upper bounds were developed independently and in parallel, and no comparison among them has been reported in the literature.

\subsection{\Rev{Performance comparison of exact algorithms for the EWMCP}}

\Rev{From an algorithmic viewpoint, the three upper bounds are at the core of three state-of-the-art exact branch-and-bound algorithms for the EWMCP: \textsc{BBEWC} by \citet{san2019new}, \textsc{MECQ+PLS} by \citet{shimizu2020maximum}, and the algorithm proposed by \citet{hosseinian2020lagrangian}, which we denote by \textsc{HFB}. It is worth noticing that the computational experiments reported in the corresponding papers were performed on different machines and with different time limits, namely 2 hours for \textsc{BBEWC} and 3 hours for both \textsc{MECQ+PLS} and \textsc{HFB}. According to the results reported in the three papers, there are 48 DIMACS instances for which at least one of the three algorithms finds a certified optimal solution within its time limit. 
Among these, \textsc{BBEWC} solves 36 instances, \textsc{MECQ+PLS} solves 47, and \textsc{HFB} solves 38. Moreover, among the 48 instances, 9 are solved exclusively by \textsc{MECQ+PLS} and 1 exclusively by \textsc{BBEWC}, whereas no instances are solved only by \textsc{HFB} within the time limit.
}

\subsection{\Rev{Main contributions of the paper}}

\Rev{The purpose of this article is to compare the upper bounds at a fundamental level: from a theoretical viewpoint, we investigate whether they can provide any performance guarantee with respect to $\optimum$ and whether there exist instances in which one bound is arbitrarily better than another, whereas from a computational viewpoint, we assess their behavior on the benchmark instances classically used in the literature for the EWMCP.}

\Rev{Our first theoretical result, presented in \Cref{sec:WCA}, shows in \Cref{sec:ub1_vs_optimum,sec:ub2_vs_optimum} that the ratio between either of the first two upper bounds and the optimal value of the EWMCP is unbounded. Our second theoretical result, in \Cref{sec:THEOR_COMP}, compares the first two upper bounds in both directions: in \Cref{sec:ub2_vs_ub1} we define a family of instances in which the second upper bound can be arbitrarily worse than the first, whereas in \Cref{sec:ub1_vs_ub2} we exhibit the opposite behavior. These results, while ruling out the possibility of obtaining any relative performance guarantee with respect to the optimal value of the EWMCP, offer insight into the structural properties of the graphs that influence the strength of the first two upper bounds.}

\Rev{In \Cref{sec:UB3}, we introduce the third upper bound on $\optimum$. In \Cref{sec:ub2_vs_ub3}, we compare it with the second upper bound: we prove that the second never exceeds twice the value of the third, that this factor is asymptotically tight, and that the third can be arbitrarily worse than the second. In \Cref{sec:ub3_vs_ub1}, we show that the first and the third upper bounds are incomparable, as either can be arbitrarily worse than the other. Finally, in \Cref{sec:ub3_vs_optimum}, we prove that the third upper bound also lacks a performance guarantee, as its ratio to $\optimum$ can be unbounded.}

\Rev{On the computational side, in \Cref{sec:COMP_COMP} we report extensive results on two sets of instances to empirically assess the relative strength of the three upper bounds. After introducing the benchmark datasets in \Cref{sec:benchmark_datasets}, we first analyze in \Cref{sec:UB2_ordering} the sensitivity of $\UB_2$ to the ordering of the independent sets in the coloring, showing that sorting them in non-decreasing order of cardinality consistently yields the best overall performance and effectively strengthens the practical behavior of this bound. We then report the main computational results on the classical DIMACS benchmark set in \Cref{sec:dimacs_results} and on randomly generated instances in \Cref{sec:random_tests}, showing that the upper bound by \citet{shimizu2020maximum} dominates, in practice, those by \citet{san2019new} and \citet{hosseinian2020lagrangian}. Next, in \Cref{sec:branching_study} we evaluate the three bounds within a simulated branch-and-bound framework, measuring the tree depth at which each bound first drops to or below the optimal value, thereby allowing backtracking. Finally, in \Cref{sec:LP_comparison} we compare the first upper bound with other LP-based upper bounds on $\optimum$ derived from ILP formulations for the EWMCP proposed in the literature. In \Cref{sec:CONCL}, we draw the main conclusions, highlighting the structural properties of graphs that favor one upper bound over the others, and outlining directions for future research.}

\section{The \Rev{first} two upper bounds on $\optimum$}\label{sec:UB}

In this section, we define the \Rev{first} two upper bounds on~$\optimum$ that are analyzed and compared in this article. For completeness, we also include a proof of validity for each \Rev{of the two upper bounds}, following the derivations from the original works in which they were defined.

To this end, we first introduce the necessary notation. Let $I \subseteq V$ be a subset of pairwise non-adjacent vertices, called an \emph{independent set}. Let $\mathscr{I}$ denote the collection of all independent sets of the graph. In addition, let \Rev{$\mathscr{C} \subseteq \mathscr{I}$} be a partition of the vertex set $V$ into 
disjoint, non-empty independent sets.
We call such a partition a \textit{coloring} of the vertices of the graph. We interpret the data of a coloring as an assignment of colors to the vertices such that any pair of connected vertices receives a different color.

\subsection{The first upper bound}
\label{sec:UB1}
Following \citet{san2019new}, we describe in this section the derivation of the first upper bound on~$\optimum$, which is based on a LP relaxation of an ILP formulation of the EWMCP, with the constraints restricted to a given coloring.

For every $u \in V$, let $x_u$ be a binary variable that takes value $1$ if and only if vertex $u$ is selected in the clique. For every edge $e \in E$, let $y_{e}$ be a binary variable that takes value $1$ if and only if both its endpoints $u$ and $v$ are selected in the clique. Using these binary variables, an ILP formulation for the EWMCP reads as follows:
\begin{subequations}
\label{form:EWMCP_ILP}
\begin{align}
\label{eq:form:EWMCP_ILP:obj}
&&\max \quad   \sum_{ e \in {E}} c_{e}y_{e} \\
\label{eq:form:EWMCP_ILP:constrEdgeFirst}
&&~~~  y_{e} &\le x_u, &e \in E,\\[1.2ex]
\label{eq:form:EWMCP_ILP:constrEdgeSecond}
&&~ y_{e} &\le x_v,~    &e \in E, \\[1.2ex]
\label{eq:form:EWMCP_ILP:constrInverse}
&&~ y_{e} &\ge x_u + x_v - 1,~    &e \in E, \\[1.2ex]
\label{eq:form:EWMCP_ILP:constrIndependents}
&& \sum_{u\in I} x_u &\leq 1,~   &I\in\mathscr{I},\\[1.2ex]
\label{eq:form:EWMCP_ILP:var1}
&&~~~  x_{u} &\in \{0,1\}, &u \in V,\\[1.2ex]
\label{eq:form:EWMCP_ILP:var2}
&&~~~  y_{e} &\in \{0,1\}, &e \in E.
\end{align} 
\end{subequations}
The objective function \eqref{eq:form:EWMCP_ILP:obj} of Model~\eqref{form:EWMCP_ILP} maximizes the total edge-weight of the selected clique. Constraints \eqref{eq:form:EWMCP_ILP:constrEdgeFirst}, \eqref{eq:form:EWMCP_ILP:constrEdgeSecond}, and \eqref{eq:form:EWMCP_ILP:constrInverse} linearize the following logical implication: an edge $e \in E$ is part of the selected clique if and only if both its endpoints $u, v \in V$ are part of the selected clique. 
{Since all edge-weights are nonnegative, Constraints~\eqref{eq:form:EWMCP_ILP:constrInverse} can be dropped without affecting the set of optimal solutions of the model. For this reason, they are omitted in the remainder of this article.}
Finally, Constraints \eqref{eq:form:EWMCP_ILP:constrIndependents} impose that, for every independent set of the graph, at most one vertex can belong to the selected clique. This ensures that the feasible solutions of Model~\eqref{form:EWMCP_ILP} correspond to cliques in the graph. The optimal value of Model~\eqref{form:EWMCP_ILP} is the edge-weighted clique number~$\optimum$. 
It is worth noting that, without loss of generality, the binary constraints on the $\boldsymbol{y}\in\{0,1\}^{|E|}$ variables can be dropped, since they are implicitly enforced by the binary nature of the $\boldsymbol{x}\in\{0,1\}^{|V|}$ variables and the structure of the objective function. Therefore, the $\boldsymbol{y}$ variables can be considered as continuous free variables without affecting the correctness of the model.

By imposing the last set of constraints of Model~\eqref{form:EWMCP_ILP} only on a \Rev{subset $\tilde{\mathscr{I}}$ of $\mathscr{I}$, where each vertex $u \in V$ is contained in at least one independent set of $\tilde{\mathscr{I}}$}, and by relaxing the integrality constraints on the variables~$\boldsymbol{x}$ with nonnegativity constraints, we obtain a LP relaxation of the ILP Model~\eqref{form:EWMCP_ILP}. 

To derive the dual of this LP relaxation, we introduce, for each edge $e \in E$, two dual variables $\varrho_{eu}$ and $\varrho_{ev}$ \Rev{associated with the Constraints \eqref{eq:form:EWMCP_ILP:constrEdgeFirst} and \eqref{eq:form:EWMCP_ILP:constrEdgeSecond}}, respectively. In addition, for each $I \in \Rev{\tilde{\mathscr{I}}}$, we define a dual variable $\pi_I$ \Rev{corresponding to the associated Constraint \eqref{eq:form:EWMCP_ILP:constrIndependents}}. 
The dual of the LP relaxation described above reads as follows:
\begin{subequations}\label{form:LP-italia-spagna}
    \begin{align}
        \label{obj:LP-italia-spagna_splitting}
        &&\Rev{\OSBoundGeneral} ~:=~\min ~~~\sum_{I \in \Rev{\tilde{\mathscr{I}}}
        } \pi_I \\[2ex]
         \label{con:LP-italia-spagna_splitting} 
        && \varrho_{eu} + \varrho_{ev} \, &= \, c_{e},  & &\quad e \in E, \\[2.5ex]
         \label{con:LP-italia-spagna_max-w}
        && \Rev{\sum_{I \in \tilde{\mathscr{I}} : \, u \in I}} \pi_I & \ge \sum_{e \in \delta(u)} \varrho_{eu}, &  &\quad u \in V,\\[2.5 ex]
         && \pi_{I}  &\ge 0,  & &\quad  I \in \tilde{\mathscr{I}}, \\[2.5ex]
         && \varrho_{ev}  &\ge 0,  & &\quad  e \in E, \\[2.5ex]
         && \varrho_{eu}  &\ge 0,  & &\quad  e \in E, 
    \end{align}
\end{subequations}
where $\delta(u)\subseteq E$ is the set of all edges incident to vertex $u\in V$.
By LP duality, the optimal value of Model~\eqref{form:LP-italia-spagna} provides an upper bound on~$\optimum$.

It is worth noting that the values of the variables~$\boldsymbol{\varrho}$, in any optimal solution of Model~\eqref{form:LP-italia-spagna}, can be interpreted as distributing the weight of each edge among its two endpoints. This interpretation is then used in the remainder of the article to illustrate the examples, and to make a comparison with the second upper bound. In more detail, for each edge $e \in E$, the optimal values of the variables $\varrho_{eu}$ and $\varrho_{ev}$ represent the portions of the edge-weight~$c_{e}$ assigned to the endpoints~$u$ and~$v$, respectively. Moreover, for every $u \in V$, the quantity
\begin{equation}\label{eq:vertex-weight-gamma}
  \gamma_{u}:=  \sum_{e \in \delta(u)} \varrho_{eu}
\end{equation}
can be interpreted as the total weight assigned to vertex~$u$. 

\Rev{When $\tilde{\mathscr{I}} = \mathscr{C}$, as in \citet{san2019new},} Constraints~\eqref{con:LP-italia-spagna_max-w} \Rev{can be written as $$\pi_I \ge \sum_{e \in \delta(u)} \varrho_{eu} \qquad I \in \mathscr{C}, \, u \in I,$$ and} imply that, for each $I \in \mathscr{C}$, the variable~$\pi_I$ must be at least as large as the weight $\gamma_u$ for every~$u\in I$.
In this interpretation, each variable~$\pi_I$ represents then the weight of the independent set~$I$, and the objective function~\eqref{obj:LP-italia-spagna_splitting} computes the total weight across all independent sets in the coloring~$\mathscr{C}$. Then we have:
\begin{equation} \label{eq:ita-spa}
   \OSBound = \sum_{I \in \mathscr{C}} \, \max \, \{\, \gamma_u: u \in I \, \}.
\end{equation}

Note that the upper bound employed in the branch-and-bound algorithm of \citet{san2019new} corresponds to a heuristic solution of Model~\eqref{form:LP-italia-spagna}, \Rev{where $\tilde{\mathscr{I}}$ is a suitably defined vertex coloring $\mathscr{C}$, and it is} computed on subgraphs generated during the branching process. The proposed heuristic procedure is designed to reduce the computational effort while computing good-quality upper bounds.
In contrast, in this paper we focus on evaluating the strength of the upper bound itself. To this end, we consider \Rev{$\OSBound$, which corresponds} to the best possible upper bound within this family.

\subsection{The second upper bound}
\label{sec:UB2}

Following \citet{shimizu2020maximum}, we describe in this section the derivation of the second upper bound on~$\optimum$, obtained by exploiting the structure of feasible EWMCP solutions and a coloring of the vertices of the graph.

\Rev{Given a coloring $\mathscr{C} \subseteq \mathscr{I}$, let $\prec$ be a total order on $\mathscr{C}$, and let $I(u) \in \mathscr{C}$ be the unique independent set of $\mathscr{C}$ containing vertex $u \in V$, i.e., $u \in I(u)$.} 

In this and all subsequent occurrences, in an expression of the form $\max\{c_e:e\in\delta(u),\ v\in I\}$, $v$ denotes the other endpoint of $e=\{u,v\}$, and the maximum is defined as~$0$ if no such edge exists.

For every $u \in V$, we define the quantity
\begin{equation}\label{eq:jap_weight}
    \sigma_u := \sum_{\Rev{I \in \mathscr{C} \, : \, I \prec I(u)}} \max \left\{\, c_e : e \in \delta(u),\ v \in \Rev{I} \,\right\},
\end{equation}
which can be interpreted as the total weight assigned to vertex~$u$. This weight corresponds to the sum, over all independent sets that precede the one containing~$u$ in the ordering, of the maximum edge-weight connecting~$u$ to a vertex in each of those independent sets.

The second upper bound studied in this article, denoted by~$\JAPbound$, is then defined as:
\begin{equation} \label{eq:jap}
    \JAPbound := \sum_{I \in \mathscr{C}} \max \{\, \sigma_u : u \in I \,\}.
\end{equation}

To verify that this is indeed an upper bound for~$\optimum$, consider any clique $C$ in~$G$. By construction, each vertex in the clique~$C$ belongs to a different independent set in the coloring~$\mathscr{C}$. 
For every independent set $I \in \mathscr{C}$, let $u_I$ be the unique vertex in $C$ that belongs to $I$, if such a vertex exists; otherwise, let $u_I$ be any vertex in $I$. We then obtain:

\Rev{\begin{subequations}
\begin{align}
\sum_{e \in E[C]} c_e 
&= \sum_{\substack{I \in \mathscr{C}: \\ C \cap I \neq \emptyset}} \quad \sum_{\substack{J \in \mathscr{C}: \\ J \prec I, C \cap J \neq \emptyset}} c_{\{u_I, u_J\}} \label{eq:jap-0st-row}\\[3ex]
&\le \sum_{\substack{I \in \mathscr{C}: \\ C \cap I \neq \emptyset}} \max \left\{ \sum_{\substack{J \in \mathscr{C}: \\J \prec I,\ C \cap J \neq \emptyset}} c_{\{u, u_J\}} : u \in I \right\} \label{eq:jap-1st-row}\\[2ex]
&\le \sum_{\substack{I \in \mathscr{C}: \\ C \cap I \neq \emptyset}} \max \left\{  \sum_{\substack{J \in \mathscr{C}: \\J \prec I}} \max \left\{ c_e : e \in\delta(u),\ v \in J \right\}  : u \in I \right\} \label{eq:jap-2nd-row}\\[4ex]
&\le \underbrace{ \sum_{I \in \mathscr{C}} \max \{ \sigma_u : u \in I \} }_{= \JAPbound}. \label{eq:jap-3rd-row}
\end{align}
\end{subequations}}

The equality in~\eqref{eq:jap-0st-row} rewrites the sum of the weights of the edges in the clique~$C$ by indexing over pairs of independent sets in~$\mathscr{C}$ that contain vertices of~$C$. 
\Rev{The inequality in~\eqref{eq:jap-1st-row} is obtained by maximizing the inner sum over all possible choices of a vertex $u \in I$. }
The inequality in~\eqref{eq:jap-2nd-row} follows by maximizing each term of the inner sum over all possible choices of a vertex \Rev{$v \in J$} and by extending the sum to all independent sets $J \prec I$.
Finally, the inequality in~\eqref{eq:jap-3rd-row} is obtained by extending the sum to all independent sets in~$\mathscr{C}$.
This derivation follows the proof of \cite[Lemma~1]{shimizu2020maximum}.

\Rev{Note that the upper bound $\JAPbound$ depends not only on a coloring, but also on the total order $\prec$ of the independent sets within the coloring. Using the same partition into independent sets, but varying the ordering of the independent sets, could result in different upper bounds. In 
\Cref{sec:UB2_ordering} we empirically measure the impact of different orderings on the quality of this upper bound.}

It is worth mentioning that \citet{shimizu2020maximum} also proposed an alternative upper bound, which can be interpreted as the value of a specific feasible solution to the dual formulation~\eqref{form:LP-italia-spagna}. The main idea is to assign the full weight of each edge \( e \in E \) to the endpoint \Rev{belonging to the independent set which precedes the other in $\mathscr{C}$ according to the total order $\prec$}.
\Rev{Assuming without loss of generality that for each edge \( e \in E \), we have \( I(u) \prec I(v) \), this feasible solution is defined as follows:}
\[
   \tilde{\varrho}_{eu} := c_e, \qquad
   \tilde{\varrho}_{ev} := 0, \quad  e \in E, \qquad
   \tilde{\pi}_I := \max\left\{ \sum_{e \in \delta(u)} \tilde{\varrho}_{eu} \colon u \in I\right\}, \quad  I \in \mathscr{C}.
\]

The upper bound obtained from this dual solution is, by construction, greater than or equal to the optimal value of the dual Model \eqref{form:LP-italia-spagna}. Accordingly, the second upper bound proposed by \citet{shimizu2020maximum} is dominated by~\(\OSBound\), and is therefore not considered in our study.

\subsection{Examples of the upper bounds computation}\label{sc:examples-figures}

In Figure~\ref{fig:example_os_jap}, we report the values of the \Rev{first} two upper bounds for the graph shown in Figure~\ref{fig:exampleSOL} and its two different edge-weight vectors. The figure also includes all the values needed to compute the upper bounds, as described in Sections~\ref{sec:UB1} and~\ref{sec:UB2}. These examples illustrate that there exist graphs, edge-weight vectors, and colorings for which~$\OSBound$ is smaller than~$\JAPbound$ (see Parts \ref{subfig:graph1_itaspa} and~\ref{subfig:graph2_itaspa} of Figure~\ref{fig:example_os_jap}), and vice versa (see Parts \ref{subfig:graph1_jap} and~\ref{subfig:graph2_jap} of Figure~\ref{fig:example_os_jap}).

\begin{figure}
\caption{%
The values of the two upper bounds $\OSBound$ and $\JAPbound$ for the instances in Figure~\ref{fig:exampleSOL}, with respect to the following coloring: $\{3,6,9\}$ (yellow), $\{1,4,8\}$ (red), $\{2,5,7\}$ (green). \Rev{The total order is given by $\{3,6,9\} \prec \{1,4,8\} \prec \{2,5,7\}$.} For each edge $e \in E$, the black label next to the edge shows its weight --- in Parts~\ref{subfig:graph1_itaspa} and~\ref{subfig:graph1_jap}, this corresponds to the edge-weight vector $\boldsymbol c_a$ of Part~\ref{subfig:graph1} of Figure~\ref{fig:exampleSOL}; in Parts~\ref{subfig:graph2_itaspa} and~\ref{subfig:graph2_jap}, to the edge-weight vector $\boldsymbol c_b$ of Part~\ref{subfig:graph2} of Figure~\ref{fig:exampleSOL}.
In Parts~\ref{subfig:graph1_itaspa} and~\ref{subfig:graph2_itaspa}, for each edge $e \in E$, the blue labels on the edge report the values of the variables $\varrho_{eu}$ and $\varrho_{ev}$ in an optimal solution of Model~\eqref{form:LP-italia-spagna}. For each vertex $u \in V$, the value $\gamma_u$ (see the definition~\eqref{eq:vertex-weight-gamma}) is shown in purple next to the vertex. For each color $I \in \mathscr{C}$, the vertex $u \in I$ with the maximum weight $\gamma_u$ is highlighted in bold; this value corresponds to the value of the variable $\pi_I$ in the same optimal solution of Model~\eqref{form:LP-italia-spagna}.
In Parts~\ref{subfig:graph1_jap} and~\ref{subfig:graph2_jap}, for each vertex $u \in V$, the value $\sigma_u$ (see the definition~\eqref{eq:jap_weight}) is shown in purple next to the vertex. Additionally, the edges contributing to $\sigma_u$ are drawn in the same color as vertex~$u$. For each color $I \in \mathscr{C}$, the vertex $u \in I$ with the maximum weight $\sigma_u$ is highlighted in bold.
\label{fig:example_os_jap}
}

\begin{subfigure}[b]{0.48\textwidth}%
\centering
  \begin{tikzpicture}[scale=1.29]
        \footnotesize

        \pgfmathsetmacro{\a}{1.6}
        \pgfmathsetmacro{\result}{sqrt((\a)^2 - (0.5*\a)^2)}

        \draw (-3/2*\a,\result cm) coordinate (1);
        \draw (-0.5*\a,\result cm) coordinate (2);

        \draw (-\a,0cm) coordinate (3);
        \draw (\a,0cm) coordinate (6);

        \draw (0.5*\a,\result cm) coordinate (4);
        \draw (3/2*\a,\result cm) coordinate (5);

        \draw (-0.5*\a,-\result cm) coordinate (7);
        \draw (0.5*\a,-\result cm) coordinate (8);
        \draw (0,-2*\result cm) coordinate (9);

        \draw[->]  (1) -- (2) node[midway, yshift = 3ex] {$5$};
        \draw[->]  (1) -- (3) node[midway, xshift = -2ex, yshift = -1ex] {$3$};
        \draw[->]  (2) -- (3) node[midway, xshift = 2ex, yshift = -1ex] {$2$};
        \draw[->]  (2) -- (4) node[midway, yshift = 3ex] {$2$};
        \draw[->]  (3) -- (7) node[midway, xshift = -2ex, yshift = -1ex] {$5$};
        \draw[->]  (4) -- (5) node[midway, yshift = 3ex] {$3$};
        \draw[->]  (4) -- (6) node[midway, xshift = -2ex, yshift = -1ex] {$3$};
        \draw[->]  (5) -- (6) node[midway, xshift = 2ex, yshift = -1ex] {$7$};
        \draw[->]  (6) -- (8) node[midway, xshift = 2ex, yshift = -1ex] {$9$};
        \draw[->]  (7) -- (8) node[midway, yshift = 3ex]{ $4$};
        \draw[->]  (7) -- (9) node[midway, xshift = -2ex, yshift = -1ex] {$3$};
        \draw[->]  (8) -- (9) node[midway, xshift = 2ex, yshift = -1ex] {$2$};

         \draw[-] (1) -- (2) node[midway,xshift=-2ex, fill=white, inner sep=1pt, text=mblue] {\scriptsize $4$}
        node[midway,xshift=2ex, fill=white, inner sep=1pt, text=mblue] {\scriptsize $1$};
        \draw[-] (1) -- (3) node[midway,xshift=-1ex,yshift=1.5ex, fill=white, inner sep=1pt, text=mblue] {\scriptsize $2$}
        node[midway,xshift=1ex,yshift=-1.5ex, fill=white, inner sep=1pt, text=mblue] {\scriptsize $1$};
        \draw[-] (2) -- (3) node[midway,xshift=-1ex,yshift=-1.5ex, fill=white, inner sep=1pt, text=mblue] {\scriptsize $0$}
        node[midway,xshift=1ex,yshift=1.5ex, fill=white, inner sep=1pt, text=mblue] {\scriptsize $2$};
        \draw[-] (2) -- (4) node[midway,xshift=-2ex, fill=white, inner sep=1pt, text=mblue] {\scriptsize $2$}
        node[midway,xshift=2ex, fill=white, inner sep=1pt, text=mblue] {\scriptsize $0$};
        \draw[-] (3) -- (7) node[midway,xshift=-1ex,yshift=1.5ex, fill=white, inner sep=1pt, text=mblue] {\scriptsize $4$}
        node[midway,xshift=1ex,yshift=-1.5ex, fill=white, inner sep=1pt, text=mblue] {\scriptsize $1$};
        \draw[-] (4) -- (5) node[midway,xshift=-2ex, fill=white, inner sep=1pt, text=mblue] {\scriptsize $3$}
        node[midway,xshift=2ex, fill=white, inner sep=1pt, text=mblue] {\scriptsize $0$};
        \draw[-] (4) -- (6) node[midway,xshift=1ex,yshift=-1.5ex, fill=white, inner sep=1pt, text=mblue] {\scriptsize $0$}
        node[midway,xshift=-1ex,yshift=1.5ex, fill=white, inner sep=1pt, text=mblue] {\scriptsize $3$};
        \draw[-] (5) -- (6) node[midway,xshift=1ex,yshift=1.5ex, fill=white, inner sep=1pt, text=mblue] {\scriptsize $5$}
        node[midway,xshift=-1ex,yshift=-1.5ex, fill=white, inner sep=1pt, text=mblue] {\scriptsize $2$};
        \draw[-] (6) -- (8) node[midway,xshift=-1ex,yshift=-1.5ex, fill=white, inner sep=1pt, text=mblue] {\scriptsize $6$}
        node[midway,xshift=1ex,yshift=1.5ex, fill=white, inner sep=1pt, text=mblue] {\scriptsize $3$};
        \draw[-] (7) -- (8) node[midway,xshift=-2ex, fill=white, inner sep=1pt, text=mblue] {\scriptsize $4$}
        node[midway,xshift=2ex, fill=white, inner sep=1pt, text=mblue] {\scriptsize $0$};
        \draw[-] (7) -- (9) node[midway,xshift=-1ex,yshift=1.5ex, fill=white, inner sep=1pt, text=mblue] {\scriptsize $0$}
        node[midway,xshift=1ex,yshift=-1.5ex, fill=white, inner sep=1pt, text=mblue] {\scriptsize $3$};
        \draw[-] (8) -- (9) node[midway,xshift=1ex,yshift=1.5ex, fill=white, inner sep=1pt, text=mblue] {\scriptsize $0$}
        node[midway,xshift=-1ex,yshift=-1.5ex, fill=white, inner sep=1pt, text=mblue] {\scriptsize $2$};

        \draw[circle,fill=red!40,minimum size=17pt,inner sep=0pt,line width=1.5pt] (1) circle (9pt) node[label=center: $\boldsymbol1$] {};
        \draw[circle,fill=green!25,minimum size=17pt,inner sep=0pt] (2) circle (9pt) node[label=center: $2$] {};
        \draw[circle,fill=yellow!100!orange!55,minimum size=17pt,inner sep=0pt] (3) circle (9pt) node[label=center: $3$] {};
        \draw[circle,fill=red!25,minimum size=17pt,inner sep=0pt] (4) circle (9pt) node[label=center: $4$] {};
        \draw[circle,fill=green!50,minimum size=17pt,inner sep=0pt,line width=1.5pt] (5) circle (9pt) node[label=center: $\boldsymbol 5$] {};          
        \draw[circle,fill=yellow!100!orange!55,minimum size=17pt,inner sep=0pt] (6) circle (9pt) node[label=center: $6$] {};
        \draw[circle,fill=green!25,minimum size=17pt,inner sep=0pt] (7) circle (9pt) node[label=center: $7$] {};
        \draw[circle,fill=red!25,minimum size=17pt,inner sep=0pt] (8) circle (9pt) node[label=center: $8$] {};
        \draw[circle,fill=yellow!100!orange!75,minimum size=17pt,inner sep=0pt,line width=1.5pt]  (9) circle (9pt) node[label=center: $\boldsymbol9$] {};

        \foreach \pos/\label/\xshift/\yshift in {
        1/6/0em/1.5em,
        2/5/0em/1.5em,
        3/5/-1.5em/0em,
        4/6/0em/1.5em,
        5/5/0/1.5em,
        6/5/1.5em/0em,
        7/5/-1.5em/0em,
        8/6/1.5em/0em,
        9/5/-1.5em/0em
    }{
        \node[text=purple] at ([xshift=\xshift, yshift=\yshift]\pos) {\scriptsize $\label$};
    }
      \end{tikzpicture}
\subcaption{\label{subfig:graph1_itaspa} $\UB_1(G, \boldsymbol c_a, \mathscr{C})=16$}
\end{subfigure}
\hfill
\begin{subfigure}[b]{0.48\textwidth}%
\centering
    \begin{tikzpicture}[scale=1.29]
    \footnotesize

        \pgfmathsetmacro{\a}{1.6}
        \pgfmathsetmacro{\result}{sqrt((\a)^2 - (0.5*\a)^2)}

        \draw (-3/2*\a,\result cm) coordinate (1);
        \draw (-0.5*\a,\result cm) coordinate (2);

        \draw (-\a,0cm) coordinate (3);
        \draw (\a,0cm) coordinate (6);

        \draw (0.5*\a,\result cm) coordinate (4);
        \draw (3/2*\a,\result cm) coordinate (5);

        \draw (-0.5*\a,-\result cm) coordinate (7);
        \draw (0.5*\a,-\result cm) coordinate (8);
        \draw (0,-2*\result cm) coordinate (9);
        
       \draw[-, green, line width =1.5pt]  (1) -- (2) node[midway, yshift = 2ex] {$\color{black}5$};
       \draw[-,red, line width =1.5pt]  (1) -- (3) node[midway, xshift = -2ex, yshift = -1ex] {$\color{black}3$};
        \draw[-, green, line width =1.5pt] (2) -- (3) node[midway, xshift = 2ex, yshift = -1ex] {$\color{black}2$};
        \draw[-]  (2) -- node[label=above: $2$] {} (4);
        \draw[-, green, line width =1.5pt]  (3) -- (7) node[midway, xshift = -2ex, yshift = -1ex] {$\color{black}5$};
        \draw[-, green, line width =1.5pt]  (4) -- node[label=above: $\color{black}3$] {} (5);
        \draw[-,red, line width =1.5pt]  (4) -- (6) node[midway, xshift = -2ex, yshift = -1ex] {$\color{black}3$};
        \ \draw[-, green, line width =1.5pt]  (5) -- (6) node[midway, xshift = 2ex, yshift = -1ex] {$\color{black}7$};
        \draw[-,red, line width =1.5pt]  (6) -- (8) node[midway, xshift = 2ex, yshift = -1ex] {$\color{black}9$};
        \draw[-, green, line width =1.5pt]  (7) --  node[label=above: $\color{black}4$] {} (8);
        \draw[-]  (7) -- (9) node[midway, xshift = -2ex, yshift = -1ex] {$3$};
        \draw[-]  (8) -- (9) node[midway, xshift = 2ex, yshift = -1ex] {$2$};

        \draw[circle,fill=red!25,minimum size=17pt,inner sep=0pt] (1) circle (9pt) node[label=center: $1$] {};
        \draw[circle,fill=green!25,minimum size=17pt,inner sep=0pt] (2) circle (9pt) node[label=center: $2$] {};
        \draw[circle,fill=yellow!100!orange!75,minimum size=17pt,inner sep=0pt,line width=1.5] (3) circle (9pt) node[label=center: $\boldsymbol 3$] {};
        \draw[circle,fill=red!25,minimum size=17pt,inner sep=0pt] (4) circle (9pt) node[label=center: $4$] {};
        \draw[circle,fill=green!50,minimum size=17pt,inner sep=0pt,line width=1.5] (5) circle (9pt) node[label=center: $\boldsymbol 5$] {};          
        \draw[circle,fill=yellow!100!orange!55,minimum size=17pt,inner sep=0pt] (6) circle (9pt) node[label=center: $6$] {};
        \draw[circle,fill=green!25,minimum size=17pt,inner sep=0pt] (7) circle (9pt) node[label=center: $7$] {};
        \draw[circle,fill=red!40,minimum size=17pt,inner sep=0pt,line width=1.5] (8) circle (9pt) node[label=center: $\boldsymbol 8$] {};
        \draw[circle,fill=yellow!100!orange!55,minimum size=17pt,inner sep=0pt] (9) circle (9pt) node[label=center: $9$] {};

        \foreach \pos/\label/\xshift/\yshift in {
        1/3/0em/1.5em,
        2/7/0em/1.5em,
        3/0/-1.5em/0em,
        4/3/0em/1.5em,
        5/10/0/1.5em,
        6/0/1.5em/0em,
        7/9/-1.5em/0em,
        8/9/1.5em/0em,
        9/0/-1.5em/0em
        }{
        \node[text=purple] at ([xshift=\xshift, yshift=\yshift]\pos) {\scriptsize $\label$};
        }

      \end{tikzpicture}
\subcaption{\label{subfig:graph1_jap} $\UB_2(G, \boldsymbol c_a, \mathscr{C}, \Rev{\prec})=19$}
\end{subfigure}\medskip

\begin{subfigure}[b]{0.48\textwidth}
\centering
     \begin{tikzpicture}[scale=1.29]
        \footnotesize

        \pgfmathsetmacro{\a}{1.6}
        \pgfmathsetmacro{\result}{sqrt((\a)^2 - (0.5*\a)^2)}

        \draw (-3/2*\a,\result cm) coordinate (1);
        \draw (-0.5*\a,\result cm) coordinate (2);

        \draw (-\a,0cm) coordinate (3);
        \draw (\a,0cm) coordinate (6);

        \draw (0.5*\a,\result cm) coordinate (4);
        \draw (3/2*\a,\result cm) coordinate (5);

        \draw (-0.5*\a,-\result cm) coordinate (7);
        \draw (0.5*\a,-\result cm) coordinate (8);
        \draw (0,-2*\result cm) coordinate (9);

        \draw[->]  (1) -- (2) node[midway, yshift = 3ex] {$9$};
        \draw[->]  (1) -- (3) node[midway, xshift = -2ex, yshift = -1ex] {$9$};
        \draw[->]  (2) -- (3) node[midway, xshift = 2ex, yshift = -1ex] {$1$};
        \draw[->]  (2) -- (4) node[midway, yshift = 3ex] {$4$};
        \draw[->]  (3) -- (7) node[midway, xshift = -2ex, yshift = -1ex] {$1$};
        \draw[->]  (4) -- (5) node[midway, yshift = 3ex] {$1$};
        \draw[->]  (4) -- (6) node[midway, xshift = -2ex, yshift = -1ex] {$7$};
        \draw[->]  (5) -- (6) node[midway, xshift = 2ex, yshift = -1ex] {$9$};
        \draw[->]  (6) -- (8) node[midway, xshift = 2ex, yshift = -1ex] {$9$};
        \draw[->]  (7) -- (8) node[midway, yshift = 3ex]{$3$};
        \draw[->]  (7) -- (9) node[midway, xshift = -2ex, yshift = -1ex] {$6$};
        \draw[->]  (8) -- (9) node[midway, xshift = 2ex, yshift = -1ex] {$7$};
        
         \draw[-] (1) -- (2) node[midway,xshift=-2ex, fill=white, inner sep=1pt, text=mblue] {\scriptsize $8$}
        node[midway,xshift=2ex, fill=white, inner sep=1pt, text=mblue] {\scriptsize $1$};
        \draw[-] (1) -- (3) node[midway,xshift=-1ex,yshift=1.5ex, fill=white, inner sep=1pt, text=mblue] {\scriptsize $0$}
        node[midway,xshift=1ex,yshift=-1.5ex, fill=white, inner sep=1pt, text=mblue] {\scriptsize $9$};
        \draw[-] (2) -- (3) node[midway,xshift=1ex,yshift=1.5ex, fill=white, inner sep=1pt, text=mblue] {\scriptsize $0$}
        node[midway,xshift=-1ex,yshift=-1.5ex, fill=white, inner sep=1pt, text=mblue] {\scriptsize $1$};
        \draw[-] (2) -- (4) node[midway,xshift=-2ex, fill=white, inner sep=1pt, text=mblue] {\scriptsize $3$}
        node[midway,xshift=2ex, fill=white, inner sep=1pt, text=mblue] {\scriptsize $1$};
        \draw[-] (3) -- (7) node[midway,xshift=-1ex,yshift=1.5ex, fill=white, inner sep=1pt, text=mblue] {\scriptsize $0$}
        node[midway,xshift=1ex,yshift=-1.5ex, fill=white, inner sep=1pt, text=mblue] {\scriptsize $1$};
        \draw[-] (4) -- (5) node[midway,xshift=-2ex, fill=white, inner sep=1pt, text=mblue] {\scriptsize $0$}
        node[midway,xshift=2ex, fill=white, inner sep=1pt, text=mblue] {\scriptsize $1$};
        \draw[-] (4) -- (6) node[midway,xshift=-1ex,yshift=1.5ex, fill=white, inner sep=1pt, text=mblue] {\scriptsize $7$}
        node[midway,xshift=1ex,yshift=-1.5ex, fill=white, inner sep=1pt, text=mblue] {\scriptsize $0$};
        \draw[-] (5) -- (6) node[midway,xshift=1ex,yshift=1.5ex, fill=white, inner sep=1pt, text=mblue] {\scriptsize $3$}
        node[midway,xshift=-1ex,yshift=-1.5ex, fill=white, inner sep=1pt, text=mblue] {\scriptsize $6$};
        \draw[-] (6) -- (8) node[midway,xshift=1ex,yshift=1.5ex, fill=white, inner sep=1pt, text=mblue] {\scriptsize $4$}
        node[midway,xshift=-1ex,yshift=-1.5ex, fill=white, inner sep=1pt, text=mblue] {\scriptsize $5$};
        \draw[-] (7) -- (8) node[midway,xshift=-2ex, fill=white, inner sep=1pt, text=mblue] {\scriptsize $3$}
        node[midway,xshift=2ex, fill=white, inner sep=1pt, text=mblue] {\scriptsize $0$};
        \draw[-] (7) -- (9) node[midway,xshift=-1ex,yshift=1.5ex, fill=white, inner sep=1pt, text=mblue] {\scriptsize $0$}
        node[midway,xshift=1ex,yshift=-1.5ex, fill=white, inner sep=1pt, text=mblue] {\scriptsize $6$};
        \draw[-] (8) -- (9) node[midway,xshift=1ex,yshift=1.5ex, fill=white, inner sep=1pt, text=mblue] {\scriptsize $3$}
        node[midway,xshift=-1ex,yshift=-1.5ex, fill=white, inner sep=1pt, text=mblue] {\scriptsize $4$};

        \draw[circle,fill=red!40,minimum size=17pt,inner sep=0pt,line width=1.5pt] (1) circle (9pt) node[label=center: $\boldsymbol 1$] {};
        \draw[circle,fill=green!25,minimum size=17pt,inner sep=0pt] (2) circle (9pt) node[label=center: $2$] {};
        \draw[circle,fill=yellow!100!orange!55,minimum size=17pt,inner sep=0pt] (3) circle (9pt) node[label=center: $3$] {};
        \draw[circle,fill=red!25,minimum size=17pt,inner sep=0pt] (4) circle (9pt) node[label=center: $4$] {};
        \draw[circle,fill=green!50,minimum size=17pt,inner sep=0pt,line width=1.5pt] (5) circle (9pt) node[label=center: $\boldsymbol 5$] {};          
        \draw[circle,fill=yellow!100!orange!55,minimum size=17pt,inner sep=0pt] (6) circle (9pt) node[label=center: $6$] {};
        \draw[circle,fill=green!25,minimum size=17pt,inner sep=0pt] (7) circle (9pt) node[label=center: $7$] {};
        \draw[circle,fill=red!25,minimum size=17pt,inner sep=0pt] (8) circle (9pt) node[label=center: $8$] {};
        \draw[circle,fill=yellow!100!orange!75,minimum size=17pt,inner sep=0pt,line width=1.5pt] (9) circle (9pt) node[label=center: $\boldsymbol9$] {};

        \foreach \pos/\label/\xshift/\yshift in {
        1/8/0em/1.5em,
        2/4/0em/1.5em,
        3/10/-1.5em/0em,
        4/8/0em/1.5em,
        5/4/0/1.5em,
        6/10/1.5em/0em,
        7/4/-1.5em/0em,
        8/8/1.5em/0em,
        9/10/-1.5em/0em
    }{
        \node[text=purple] at ([xshift=\xshift, yshift=\yshift]\pos) {\scriptsize $\label$};
    }
    
      \end{tikzpicture}
\subcaption{\label{subfig:graph2_itaspa} $\UB_1(G, \boldsymbol c_b, \mathscr{C})=22$}
\end{subfigure}
\hfill
\begin{subfigure}[b]{0.48\textwidth}
\centering
    \begin{tikzpicture}[scale=1.29]
        \footnotesize

        \pgfmathsetmacro{\a}{1.6}
        \pgfmathsetmacro{\result}{sqrt((\a)^2 - (0.5*\a)^2)}

        \draw (-3/2*\a,\result cm) coordinate (1);
        \draw (-0.5*\a,\result cm) coordinate (2);

        \draw (-\a,0cm) coordinate (3);
        \draw (\a,0cm) coordinate (6);

        \draw (0.5*\a,\result cm) coordinate (4);
        \draw (3/2*\a,\result cm) coordinate (5);

        \draw (-0.5*\a,-\result cm) coordinate (7);
        \draw (0.5*\a,-\result cm) coordinate (8);
        \draw (0,-2*\result cm) coordinate (9);

       \draw[-,green, line width =1.5pt]  (1) -- node[label=above: $\color{black}9$] {} (2);
        \draw[-,red, line width =1.5pt]  (1) -- (3) node[midway, xshift = -2ex, yshift = -1ex] {$\color{black}9$};
        \draw[-,green, line width =1.5pt]  (2) -- (3) node[midway, xshift = 2ex, yshift = -1ex] {$\color{black}1$};
        \draw[-]  (2) -- node[label=above: $\color{black}4$] {} (4);
        \draw[-]  (3) -- (7) node[midway, xshift = -2ex, yshift = -1ex] {$\color{black}1$};
        \draw[-,green, line width =1.5pt]  (4) -- node[label=above: $\color{black}1$] {} (5);
        \draw[-,red, line width =1.5pt]  (4) -- (6) node[midway, xshift = -2ex, yshift = -1ex] {$\color{black}7$};
        \draw[-,green, line width =1.5pt]  (5) -- (6) node[midway, xshift = 2ex, yshift = -1ex] {$\color{black}9$};
        \draw[-,red, line width =1.5pt]  (6) -- (8) node[midway, xshift = 2ex, yshift = -1ex] {$\color{black}9$};
        \draw[-,green, line width =1.5pt]  (7) --  node[label=above: $\color{black}3$] {} (8);
        \draw[-,green, line width =1.5pt]  (7) -- (9) node[midway, xshift = -2ex, yshift = -1ex] {$\color{black}6$};
        \draw[-]  (8) -- (9) node[midway, xshift = 2ex, yshift = -1ex] {$\color{black}7$};

        \draw[circle,fill=red!40,minimum size=17pt,inner sep=0pt,line width=1.5] (1) circle (9pt) node[label=center: $\boldsymbol 1$] {};
        \draw[circle,fill=green!25,minimum size=17pt,inner sep=0pt] (2) circle (9pt) node[label=center: $2$] {};
        \draw[circle,fill=yellow!100!orange!75,minimum size=17pt,inner sep=0pt,line width=1.5] (3) circle (9pt) node[label=center: $\boldsymbol 3$] {};
        \draw[circle,fill=red!25,minimum size=17pt,inner sep=0pt] (4) circle (9pt) node[label=center: $4$] {};
        \draw[circle,fill=green!50,minimum size=17pt,inner sep=0pt,line width=1.5] (5) circle (9pt) node[label=center: $\boldsymbol 5$] {};          
        \draw[circle,fill=yellow!100!orange!55,minimum size=17pt,inner sep=0pt] (6) circle (9pt) node[label=center: $6$] {};
        \draw[circle,fill=green!25,minimum size=17pt,inner sep=0pt] (7) circle (9pt) node[label=center: $7$] {};
        \draw[circle,fill=red!25,minimum size=17pt,inner sep=0pt] (8) circle (9pt) node[label=center: $8$] {};
        \draw[circle,fill=yellow!100!orange!55,minimum size=17pt,inner sep=0pt] (9) circle (9pt) node[label=center: $9$] {};

            \foreach \pos/\label/\xshift/\yshift in {
        1/9/0em/1.5em,
        2/10/0em/1.5em,
        3/0/-1.5em/0em,
        4/7/0em/1.5em,
        5/10/0/1.5em,
        6/0/1.5em/0em,
        7/9/-1.5em/0em,
        8/9/1.5em/0em,
        9/0/-1.5em/0em
    }{
        \node[text=purple] at ([xshift=\xshift, yshift=\yshift]\pos) {\scriptsize $\label$};
    }

    \end{tikzpicture}
    \subcaption{\label{subfig:graph2_jap}$\UB_2(G, \boldsymbol c_b, \mathscr{C}, \Rev{\prec})=19$}
    \end{subfigure}
\end{figure}

\section{Theoretical comparison of the \Rev{first} two upper bounds with $\optimum$}
\label{sec:WCA}

In this section, we provide a sequence of EWMCP instances in which the ratio between either of the two upper bounds \Rev{$\OSBound$ and $\JAPbound$} and the edge-weighted clique number becomes arbitrarily large as the number of vertices in the graphs tends to infinity.

For every integer \( n \ge 2 \), let \( G^1_n = (V^1_n, E^1_n) \) be a graph with a vertex set containing \( 2n \) vertices, partitioned into two cliques \( C_1, C_2 \subseteq V^1_n \), where
\[
C_1 = \{1, 2, \dots, n\}, \qquad C_2 = \{n+1, n+2, \dots, 2n\}.
\]
The edge set of the graph includes all edges within the two cliques, each assigned an edge-weight of~1. Moreover, for each vertex \( u \in C_1 \), there is an edge \( \{u, n+u\} \) with edge-weight
\[
\bar{c} = \frac{n(n-1)}{2} - 1,
\]
and, by construction, the vertex \( n+u \) belongs to \( C_2 \). These edge-weights define the edge-weight vector \( \boldsymbol{c}^1_n \) for the graph \( G^1_n \).
An illustration of \( G^1_n \) is provided in Figure~\ref{fig:exampleBAD}.

\begin{figure}
    \centering

     \caption{ 
  Representation of the graph \( G^1_n \), where the vertices are colored using the minimum number of colors. Each vertex is labeled with a natural number displayed inside the node. The two ellipses represent the cliques \( C_1 \) and \( C_2 \). The edge-weight \( \bar{c} \) of each edge connecting a vertex in \( C_1 \) to a vertex in \( C_2 \) is shown next to the edge. All edges within each clique have weight one; these values are omitted from the figure.
  \label{fig:exampleBAD}
}
  \begin{tikzpicture}[scale=1.29]

        \pgfmathsetmacro{\a}{1.6}
        \pgfmathsetmacro{\r}{sqrt((\a)^2 - (0.5*\a)^2)}

        \draw (1,0) coordinate (1);
        \draw (1,-\r) coordinate (2);
        \draw (1,-3*\r) coordinate (3);

        \draw (5,0) coordinate (4);
        \draw (5,-\r) coordinate (5);
        \draw (5,-3*\r) coordinate (6);

        \draw (3, -2*\r) node {$\spacedvdots$};
        \draw (1, -2*\r) node {$\spacedvdots$};
        \draw (5, -2*\r) node {$\spacedvdots$};

        \draw[-]  (1) -- (4) node[above, midway] {$\bar{\pe}$};
        \draw[-]  (2) -- (5) node[above, midway] {$\bar{\pe}$};
        \draw[-]  (3) -- (6) node[above, midway] {$\bar{\pe}$};

        \draw[-]  (1) -- (2) ;
        \draw[-] (2) to[out=160, in=160] (3);
        \draw[-] (1) to[out=180, in=180] (3);

        \draw[-]  (4) -- (5) ;
        \draw[-] (5) to[out=20, in=20] (6);
        \draw[-] (4) to[out=0, in=0] (6);

        \node[draw, ellipse, minimum width=4.4cm, inner sep=0pt,  fit=(1) (2) (3), label=above: $C_1$] {};
        \node[draw, ellipse, minimum width=4.4cm, inner sep=0pt, fit=(4) (5) (6), label=above: $C_2$] {};
        
        \draw[circle,fill=red!25,minimum size=17pt,inner sep=0pt] (1) circle (9pt) node[label=center: \scriptsize{$1$}] {};
        \draw[circle,fill=green!25,minimum size=17pt,inner sep=0pt] (2) circle (9pt) node[label=center: \scriptsize{$2$}] {};
        \draw[circle,fill=yellow!100!orange!55,minimum size=17pt,inner sep=0pt] (3) circle (9pt) node[label=center: \scriptsize{$n$}] {};
        \draw[circle,fill=yellow!100!orange!55,minimum size=17pt,inner sep=0pt] (4) circle (9pt) node[label=center: \scriptsize{$n+1$}] {};
        \draw[circle,fill=red!25,minimum size=17pt,inner sep=0pt] (5) circle (9pt) node[label=center: \scriptsize{$n+2$}] {};       
        \draw[circle,fill=green!25,minimum size=17pt,inner sep=0pt] (6) circle (9pt) node[label=center: \scriptsize{$2n$}] {};
    
      \end{tikzpicture}

\end{figure}

The inclusion-wise maximal cliques of \( G^1_n \) are \( C_1 \), \( C_2 \), with total edge-weight \( \frac{n(n-1)}{2} \), and the \( n \) cliques of cardinality 2 given by the edges connecting \( C_1 \) and \( C_2 \), with total edge-weight equal to \( \bar{c}\). Hence, the edge-weighted clique number of the EWMCP instance \( (G^1_n, \boldsymbol{c}^1_n) \) is
\begin{equation}    
 \optunon \, = \, \frac{n(n-1)}{2}.
 \label{omegaN}
\end{equation}

Since \( G^1_n \) contains a clique of cardinality \( n \), any coloring of its vertices requires at least \( n \) colors. It follows that for every coloring \( \mathscr{C} \) of the vertices of \( G^1_n \), we have \( |\mathscr{C}| \ge n \). Moreover, any independent set \( I \) in any coloring \( \mathscr{C} \) contains at most one vertex from \( C_1 \) and at most one from \( C_2 \), and therefore \( |I| \le 2 \).

\subsection{\Rev{The first upper bound versus $\optimum$}}\label{sec:ub1_vs_optimum}

The following proposition states that the ratio between the upper bound $\OSBound$ and edge-weighted clique number $\optimum$ can be arbitrarily large.

\begin{proposition}\label{prop:example_bad_itaspa}
The sequence $\{(G^1_n,\boldsymbol c^1_n)\}$ of EWMCP instances is such that
\begin{equation*}
\frac{\OSBunon}{\optunon} \rightarrow + \infty \quad   ~~\text{as}~~ \quad n \rightarrow +\infty,
\end{equation*}
where $\mathscr{C}^1_n$ is any coloring of the vertices of $G^1_n$.
\end{proposition}

\begin{proof}\label{example:poor_ub_itaspa}

Given any independent set $I\in\mathscr C^1_n$, by adding up the Constraints \eqref{con:LP-italia-spagna_max-w} related to $I$, we obtain:
\begin{gather*}\label{proof_step:sum_indep_1}
|I| \, \pi_I \, \ge \, \sum_{u \in I}\sum_{e \in \delta(u)} \ \ \varrho_{eu}\quad
\Longrightarrow \quad 
2 \, \pi_I \, \ge \, \sum_{u \in I}\sum_{e \in \delta(u)} \ \ \varrho_{eu},
\end{gather*}
since $|I|\le 2$.
Moreover, adding up all these last inequalities for every independent set $I\in\mathscr C^1_n$, we then have:
\begin{equation*}
    2 \, \sum_{I \in \mathscr C^1_n} \pi_I \geq \sum_{e \in E} \ (\varrho_{eu} + \varrho_{ev}),
\end{equation*}
since any coloring is a partition of the vertex set.
Then, according to Constraints \eqref{con:LP-italia-spagna_splitting}, the following holds:
\begin{align*}
  \underbrace{\sum_{I \in \mathscr C^1_n} \pi_I}_{=\OSBunon} \,\geq \frac 12\sum_{e \in E^1_n} c_{e} = \frac12\sum_{e \in E^1_n[C_1]} c_{e} +\frac 12 \sum_{e \in E^1_n[C_2]} c_{e} + \frac 12 n\, \bar{c}
=  \underbrace{\frac{n(n-1)}{2}}_{=\optunon} + \frac 12 n \, \bar{c}. 
\end{align*}
Finally, it follows that:
\begin{equation*}
    \lim_{n \to +\infty} \frac{\OSBunon}{\optunon} \,  \geq \, \lim_{n \to +\infty} 1 + \frac{n}{2} - \frac{1}{n - 1}= +\infty.
\end{equation*}
\end{proof}

\subsection{\Rev{The second upper bound versus $\optimum$}}\label{sec:ub2_vs_optimum}

The following proposition states that the ratio between the upper bound $\JAPbound$ and edge-weighted clique number $\optimum$ can be arbitrarily large.

\begin{proposition}\label{prop:example_bad_japan}
The sequence $\{(G^1_n,\boldsymbol c^1_n)\}$ of EWMCP instances is such that
\begin{equation*}
\frac{\JAPunon}{\optunon} \rightarrow + \infty \quad   ~~\text{as}~~ \quad n \rightarrow +\infty,
\end{equation*}
where $\mathscr{C}^1_n$ is any coloring of the vertices of $G^1_n$ \Rev{and $\prec^1_n$ defines any total order of the independent sets in $\mathscr{C}^1_n$.}
\end{proposition}

\begin{proof}

For any edge $\{w,n+w\}\in E^1_n$ with $w\in\{1,2,\dots,n\}$, connecting a vertex in $C_1$ with a vertex in $C_2$, its endpoints belong to distinct independent sets of $\mathscr C^1_n$.
Let $V'\subseteq V^1_n$ be the subset of vertices $w\in C_1$ such that \Rev{$I(w) \prec^1_n I(n+w)$ and vertices $n+w\in C_2$ such that $I(n+ w)\prec^1_n I(w)$}. Let $V'' := V^1_n\setminus V'$. This implies that each vertex $v$ in $V'$ is connected to a vertex $u$ in $V''$ with an edge of weight $\bar c$ and we have $I(v) \prec^1_n I(u)$.
Then, we have:
$$
\sum_{v\in V^1_n}\sigma_v=\sum_{v\in V'}\sigma_v+\sum_{u\in V''}\sigma_u=
\sum_{v\in V'}\sigma_v+\sum_{u\in V''}\sum_{\Rev{J \in \mathscr{C}^1_n : J \prec^1_n I(u)}}\max\{c_e\colon e\in\delta(u),v\in J\}.
$$
Since, for every $u\in V''$ there exists an edge of weight $\bar c$ belonging to the set $\{e\in\delta(u)\colon v\in J, J \prec^1_n I(u)\}$, and since $|V''|=n$, we have:
$$
\sum_{v\in V^1_n}\sigma_v \; \ge\; n\,\bar c.
$$
For every $I \in \mathscr{C}^1_n$ we have:
\[
\max\{\sigma_u\colon u\in I\}
\;\ge\;
\frac12\sum_{u\in I}\sigma_u,
\]
since the right-hand side is at most as large as the average of the values in $\{\sigma_u\colon u\in I\}$. The latter is true because $|I|\le 2$ for every $I \in \mathscr{C}^1_n$.
Adding up these inequalities for all \(I \in \mathscr{C}^1_n\) yields
\[
\;\underbrace{\sum_{I\in\mathscr C^1_n}\max\{\sigma_u\colon u\in I\}}_{=\JAPunon}
\;\ge\;
\frac12\sum_{v\in V^1_n}\sigma_v
\;\ge\;
\frac{n\,\bar c}{2}=\dfrac{n}{2}\left(\underbrace{\frac{n(n-1)}{2}}_{=\optunon}-1\right).
\]

Finally, it follows that
$$
\lim_{n\to +\infty}\frac{\JAPunon}{\optunon}\ge\lim_{n\to +\infty} \frac{n}{2}-\frac{1}{n-1}=+\infty.
$$

\end{proof}

\section{Theoretical comparison between the \Rev{first} two upper bounds}\label{sec:THEOR_COMP}

In this section, we provide two sequences of EWMCP instances in which the ratio between the two upper bounds \Rev{$\OSBound$ and $\JAPbound$}, or its reciprocal, becomes arbitrarily large as the number of vertices in the graphs tends to infinity.

\subsection{\Rev{The second upper bound versus the first}}\label{sec:ub2_vs_ub1}

\bigskip
For every integer \( n \geq 2 \), let \( G^2_n = (V^2_n, E^2_n) \)  be a graph with vertex set containing a set \( C_0  \) of $n$ vertices, and, for each vertex $u\in C_0$, a set $C_u$ of additional $n-1$ vertices, with $C_u\cap C_0=\emptyset$. Moreover, the sets $C_u$ for every $u\in C_0$ are pairwise disjoint. 
The edge set contains only the edges that make $C_0$ and $C_u\cup\{u\}$, for every $u\in C_0$, cliques.
Hence, the graph is composed of \( n \) \emph{peripheral cliques} $C_u\cup\{u\}$ of size \( n \), each sharing exactly one vertex $u$ with the \emph{central clique} $C_0$. For every edge $e\in E$, its edge-weight is defined as follows:
\[
c_{e} = \begin{cases}
1 & \text{if }  u \in C_0\text{ and } \, v \in C_u,\\
0 & \text{otherwise}.
\end{cases}
\]
That is, the edges connecting the central clique to the peripheral cliques have weight $1$, and  all the remaining edges (within each clique) have weight $0$.
These edge-weights define the edge-weight vector $\boldsymbol c^2_n$ for graph $G^2_n$. 
An illustration of a graph of this family, with \( n = 3\), is shown in Figure \ref{fig:itaspa_wins}.

\begin{figure}
\centering
\caption{
Representation of the graph \( G^2_3 \), where the vertices are colored using the minimum number of colors. Each vertex is labeled with a natural number displayed inside the node. 
\Rev{Edges of weight~$1$ are shown in black with their weight displayed next to them; edges of weight~$0$ are shown in gray.}
\label{fig:itaspa_wins}}
  \begin{tikzpicture}[scale=1.29]
        \footnotesize

        \pgfmathsetmacro{\a}{1.6}
        \pgfmathsetmacro{\result}{sqrt((\a)^2 - (0.5*\a)^2)}

        \draw (0, 0 cm) coordinate (0);
        \draw (-0.5*\a,0.5*\result cm) coordinate (1);
        \draw (0.5*\a,0.5*\result cm) coordinate (2);
        \draw (0,-0.5*\result cm) coordinate (3);

        \draw  (-0.5*\a,-1.5*\result cm) coordinate (8);
        \draw  (0.5*\a,-1.5*\result cm) coordinate (9);

        \draw (-1.5*\a,0.5*\result cm) coordinate (4);
        \draw (-1*\a,1.5*\result cm) coordinate (5);

        \draw (1.5*\a,0.5*\result cm) coordinate (6);
        \draw (1*\a,1.5*\result cm) coordinate (7);

        \draw[-, gray!80!white]  (1) -- (2) node[midway, xshift =0, yshift = 1.5ex] {};
        \draw[-, gray!80!white]  (1) -- (3) node[midway, xshift = -2ex, yshift = -1ex] {};
        \draw[-, gray!80!white]  (2) -- (3) node[midway, xshift = 2ex, yshift = -1ex] {};
        \draw[-, very thick]  (1) -- (4) node[midway, yshift = -2ex] {$1$};
        \draw[-, very thick]  (1) -- (5) node[midway, xshift = 2ex, yshift = 1ex] {$1$};
        \draw[-, gray!80!white]  (4) -- (5) node[midway, xshift = -1ex, yshift = 2ex] {};
        \draw[-, very thick]  (2) -- (6) node[midway, yshift = -2ex] {$1$};
        \draw[-, very thick]  (2) -- (7) node[midway, xshift = -2ex, yshift = 1ex] {$1$};
        \draw[-, gray!80!white]  (6) -- (7) node[midway, xshift = 2ex, yshift = 1ex] {};
        \draw[-, very thick]  (3) -- (8) node[midway, xshift = -2ex, yshift = 0.8ex] {$1$};
        \draw[-, very thick]  (3) -- (9) node[midway, xshift = 2ex, yshift = 0.8ex] {$1$};
        \draw[-, gray!80!white]  (8) -- (9) node[midway, xshift = 0, yshift = -1.5ex] {};

        \draw[circle,fill=red!25, minimum size=17pt,inner sep=0pt] (1) circle (9pt) node[label=center: $1$] {};
        \draw[circle,fill=green!25,minimum size=17pt,inner sep=0pt] (2) circle (9pt) node[label=center: $2$] {};
        \draw[circle,fill=yellow!100!orange!55,minimum size=17pt,inner sep=0pt] (3) circle (9pt) node[label=center: $3$] {};
        \draw[circle,fill=green!25,minimum size=17pt,inner sep=0pt] (4) circle (9pt) node[label=center: $4$] {};
        \draw[circle,fill=yellow!100!orange!55,minimum size=17pt,inner sep=0pt] (5) circle (9pt) node[label=center: $5$] {};          
        \draw[circle,fill=yellow!100!orange!55,minimum size=17pt,inner sep=0pt] (6) circle (9pt) node[label=center: $6$] {};
        \draw[circle,fill=red!25,minimum size=17pt,inner sep=0pt] (7) circle (9pt) node[label=center: $7$] {};
        \draw[circle,fill=red!25,minimum size=17pt,inner sep=0pt] (8) circle (9pt) node[label=center: $8$] {};
        \draw[circle,fill=green!25,minimum size=17pt,inner sep=0pt] (9) circle (9pt) node[label=center: $9$] {};

      \end{tikzpicture}
\end{figure}

A coloring $\mathscr C^2_n$ of the vertices of $G^2_n$ with the minimum number $n$ of colors can be constructed as follows. For all $u \in C_0$, an independent set $I_u$ is created containing the vertex $u$ and one vertex for every set $C_v$ with $v \ne u$, which has not been included in an independent set yet. By construction, every edge of weight $1$ joins vertices in different independent sets in $\mathscr C^2_n$. Additionally, each $I_u$ is an independent set of cardinality $n$ and $\mathscr C^2_n := \{I_u\colon u\in C_0\}$ is indeed a coloring of $G^2_n$.

The following proposition states that the ratio between the upper bound $\JAPbound$ and $\OSBound$  can be arbitrarily large.

\begin{proposition} \label{prop:example_itaspa_wins}
The sequence $\{(G^2_n,\boldsymbol c^2_n)\}$ of EWMCP instances is such that 
\begin{equation*}
\frac{\JAPduen}{\OSBduen} \rightarrow + \infty \quad   ~~\text{as}~~ \quad n \rightarrow +\infty,
\end{equation*}
where $\mathscr C^2_n$ is the coloring of $G^2_n$ defined above \Rev{and $\prec^2_n$ defines any total order of the independent sets in $\mathscr{C}^2_n$.}
\end{proposition}

\begin{proof}

We now define a feasible solution to Model \eqref{form:LP-italia-spagna} for the instance $(G^2_n,\boldsymbol c^2_n)$, leading to an upper bound on $\OSBduen$. We assign the weight of every edge of weight $1$ entirely on its endpoint in the peripheral clique, namely

\[
  \tilde\varrho_{eu}=0,\qquad
  \tilde\varrho_{ev}=1,
  \qquad u\in C_0,\; v\in C_u.
\]

We have that each vertex in the peripheral cliques receives a total weight equal to $1$, while each vertex in the central clique receives a total weight equal to $0$. Hence, the feasible solution is completed by setting
$$\tilde \pi_{I}=1, \quad I \in \mathscr C^2_n.$$

Therefore

\[
  \OSBduen \; \le \; \sum_{I \in \mathscr C^2_n} \tilde \pi_{I} \; = \; n.
\]

Let us now compute the value of $\JAPduen$.
We know that, for each vertex in the central clique $u \in \{1,2,\dots,n\}$, belonging to the independent set $I_u \in \mathscr C^2_n$, there is exactly one neighbor of weight 1 in each of the independent sets $I_v$ with \Rev{$I_v \prec^2_n I_u$}. Hence, for every $u \in \{1,2,\dots, n\}$, we have:
\Rev{
$$
\sigma_u = |\{I_v\in\mathscr C^2_n\colon I_v\prec^2_n I_u\}|.
$$
}
On the other hand, each vertex in the peripheral cliques $v \notin C_0$ has at most one incident edge of weight $1$, hence $\sigma_{v}\le 1$.

Therefore, for all independent sets $I_u$ with $u \in \{1,2,\dots,n\}$, the maximum $\max\{\sigma_v\colon v\in I_u\}$ is attained by the central vertex $u$. Then, for every $u \in \{1,2,\dots,n\}$, we have:\Rev{
\[
  \max\{\sigma_v\colon v\in I_u\} = |\{I_v\in\mathscr C^2_n\colon I_v\prec^2_n I_u\}|,
\]}
and thus, \Rev{since $\prec^2_n$ is a total order}
\[
  \JAPduen
   \Rev{\;=\;
    \sum_{u=1}^{n}|\{I_v\in\mathscr C^2_n\colon I_v\prec^2_n I_u\}|
    \;=\;
    \sum_{h=0}^{n-1}h}
    \;=\;
    \frac{n(n-1)}{2}.
\]

Finally, it follows that
$$
\lim_{n\to +\infty}\frac{\JAPduen}{\OSBduen}\ge \lim_{n\to+\infty}\frac{\frac{n(n-1)}2}{n}=\lim_{n\to+\infty}\frac{n-1}2=+\infty.
$$
\end{proof}

\subsection{\Rev{The first upper bound versus the second}}\label{sec:ub1_vs_ub2}

\bigskip
For every integer $n\ge 1$, let $G^3_n=(V^3_n,E^3_n)$ be a graph where the vertex set \( V^3_n \) has cardinality \( 2n \) and  it is partitioned into two independent sets \( I = \{1,2,\dots, n\} \) and \( J = \{n+1,n+2, \dots, 2n\} \), each of cardinality \( n \). By construction, no edges exist between vertices within the same independent set, while each vertex in \( I \) is connected to every vertex in \( J \), i.e., for each $u \in I, \, v \in J$, there exists an edge $e \in E$ connecting $u$ and $v$. To each edge $e\in E^3_n$ we assign the edge-weight \(c_{e} = 1 \). These edge-weights define the edge-weight vector $\boldsymbol c^3_n$ for the graph $G^3_n$. An illustration of the graph $G^3_n$ is provided in Figure \ref{fig:jap_wins}.

\begin{figure}
\centering
\caption{
Representation of the graph \( G^3_n \), where the vertices are colored using the minimum number of colors. Each vertex is labeled with a natural number displayed inside the node. All edge-weights are equal to~\(1\); these values are omitted from the figure.
\label{fig:jap_wins}}
    \begin{tikzpicture}[scale = 1.29]
                \footnotesize

                \pgfmathsetmacro{\n}{4}
                \pgfmathsetmacro{\a}{1.6}
                \pgfmathsetmacro{\result}{sqrt((\a)^2 - (0.5*\a)^2)}

                \foreach \i in {1,...,\n} {
                        \pgfmathsetmacro{\posY}{((\n-1)/2 - (\i - 1))*\result}
                        \draw (-\result cm, {\posY}) coordinate (\i);
                }

                \foreach \i in {1,...,\n} {
                        \pgfmathsetmacro{\posY}{((\n-1)/2 - (\i - 1))*\result}
                        \draw (\result cm, {\posY}) coordinate (\n+\i);
                }
        
        \foreach \i in {1,...,\n} {%
                \ifnum\i=\numexpr\n-1\relax
                \else
                     \foreach \j in {1,...,\n} {%
                             \ifnum\j=\numexpr\n-1\relax
                             \else
                                     \draw[-] (\i) -- ({\n+\j}) 
                                             node[midway, yshift=2ex] {};
                             \fi
                     }%
                \fi
        }

        \foreach \i in {1,...,\n} {
            \ifnum\i=\numexpr\n-1\relax
                 \node at (\i) {$\spacedvdots$};
            \else
                 \ifnum\i=\numexpr\n\relax
                         \draw[circle,fill=red!25, minimum size=17pt, inner sep=0pt] (\i) circle (9pt)
                                node[label=center: {\scriptsize $n$}] {};
                 \else
                         \draw[circle,fill=red!25, minimum size=17pt, inner sep=0pt] (\i) circle (9pt)
                                node[label=center: {\scriptsize $\i$}] {};
                 \fi
            \fi
        }
        
        \foreach \i in {1,...,\n} {
            \ifnum\i=\numexpr\n-1\relax
                 \node at ({\n+\i}) {$\spacedvdots$};
            \else
                 \ifnum\i=\numexpr\n\relax
                         \draw[circle,fill=green!25, minimum size=17pt, inner sep=0pt] ({\n+\i}) circle (9pt)
                                node[label=center: {\scriptsize $2\,n$}] {};
                 \else
                         \draw[circle,fill=green!25, minimum size=17pt, inner sep=0pt] ({\n+\i}) circle (9pt)
                                node[label=center: {\scriptsize $n+\i$}] {};
                 \fi
            \fi
        }
        
            \end{tikzpicture}
\end{figure}

A coloring of the vertices of $G^3_n$ with the minimum number of colors is naturally given as $\mathscr{C}^3_n = \{I_1, I_2\}$.

The following proposition states that the ratio between the upper bound $\OSBound$  and $\JAPbound$  can be arbitrarily large.

\begin{proposition} \label{prop:example_japan_wins}
The sequence $\{(G^3_n,\boldsymbol c^3_n)\}$ of EWMCP instances is such that 
\begin{equation*}
\frac{\OSBtren}{\JAPtren} \rightarrow + \infty \quad   ~~\text{as}~~ \quad n \rightarrow +\infty,
\end{equation*}
where $\mathscr C^3_n$ is the coloring of $G^3_n$ defined above \Rev{and $\prec^3_n$ defines any total order of the independent sets in $\mathscr{C}^3_n$}.
\end{proposition}

\begin{proof}
We first show that the optimal value of Model~\eqref{form:LP-italia-spagna}, namely $\OSBtren$, is $n$. 
We notice that, in any feasible solution of Model \eqref{form:LP-italia-spagna}, for every $u\in I$ and $v\in J$, Constraints \eqref{con:LP-italia-spagna_max-w} relative to the edge $e=\{u,v\}$ write as 
$$\pi_{I}\ge \sum_{e\in\delta(u)}\varrho_{eu},\qquad \pi_{J}\ge \sum_{e\in\delta(v)}\varrho_{ev}.$$

Adding up these constraints over all vertices, therefore, yields
\[
   n\bigl(\pi_{I}+\pi_{J}\bigr)
       \;\ge\;
       \sum_{e \in E} \ c_{e},
\]
where we used Constraint \eqref{con:LP-italia-spagna_splitting}.

Moreover, all edge-weights are equal to 1, so the sum on the right-hand side equals $|E^3_n|=n^2$.
Consequently, $\pi_{I}+\pi_{J}\ge n$, and every feasible solution of Model~\eqref{form:LP-italia-spagna} has value at least $n$.

This lower bound of $n$ on $\OSBtren$ is tight: if we assign each edge-weight wholly to its endpoint in $J$ and make the variable assignment $\pi_{I}=0,\ \pi_{J}=n$, all constraints of Model~\eqref{form:LP-italia-spagna} are satisfied and the objective function equals $n$.

Hence
\[
  \OSBtren = n.
\]

\medskip
As far as the bound $\JAPtren$ is concerned, we \Rev{assume, without loss of generality that $I \prec^3_n J$:} observe that every vertex $u\in I$ is such that $\sigma_u=0$, since there are no independent sets \Rev{preceding $I$ under the total order $\prec^3_n$}. Whereas every vertex $v\in J$ is such that $\sigma_v=1$, since it is connected only to vertices in $I$ with edges of weight $1$. Hence, we have:
$$\max\{\sigma_u\colon u\in I\}=0, \qquad \max\{\sigma_v\colon v\in J\}=1,
$$
giving the upper bound 
$$\JAPtren=1.$$
Finally, it follows that
$$
\lim_{n\to +\infty}\frac{\OSBtren}{\JAPtren}= \lim_{n\to+\infty}\frac{n}{1}=+\infty.
$$

\end{proof}

Note that the values of the two upper bounds in the examples studied in this section, and thus the conclusion of the two propositions, depend on the choice of a particular coloring of the vertices of the graphs $G^2_n$ and $G^3_n$. Nevertheless, our choices feature a minimum number of colors and the results hold with any coloring with this natural property.

\Rev{\section{A third upper bound on $\optimum$}\label{sec:UB3}}
\Rev{In this section, we introduce and theoretically study a third upper bound on the edge-weighted clique number $\optimum$, originally proposed in \citet{hosseinian2020lagrangian} and used within the branch-and-bound algorithm for the EWMCP detailed in the same article.}

\Rev{This bound requires an integer upper bound $k$ on the cardinality of any clique in $G$, namely on the clique number $\omega(G)$ of the graph, with $1\leq k\leq |V|$. Additionally, for all $u \in V$ and $t\in\mathbb{Z}_{\geq 0}$, we denote by $\delta_t(u)\subseteq\delta(u)$ any set of the $\min\{t,|\delta(u)|\}$ heaviest edges incident to $u$; in particular, $\delta_0(u)=\emptyset$.}

\Rev{For all $u\in V$, we define the quantity
\begin{equation}\label{eq:eta_ub3}
    {\eta}_u := \frac{1}{2} \sum_{e \in \delta_{k-1}(u)} c_e,
\end{equation}
which can be interpreted as an upper bound on the contribution of edges incident to vertex~$u$ 
to the total edge-weight of any clique of cardinality at most~$k$ 
containing~$u$, since such a clique can contain at most $k-1$ edges 
incident to~$u$. The factor $\frac{1}{2}$ accounts for the fact that 
each edge of a clique is incident to two of its vertices. This value, similarly to $\gamma_u$, defined in \eqref{eq:vertex-weight-gamma}, and $\sigma_u$, defined in \eqref{eq:jap_weight}, can be also interpreted as the total weight assigned to vertex $u$.}

\Rev{Finally, let $\varphi(j)$, with $j\in\{1,2,\ldots,|V|\}$, be a function that returns the vertex $u$ with the $j$-th largest $\eta_u$, with ties broken arbitrarily. 
The function $\varphi$ induces the following ordering on the vertex quantities $\eqref{eq:eta_ub3}$:
\[
    {\eta}_{\varphi(1)}\geq {\eta}_{\varphi(2)} \geq \cdots \geq {\eta}_{\varphi(|V|)}.
\]
The third upper bound studied in this article, denoted by~$\HFBbound$, is then defined as:
\begin{equation} \label{eq:ub3}
    \HFBbound := \sum_{j=1}^{k} {\eta}_{\varphi(j)},
\end{equation}

that is the sum of the weights $\eta_u$ of the $k$ heaviest vertices of $V$. 
}

\Rev{
For completeness, we verify that $\HFBbound$ is a valid upper bound on~$\optimum$. Consider any clique $C$ in~$G$; since $k$ 
is an upper bound on the clique number of~$G$, we have 
$|C| \leq k$. Then:
\begin{align}
\sum_{e \in E[C]} c_e 
&= \sum_{u \in C} \frac{1}{2} \sum_{v \in C \setminus \{u\}} c_{\{u,v\}} 
\label{eq:hfb-0st-row}\\[2ex]
&\leq \sum_{u \in C} \frac{1}{2} \sum_{e \in  \delta_{k-1}(u)} 
c_e
\label{eq:hfb-1st-row}\\[2ex]
&= \sum_{u \in C} {\eta}_u
\label{eq:hfb-2nd-row}\\[2ex]
&\leq \sum_{j=1}^{k} {\eta}_{\varphi(j)} 
= \HFBbound. \label{eq:hfb-3rd-row}
\end{align}
The equality in~\eqref{eq:hfb-0st-row} rewrites the total edge-weight of the clique~$C$ by distributing each edge-weight equally between its two endpoints.
The inequality in~\eqref{eq:hfb-1st-row} follows from the fact that $|C \setminus \{u\}| \leq k-1$. Therefore, the total weight of the edges connecting~$u$ to the other vertices of~$C$ is bounded above by the sum of the weights of the $k-1$ heaviest edges incident to~$u$.
The equality in~\eqref{eq:hfb-2nd-row} holds by the 
definition~\eqref{eq:eta_ub3} of ${\eta}_u$.
Finally, the inequality in~\eqref{eq:hfb-3rd-row} follows from $|C| \leq k$, which implies that the sum of~$\eta_u$ over the vertices of~$C$ is bounded above by the sum of the $k$ largest values of~$\eta_u$ over all vertices in~$V$.}

\Rev{In Figure~\ref{fig:example_hfb}, we report the values of the upper bound $\HFBbound$ for the two instances introduced in Figure~\ref{fig:exampleSOL}, with $k = 3$.}

\begin{figure}
\caption{%
\Rev{The values of the upper bound $\HFBbound$ for the instances in Figure~\ref{fig:exampleSOL}, with $k = 3$. For each edge $e \in E$, the black label next to the edge shows its weight~$c_e$ --- in Part~\ref{subfig:graph1_hfb}, this corresponds to the edge-weight vector $\boldsymbol c_a$ of Part~\ref{subfig:graph1} of Figure~\ref{fig:exampleSOL}; in Part~\ref{subfig:graph2_hfb}, to the edge-weight vector $\boldsymbol c_b$ of Part~\ref{subfig:graph2} of Figure~\ref{fig:exampleSOL}. 
For each edge $e \in E$, the value $\frac{c_e}{2}$ is reported in blue on each side of the edge.
A blue label is shown in bold if the corresponding edge is among the $k - 1 = 2$ heaviest edges incident to that vertex, meaning it contributes to the vertex weight~$ \eta_u$. For each vertex $u \in V$, the value $ \eta_u$, computed as the sum of the bold blue labels incident to~$u$, is shown in purple next to the vertex. The $k = 3$ vertices with the largest~$\eta_u$ are highlighted in bold with a striped pattern; their weights are summed to obtain~$\HFBbound$.
\label{fig:example_hfb}}
}
    \centering

\begin{subfigure}[b]{0.495\textwidth}
\centering
  \begin{tikzpicture}[scale=1.29]
        \footnotesize

        \pgfmathsetmacro{\a}{1.6}
        \pgfmathsetmacro{\result}{sqrt((\a)^2 - (0.5*\a)^2)}

        \draw (-3/2*\a,\result cm) coordinate (1);
        \draw (-0.5*\a,\result cm) coordinate (2);
        \draw (-\a,0cm) coordinate (3);
        \draw (\a,0cm) coordinate (6);
        \draw (0.5*\a,\result cm) coordinate (4);
        \draw (3/2*\a,\result cm) coordinate (5);
        \draw (-0.5*\a,-\result cm) coordinate (7);
        \draw (0.5*\a,-\result cm) coordinate (8);
        \draw (0,-2*\result cm) coordinate (9);

        \draw[-]  (1) -- (2) node[midway, yshift = 3ex] {$5$};
        \draw[-]  (1) -- (3) node[midway, xshift = -2ex, yshift = -1ex] {$3$};
        \draw[-]  (2) -- (3) node[midway, xshift = 2ex, yshift = -1ex] {$2$};
        \draw[-]  (2) -- (4) node[midway, yshift = 3ex] {$2$};
        \draw[-]  (3) -- (7) node[midway, xshift = -2ex, yshift = -1ex] {$5$};
        \draw[-]  (4) -- (5) node[midway, yshift = 3ex] {$3$};
        \draw[-]  (4) -- (6) node[midway, xshift = -2ex, yshift = -1ex] {$3$};
        \draw[-]  (5) -- (6) node[midway, xshift = 2ex, yshift = -1ex] {$7$};
        \draw[-]  (6) -- (8) node[midway, xshift = 2ex, yshift = -1ex] {$9$};
        \draw[-]  (7) -- (8) node[midway, yshift = 3ex] {$4$};
        \draw[-]  (7) -- (9) node[midway, xshift = -2ex, yshift = -1ex] {$3$};
        \draw[-]  (8) -- (9) node[midway, xshift = 2ex, yshift = -1ex] {$2$};

        \draw[-] (1) -- (2) node[midway,xshift=-2ex, fill=white, inner sep=1pt, text=mblue] {\scriptsize $\boldsymbol{\frac{5}{2}}$}
        node[midway,xshift=2ex, fill=white, inner sep=1pt, text=mblue] {\scriptsize $\boldsymbol{\frac{5}{2}}$};
        \draw[-] (1) -- (3) node[midway,xshift=-1ex,yshift=1.5ex, fill=white, inner sep=1pt, text=mblue] {\scriptsize $\boldsymbol{\frac{3}{2}}$}
        node[midway,xshift=1ex,yshift=-1.5ex, fill=white, inner sep=1pt, text=mblue] {\scriptsize $\boldsymbol{\frac{3}{2}}$};
        \draw[-] (2) -- (3) node[midway,xshift=1ex,yshift=1.5ex, fill=white, inner sep=1pt, text=mblue] {\scriptsize $\boldsymbol{1}$}
        node[midway,xshift=-1ex,yshift=-1.5ex, fill=white, inner sep=1pt, text=mblue] {\scriptsize $1$};
        \draw[-] (2) -- (4) node[midway,xshift=-2ex, fill=white, inner sep=1pt, text=mblue] {\scriptsize $1$}
        node[midway,xshift=2ex, fill=white, inner sep=1pt, text=mblue] {\scriptsize $1$};
        \draw[-] (3) -- (7) node[midway,xshift=-1ex,yshift=1.5ex, fill=white, inner sep=1pt, text=mblue] {\scriptsize $\boldsymbol{\frac{5}{2}}$}
        node[midway,xshift=1ex,yshift=-1.5ex, fill=white, inner sep=1pt, text=mblue] {\scriptsize $\boldsymbol{\frac{5}{2}}$};
        \draw[-] (4) -- (5) node[midway,xshift=-2ex, fill=white, inner sep=1pt, text=mblue] {\scriptsize $\boldsymbol{\frac{3}{2}}$}
        node[midway,xshift=2ex, fill=white, inner sep=1pt, text=mblue] {\scriptsize $\boldsymbol{\frac{3}{2}}$};
        \draw[-] (4) -- (6) node[midway,xshift=-1ex,yshift=1.5ex, fill=white, inner sep=1pt, text=mblue] {\scriptsize $\boldsymbol{\frac{3}{2}}$}
        node[midway,xshift=1ex,yshift=-1.5ex, fill=white, inner sep=1pt, text=mblue] {\scriptsize $\frac{3}{2}$};
        \draw[-] (5) -- (6) node[midway,xshift=1ex,yshift=1.5ex, fill=white, inner sep=1pt, text=mblue] {\scriptsize $\boldsymbol{\frac{7}{2}}$}
        node[midway,xshift=-1ex,yshift=-1.5ex, fill=white, inner sep=1pt, text=mblue] {\scriptsize $\boldsymbol{\frac{7}{2}}$};
        \draw[-] (6) -- (8) node[midway,xshift=-1ex,yshift=-1.5ex, fill=white, inner sep=1pt, text=mblue] {\scriptsize $\boldsymbol{\frac{9}{2}}$}
        node[midway,xshift=1ex,yshift=1.5ex, fill=white, inner sep=1pt, text=mblue] {\scriptsize $\boldsymbol{\frac{9}{2}}$};
        \draw[-] (7) -- (8) node[midway,xshift=-2ex, fill=white, inner sep=1pt, text=mblue] {\scriptsize $\boldsymbol{2}$}
        node[midway,xshift=2ex, fill=white, inner sep=1pt, text=mblue] {\scriptsize $\boldsymbol{2}$};
        \draw[-] (7) -- (9) node[midway,xshift=-1ex,yshift=1.5ex, fill=white, inner sep=1pt, text=mblue] {\scriptsize $\frac{3}{2}$}
        node[midway,xshift=1ex,yshift=-1.5ex, fill=white, inner sep=1pt, text=mblue] {\scriptsize $\boldsymbol{\frac{3}{2}}$};
        \draw[-] (8) -- (9) node[midway,xshift=1ex,yshift=1.5ex, fill=white, inner sep=1pt, text=mblue] {\scriptsize $1$}
        node[midway,xshift=-1ex,yshift=-1.5ex, fill=white, inner sep=1pt, text=mblue] {\scriptsize $\boldsymbol{1}$};

        \draw[circle,fill=white,minimum size=17pt,inner sep=0pt] (1) circle (9pt) node[label=center: $1$] {};
        \draw[circle,fill=white,minimum size=17pt,inner sep=0pt] (2) circle (9pt) node[label=center: $2$] {};
        \draw[circle,fill=white,minimum size=17pt,inner sep=0pt] (3) circle (9pt) node[label=center: $3$] {};
        \draw[circle,fill=white,minimum size=17pt,inner sep=0pt] (4) circle (9pt) node[label=center: $4$] {};
        \draw[circle,preaction={fill=white},pattern=north east lines,pattern color=gray!70,minimum size=17pt,inner sep=0pt,line width=1.5pt] (5) circle (9pt) node[label=center: $\boldsymbol{5}$] {};        
        \draw[circle,preaction={fill=white},pattern=north east lines,pattern color=gray!70,minimum size=17pt,inner sep=0pt,line width=1.5pt] (6) circle (9pt) node[label=center: $\boldsymbol{6}$] {};
        \draw[circle,fill=white,minimum size=17pt,inner sep=0pt] (7) circle (9pt) node[label=center: $7$] {};
        \draw[circle,preaction={fill=white},pattern=north east lines,pattern color=gray!70,minimum size=17pt,inner sep=0pt,line width=1.5pt] (8) circle (9pt) node[label=center: $\boldsymbol{8}$] {};
        \draw[circle,fill=white,minimum size=17pt,inner sep=0pt] (9) circle (9pt) node[label=center: $9$] {};

        \foreach \pos/\label/\xshift/\yshift in {
        1/{4}/0em/1.5em,
        2/{\frac{7}{2}}/0em/1.5em,
        3/{4}/-1.5em/0em,
        4/{3}/0em/1.5em,
        5/{5}/0/1.5em,
        6/{8}/1.5em/0em,
        7/{\frac{9}{2}}/-1.5em/0em,
        8/{\frac{13}{2}}/1.5em/0em,
        9/{\frac{5}{2}}/-1.5em/0em
    }{
        \node[text=purple] at ([xshift=\xshift, yshift=\yshift]\pos) {\scriptsize $\label$};
    }
      \end{tikzpicture}
\subcaption{\label{subfig:graph1_hfb} $\textnormal{UB}_3(G, \boldsymbol c_a,3)=\frac{39}{2}$}
\end{subfigure}
\begin{subfigure}[b]{0.495\textwidth}
\centering
  \begin{tikzpicture}[scale=1.29]
        \footnotesize

        \pgfmathsetmacro{\a}{1.6}
        \pgfmathsetmacro{\result}{sqrt((\a)^2 - (0.5*\a)^2)}

        \draw (-3/2*\a,\result cm) coordinate (1);
        \draw (-0.5*\a,\result cm) coordinate (2);
        \draw (-\a,0cm) coordinate (3);
        \draw (\a,0cm) coordinate (6);
        \draw (0.5*\a,\result cm) coordinate (4);
        \draw (3/2*\a,\result cm) coordinate (5);
        \draw (-0.5*\a,-\result cm) coordinate (7);
        \draw (0.5*\a,-\result cm) coordinate (8);
        \draw (0,-2*\result cm) coordinate (9);

        \draw[-]  (1) -- (2) node[midway, yshift = 3ex] {$9$};
        \draw[-]  (1) -- (3) node[midway, xshift = -2ex, yshift = -1ex] {$9$};
        \draw[-]  (2) -- (3) node[midway, xshift = 2ex, yshift = -1ex] {$1$};
        \draw[-]  (2) -- (4) node[midway, yshift = 3ex] {$4$};
        \draw[-]  (3) -- (7) node[midway, xshift = -2ex, yshift = -1ex] {$1$};
        \draw[-]  (4) -- (5) node[midway, yshift = 3ex] {$1$};
        \draw[-]  (4) -- (6) node[midway, xshift = -2ex, yshift = -1ex] {$7$};
        \draw[-]  (5) -- (6) node[midway, xshift = 2ex, yshift = -1ex] {$9$};
        \draw[-]  (6) -- (8) node[midway, xshift = 2ex, yshift = -1ex] {$9$};
        \draw[-]  (7) -- (8) node[midway, yshift = 3ex] {$3$};
        \draw[-]  (7) -- (9) node[midway, xshift = -2ex, yshift = -1ex] {$6$};
        \draw[-]  (8) -- (9) node[midway, xshift = 2ex, yshift = -1ex] {$7$};

        \draw[-] (1) -- (2) node[midway,xshift=-2ex, fill=white, inner sep=1pt, text=mblue] {\scriptsize $\boldsymbol{\frac{9}{2}}$}
        node[midway,xshift=2ex, fill=white, inner sep=1pt, text=mblue] {\scriptsize $\boldsymbol{\frac{9}{2}}$};
        \draw[-] (1) -- (3) node[midway,xshift=-1ex,yshift=1.5ex, fill=white, inner sep=1pt, text=mblue] {\scriptsize $\boldsymbol{\frac{9}{2}}$}
        node[midway,xshift=1ex,yshift=-1.5ex, fill=white, inner sep=1pt, text=mblue] {\scriptsize $\boldsymbol{\frac{9}{2}}$};
        \draw[-] (2) -- (3) node[midway,xshift=1ex,yshift=1.5ex, fill=white, inner sep=1pt, text=mblue] {\scriptsize $\frac{1}{2}$}
        node[midway,xshift=-1ex,yshift=-1.5ex, fill=white, inner sep=1pt, text=mblue] {\scriptsize $\boldsymbol{\frac{1}{2}}$};
        \draw[-] (2) -- (4) node[midway,xshift=-2ex, fill=white, inner sep=1pt, text=mblue] {\scriptsize $\boldsymbol{2}$}
        node[midway,xshift=2ex, fill=white, inner sep=1pt, text=mblue] {\scriptsize $\boldsymbol{2}$};
        \draw[-] (3) -- (7) node[midway,xshift=-1ex,yshift=1.5ex, fill=white, inner sep=1pt, text=mblue] {\scriptsize $\frac{1}{2}$}
        node[midway,xshift=1ex,yshift=-1.5ex, fill=white, inner sep=1pt, text=mblue] {\scriptsize $\frac{1}{2}$};
        \draw[-] (4) -- (5) node[midway,xshift=-2ex, fill=white, inner sep=1pt, text=mblue] {\scriptsize $\frac{1}{2}$}
        node[midway,xshift=2ex, fill=white, inner sep=1pt, text=mblue] {\scriptsize $\boldsymbol{\frac{1}{2}}$};
        \draw[-] (4) -- (6) node[midway,xshift=-1ex,yshift=1.5ex, fill=white, inner sep=1pt, text=mblue] {\scriptsize $\boldsymbol{\frac{7}{2}}$}
        node[midway,xshift=1ex,yshift=-1.5ex, fill=white, inner sep=1pt, text=mblue] {\scriptsize $\frac{7}{2}$};
        \draw[-] (5) -- (6) node[midway,xshift=1ex,yshift=1.5ex, fill=white, inner sep=1pt, text=mblue] {\scriptsize $\boldsymbol{\frac{9}{2}}$}
        node[midway,xshift=-1ex,yshift=-1.5ex, fill=white, inner sep=1pt, text=mblue] {\scriptsize $\boldsymbol{\frac{9}{2}}$};
        \draw[-] (6) -- (8) node[midway,xshift=-1ex,yshift=-1.5ex, fill=white, inner sep=1pt, text=mblue] {\scriptsize $\boldsymbol{\frac{9}{2}}$}
        node[midway,xshift=1ex,yshift=1.5ex, fill=white, inner sep=1pt, text=mblue] {\scriptsize $\boldsymbol{\frac{9}{2}}$};
        \draw[-] (7) -- (8) node[midway,xshift=-2ex, fill=white, inner sep=1pt, text=mblue] {\scriptsize $\boldsymbol{\frac{3}{2}}$}
        node[midway,xshift=2ex, fill=white, inner sep=1pt, text=mblue] {\scriptsize $\frac{3}{2}$};
        \draw[-] (7) -- (9) node[midway,xshift=-1ex,yshift=1.5ex, fill=white, inner sep=1pt, text=mblue] {\scriptsize $\boldsymbol{3}$}
        node[midway,xshift=1ex,yshift=-1.5ex, fill=white, inner sep=1pt, text=mblue] {\scriptsize $\boldsymbol{3}$};
        \draw[-] (8) -- (9) node[midway,xshift=1ex,yshift=1.5ex, fill=white, inner sep=1pt, text=mblue] {\scriptsize $\boldsymbol{\frac{7}{2}}$}
        node[midway,xshift=-1ex,yshift=-1.5ex, fill=white, inner sep=1pt, text=mblue] {\scriptsize $\boldsymbol{\frac{7}{2}}$};

        \draw[circle,preaction={fill=white},pattern=north east lines,pattern color=gray!70,minimum size=17pt,inner sep=0pt,line width=1.5pt] (1) circle (9pt) node[label=center: $\boldsymbol{1}$] {};
        \draw[circle,fill=white,minimum size=17pt,inner sep=0pt] (2) circle (9pt) node[label=center: $2$] {};
        \draw[circle,fill=white,minimum size=17pt,inner sep=0pt] (3) circle (9pt) node[label=center: $3$] {};
        \draw[circle,fill=white,minimum size=17pt,inner sep=0pt] (4) circle (9pt) node[label=center: $4$] {};
        \draw[circle,fill=white,minimum size=17pt,inner sep=0pt] (5) circle (9pt) node[label=center: $5$] {};
        \draw[circle,preaction={fill=white},pattern=north east lines,pattern color=gray!70,minimum size=17pt,inner sep=0pt,line width=1.5pt] (6) circle (9pt) node[label=center: $\boldsymbol{6}$] {};
        \draw[circle,fill=white,minimum size=17pt,inner sep=0pt] (7) circle (9pt) node[label=center: $7$] {};
        \draw[circle,preaction={fill=white},pattern=north east lines,pattern color=gray!70,minimum size=17pt,inner sep=0pt,line width=1.5pt] (8) circle (9pt) node[label=center: $\boldsymbol{8}$] {};
        \draw[circle,fill=white,minimum size=17pt,inner sep=0pt] (9) circle (9pt) node[label=center: $9$] {};

        \foreach \pos/\label/\xshift/\yshift in {
        1/{9}/0em/1.5em,
        2/{\frac{13}{2}}/0em/1.5em,
        3/{5}/-1.5em/0em,
        4/{\frac{11}{2}}/0em/1.5em,
        5/{5}/0/1.5em,
        6/{9}/1.5em/0em,
        7/{\frac{9}{2}}/-1.5em/0em,
        8/{8}/1.5em/0em,
        9/{\frac{13}{2}}/-1.5em/0em
    }{
        \node[text=purple] at ([xshift=\xshift, yshift=\yshift]\pos) {\scriptsize $\label$};
    }
      \end{tikzpicture}
\subcaption{\label{subfig:graph2_hfb} $\textnormal{UB}_3(G, \boldsymbol c_b, 3)=26$}
\end{subfigure}
\end{figure}

\subsection{\Rev{Theoretical comparison between the second and third upper bounds}}\label{sec:ub2_vs_ub3}

\Rev{
In this subsection, we theoretically compare $\HFBbound$ and $\JAPbound$. While $\JAPbound$ depends on both the coloring $\mathscr{C}$ and the ordering of its independent sets, $\HFBbound$ 
depends only on~$k$, an upper bound on the clique number of~$G$. 
Since any coloring $\mathscr{C}$ provides an upper bound $|\mathscr{C}|$ on the clique number of~$G$, we set $k = |\mathscr{C}|$ to compare the two upper bounds.
We remark that a tighter upper bound on the clique number of~$G$ would yield a smaller~$k$ and thus a stronger $\HFBbound$. 
The following proposition states that, under the choice $k = |\mathscr{C}|$, $\JAPbound$ is always within a factor of~$2$ of~$\HFBbound$, independently of the total order~$\prec$.
}

\Rev{\begin{proposition}\label{prop:UB2_leq_2UB3}
For every EWMCP instance $(G,\boldsymbol c)$, for every coloring $\mathscr{C}$ of the vertices of $G$ and for every total order $\prec$ of the independent sets of $\mathscr{C}$, it holds:
\begin{equation}
    \JAPbound \leq 2\;\mathrm{UB}_3(G, \boldsymbol c, |\mathscr{C}|).
\end{equation}
\end{proposition}
}

\Rev{\begin{proof}
For every vertex $u \in V$, by the definition~\eqref{eq:jap_weight}, the total weight assigned to~$u$ in $\JAPbound$ is:
\[
\sigma_u 
\;=\; \sum_{I \in \mathscr{C} \, : \, I \prec I(u)} \max \left\{\, c_e : e \in \delta(u),\ v \in I \,\right\},
\]
which is a sum of at most $|\{I \in \mathscr{C} \, : \, I \prec I(u)\}| \leq |\mathscr{C}|-1$ terms, each being the maximum edge-weight among edges connecting~$u$ to vertices in a given independent set.

Since the independent sets $I \prec I(u)$ are pairwise disjoint, the terms in $\sigma_u$ are weights of edges incident to $u$ with distinct endpoints. 
\Rev{Hence, $\sigma_u$ is bounded above by the sum of the $|\mathscr{C}|-1$ heaviest weights of edges incident to~$u$. Therefore, using the definition of ${\eta}_u$ given in~\eqref{eq:eta_ub3}, we obtain:
\begin{equation}\label{eq:sigma_le_2eta}
\sigma_u \;\le\; \sum_{e\in{\delta}_{|\mathscr{C}|-1}(u)} 
c_{e} \;=\; 2{\eta}_u.
\end{equation}
}

It follows that, for every independent set $I \in \mathscr{C}$,
\[
\max\{\,\sigma_u : u \in I\,\}
\;\le\;
2\,\max\{\,{\eta}_u : u \in I\,\}.
\]

For each $I \in \mathscr{C}$, let $u_I^* \in \arg\max\{\,{\eta}_u : u \in I\,\}$. Hence, by the definition of $\varphi$, 

\[
\sum_{I \in \mathscr{C}} {\eta}_{u_I^*}
\;\le\;
\sum_{j=1}^{|\mathscr{C}|} {\eta}_{\varphi(j)}
\;=\;
\mathrm{UB}_3(G, \boldsymbol c, |\mathscr{C}|).
\]
Combining the above inequalities, we obtain
\[
\JAPbound
\;=\;
\sum_{I \in \mathscr{C}} \max\{\,\sigma_u : u \in I\,\}
\;\le\;
2\sum_{I \in \mathscr{C}} {\eta}_{u_I^*}
\;\le\;
2\,\mathrm{UB}_3(G, \boldsymbol c, |\mathscr{C}|),
\]
which concludes the proof.
\end{proof}}

\Rev{To show that the factor of~$2$ in Proposition~\ref{prop:UB2_leq_2UB3} is asymptotically tight, we construct a family of instances for which the ratio between the upper bound $\JAPbound$ and $\HFBbound$ converges to~$2$ as the number of vertices tends to infinity.

For every integer $k \geq 2$, let $G^4_k = (V^4_k, E^4_k)$ be a graph where the vertex set $V^4_k$ is partitioned into $k$ independent sets, each of cardinality~$k-1$, defining a coloring $\mathscr{C}^4_k$, and let $\prec^4_k$ be a total order on $\mathscr{C}^4_k$. Thus, the vertex set has cardinality $|V^4_k| = k(k-1)$.

The edge set $E^4_k$ is defined as follows. For every pair of independent sets $I, J \in \mathscr{C}^4_k$ with $I \prec^4_k J$, we select a different pair of vertices $\{v_{I,J},v_{J,I}\}$, with vertex $v_{I,J} \in I$ and vertex $v_{J,I} \in J$, and connect them by an edge of weight~$1$. Hence, the $\binom{k}{2}$ edges of weight~$1$ form a perfect matching on~$V^4_k$, meaning that each vertex is incident to exactly one edge of positive weight.

Then, every pair of non-adjacent vertices belonging to distinct independent sets is connected by an edge of weight~$0$.
These edge-weights define the edge-weight vector $\boldsymbol{c}^4_k$ for the graph~$G^4_k$. 

Since each vertex belongs to an independent set of cardinality~$k-1$ and is adjacent to all vertices in the remaining~$k-1$ independent sets, it follows that $|\delta(v)| = (k-1)^2$ for every $v \in V^4_k$, with exactly one edge in $\delta(v)$ with weight $1$ and the remaining $(k-1)^2 - 1$ having weight $0$.

By construction, the coloring $\mathscr{C}^4_k$ uses the minimum number of independent sets for the graph $G^4_k$.
An illustration of a graph of the family $G^4_k$ for $k = 3$ is provided in Figure~\ref{fig:tightness_k3}.}

{
\begin{figure}
\centering
\caption{
\Rev{Representation of the graph $G^4_3$, where the vertices are colored using the minimum number of colors $k = 3$. The label of each vertex is displayed inside the node. The three independent 
sets are $I = \{1, 6\}$ (red), $J = \{2, 3\}$ (green), and $K = \{4, 5\}$ (yellow). {The total order is given by $I \prec J \prec K$.} Edges of weight~$1$ (black) form a perfect matching on $V^4_3$, with their weight shown next to them; edges of weight~$0$ (gray) ensure the existence of a clique of cardinality~$3$.}
\label{fig:tightness_k3}
}

\vspace{1ex}

\begin{tikzpicture}[scale=1.29]
\footnotesize
\tikzset{
  zero edge/.style={gray!80!white},
  heavy edge/.style={very thick}
}

\def\R{2}
\coordinate (1) at (90:\R);
\coordinate (2) at (30:\R);
\coordinate (3) at (-30:\R);
\coordinate (4) at (-90:\R);
\coordinate (5) at (-150:\R);
\coordinate (6) at (150:\R);

\draw[heavy edge] (1) -- (2) node[midway, yshift=2ex] {$1$};

\draw[heavy edge] (3) -- (4) node[midway, yshift=-2ex] {$1$};
\draw[heavy edge] (5) -- (6) node[midway, xshift=-2ex] {$1$};

\draw[zero edge] (1) -- (3) node[midway, yshift=-2ex] {};
\draw[zero edge](1) -- (4) node[midway, yshift= 2ex]  {};

\draw[zero edge] (2) -- (5) node[midway, yshift=-2ex]  {};
\draw[zero edge](2) -- (6) node[midway, yshift= 2ex] {};

\draw[zero edge] (3) -- (5) node[midway, yshift=-2ex]  {};
\draw[zero edge] (4) -- (6) node[midway, yshift= 2ex]  {};

\draw[zero edge] (4) -- (2) node[midway, yshift= 2ex] {};
\draw[zero edge](1) -- (5) node[midway, yshift= -2ex] {};

\draw[zero edge] (3) -- (6) node[midway, yshift= -2ex] {};

\draw[circle,fill=red!25, minimum size=17pt,inner sep=0pt] (1) circle (9pt) node {$v_{I,J}$};
\draw[circle,fill=green!25,minimum size=17pt,inner sep=0pt] (2) circle (9pt) node {$v_{J,I}$};
\draw[circle,fill=green!25,minimum size=17pt,inner sep=0pt] (3) circle (9pt) node {$v_{J,K}$};
\draw[circle,fill=yellow!100!orange!55,minimum size=17pt,inner sep=0pt] (4) circle (9pt) node {$v_{K,J}$};
\draw[circle,fill=yellow!100!orange!55,minimum size=17pt,inner sep=0pt] (5) circle (9pt) node {$v_{K,I}$};
\draw[circle,fill=red!25, minimum size=17pt,inner sep=0pt] (6) circle (9pt) node {$v_{I,K}$};
\end{tikzpicture}
\end{figure}
}

\Rev{
The following proposition shows that the ratio between the upper bound $\JAPbound$ and $\HFBbound$ can be arbitrarily close to~$2$, thus proving the asymptotic tightness of Proposition~\ref{prop:UB2_leq_2UB3}.
}

\Rev{
\begin{proposition} \label{prop:tightness_UB2_UB3}
The sequence $\{(G^4_k, \boldsymbol{c}^4_k)\}_{k \geq 2}$ of EWMCP instances is such that
\begin{equation*}
\frac{\JAPquattrok}{\HFBquattrok} \rightarrow 2 \quad ~~\text{as}~~ \quad k \rightarrow +\infty,
\end{equation*}
where $\mathscr{C}^4_k$ is the coloring of $G^4_k$ with cardinality $k$ defined above and $\prec^4_k$ defines any total order of the independent sets in $\mathscr{C}^4_k$.
\end{proposition}
}

\begin{proof}
\Rev{Let us first determine the value of $\HFBquattrok$. By construction of $G^4_k$, each vertex $v_{I,J} \in I$ is incident to exactly one edge of positive weight, namely the edge $\{v_{I,J},v_{J,I}\}$ of weight~$1$, while all the other $(k-1)^2-1$ incident edges have weight~$0$. 
Therefore, by definition~\eqref{eq:eta_ub3}, for every $v_{I,J} \in V^4_k$ we have
\[
    \eta_{v_{I,J}}
    = \frac{1}{2}\sum_{e \in {\delta}_{k-1}(v_{I,J})} c_e
    = \frac{1}{2}.
\]

Since $\eta_u = 1/2$ for every $u \in V^4_k$, it follows from definition \eqref{eq:ub3} that
\[
    \HFBquattrok
    = \sum_{j=1}^{k} \eta_{\varphi(j)}
    = k \cdot \frac{1}{2}
    = \frac{k}{2}.
\]

Let us now compute $\JAPquattrok$. 
According to definition~\eqref{eq:jap}, this amounts to determining, for every independent set $I\in\mathscr{C}^4_k$, the quantity
\[
\max\{\sigma_u : u\in I\},
\]
that is the maximum total weight assigned to a vertex in $I$. We distinguish two cases:
\begin{enumerate}
\item[(i)] Let $I$ be the lower independent set in the order $\prec^4_k$. Then, there is no $J \in \mathscr{C}^4_k$ such that $J \prec^4_k I$. 
Hence, by definition~\eqref{eq:jap_weight}, for every vertex $v_{I,J} \in I$, the sum defining $\sigma_{v_{I,J}}$ is empty, and therefore
\[
    \sigma_{v_{I,J}} \;=\; 0.
\]
Thus, $\max\{\,\sigma_u : u \in I\,\} = 0.$

\item[(ii)] Let $I$ be any other independent set under $\prec^4_k$.
For every $J \in \mathscr{C}^4_k$ with $J \prec^4_k I$, by construction there is a vertex $v_{I,J} \in I$ incident to the edge $\{v_{I,J}, v_{J,I}\}$ of weight $1$, where $v_{J,I} \in J$. All other edges incident to $v_{I,J}$ have weight $0$.
Therefore, for every $v_{I,J} \in I$ with $J \prec^4_k I$,
    \[
        \sigma_{v_{I,J}} \;=\; 1.
    \]
   
On the other hand, for every $K \in \mathscr{C}^4_k$ with $I \prec^4_k K$, the unique positive-weight edge incident to $v_{I,K} \in I$ connects it to the vertex $v_{K,I}$ in $K$ and therefore does not contribute to the sum in~\eqref{eq:jap_weight}; hence
    \[
        \sigma_{v_{I,K}} \;=\; 0.
    \]
It follows that $\max\{\,\sigma_u : u \in I\,\} = 1$.
\end{enumerate}

Since $\mathscr{C}^4_k$ consists of $k$ independent sets and exactly one of them is the first under the total order $\prec^4_k$, we obtain
\[
    \JAPquattrok \;=\; \sum_{I \in \mathscr{C}^4_k} \max\{\,\sigma_u : u \in I\,\}
    \;=\; 0 + \underbrace{1 + 1 + \cdots + 1}_{k-1 \text{ terms}}
    \;=\; k-1.
\]
Finally, the ratio between the two upper bounds is 
\[
    \frac{\JAPquattrok}{\HFBquattrok} \;=\; \frac{k-1}{\dfrac{k}{2}}
    \;=\; \frac{2(k-1)}{k} \;=\; 2 - \frac{2}{k},
\]
and therefore
\[
    \lim_{k \to +\infty} \frac{\JAPquattrok}{\HFBquattrok} \;=\; 2,
\]
which concludes the proof.
}
\end{proof}

\Rev{We further provide a sequence of EWMCP instances in which the ratio between the upper bound $\HFBbound$ and $\JAPbound$ becomes arbitrarily large as the number of vertices in the graphs tends to infinity. 
}

\Rev{
For every integer $n \geq 2$, let $G^5_n = (V^5_n, E^5_n)$ be a graph where the vertex set $V^5_n$ has cardinality $2n$ and is partitioned into an independent set $I \subseteq V^5_n$ and a clique $C \subseteq V^5_n$, where 
\[
I = \{1, 2, \dots, n\}, \qquad C = \{n+1, n+2, \dots, 2n\}.
\]
The edge set $E^5_n$ of the graph includes all edges within the clique, each assigned an edge-weight of~$0$, and all edges connecting each vertex in~$I$ to every vertex in~$C$, each assigned an edge-weight of~$1$. 
These edge-weights define the edge-weight vector $\boldsymbol{c}^5_n$ for the graph~$G^5_n$. An illustration of this graph is provided in Figure~\ref{fig:hfb_jap_wins}.
}

\Rev{
Since $G^5_n$ contains a clique of cardinality~$n$, any coloring of its vertices requires at least~$n$ colors. Moreover, every vertex in~$I$ is adjacent to every vertex in~$C$, so the vertices of~$I$ require at least one additional color, giving a total of at least $n+1$ colors.
A coloring $\mathscr{C}^5_n$ of the vertices of~$G^5_n$ with the minimum number $n+1$ of colors is given by 
\[
\mathscr{C}^5_n = \{I, \{n+1\}, \{n+2\}, \dots, \{2n\}\}.
\]
}

\begin{figure}[H]
    \centering
    \caption{
    \Rev{Representation of the graph \( G^5_n \), where the vertices are colored using the minimum number of colors. Each vertex is labeled with a natural number displayed inside the node. The dashed ellipse on the left represents the independent set \( I \), while the ellipse on the right represents the clique \( C \). Edges of weight~$1$ connecting a vertex in~$I$ to a vertex in~$C$ are shown in black, while edges of weight~$0$ within~$C$ are shown in gray. Edge-weight labels are omitted from the figure.}
    \label{fig:hfb_jap_wins}
  }
    \begin{tikzpicture}[scale=1.29]

        \pgfmathsetmacro{\a}{1.6}
        \pgfmathsetmacro{\r}{sqrt((\a)^2 - (0.5*\a)^2)}

        \draw (1,0) coordinate (1);
        \draw (1,-\r) coordinate (2);
        \draw (1,-3*\r) coordinate (3);

        \draw (5,0) coordinate (4);
        \draw (5,-\r) coordinate (5);
        \draw (5,-3*\r) coordinate (6);

        \draw (3, -2*\r) node {$\spacedvdots$};
        \draw (1, -2*\r) node {$\spacedvdots$};
        \draw (5, -2*\r) node {$\spacedvdots$};

        \draw[-]  (1) -- (4);
        \draw[-]  (1) -- (5);
        \draw[-]  (2) -- (4);
        \draw[-]  (2) -- (5);
        \draw[-]  (3) -- (4);
        \draw[-]  (3) -- (5);
        \draw[-]  (3) -- (6);
        \draw[-]  (2) -- (6);
        \draw[-]  (1) -- (6);

        \draw[-, gray]  (4) -- (5);
        \draw[-, gray] (5) to[out=20, in=20] (6);
        \draw[-, gray] (4) to[out=0, in=0] (6);

        \node[draw, dashed, ellipse, minimum width=4.4cm, inner sep=0pt, fit=(1) (2) (3), label=above: $I$] {};
        \node[draw, ellipse, minimum width=4.4cm, inner sep=0pt, fit=(4) (5) (6), label=above: $C$] {};
        
        \draw[circle,fill=red!25,minimum size=17pt,inner sep=0pt] (1) circle (9pt) node[label=center: \scriptsize{$1$}] {};
        \draw[circle,fill=red!25,minimum size=17pt,inner sep=0pt] (2) circle (9pt) node[label=center: \scriptsize{$2$}] {};
        \draw[circle,fill=red!25,minimum size=17pt,inner sep=0pt] (3) circle (9pt) node[label=center: \scriptsize{$n$}] {};

        \draw[circle,fill=green!25,minimum size=17pt,inner sep=0pt] (4) circle (9pt) node[label=center: \scriptsize{$n\!+\!1$}] {};
        \draw[circle,fill=yellow!100!orange!55,minimum size=17pt,inner sep=0pt] (5) circle (9pt) node[label=center: \scriptsize{$n\!+\!2$}] {};       
        \draw[circle,fill=azzurro!30,minimum size=17pt,inner sep=0pt] (6) circle (9pt) node[label=center: \scriptsize{$2n$}] {};
    
    \end{tikzpicture}
\end{figure}

\Rev{
The following proposition states that the ratio between the upper bound $\HFBbound$
and $\JAPbound$ can be arbitrarily large.

\begin{proposition}\label{prop:hfb_over_jap}
The sequence $\{(G^5_n, \boldsymbol{c}^5_n)\}_{n \geq 2}$ of EWMCP instances is such that
\begin{equation*}
\frac{\mathrm{UB}_3(G^5_n, \boldsymbol{c}^5_n, n+1)}{\mathrm{UB}_2(G^5_n, \boldsymbol{c}^5_n, \mathscr{C}^5_n, \prec^5_n)} \rightarrow +\infty \quad ~~\text{as}~~ \quad n \rightarrow +\infty,
\end{equation*}
where $\mathscr{C}^5_n$ is the coloring of $G^5_n$ with cardinality $n+1$ defined above and $\prec^5_n$ defines any total order of the independent sets in $\mathscr{C}^5_n$.
\end{proposition}
}

\begin{proof} 
\Rev{
We first compute $\HFBcinquen$, with {$k = |\mathscr{C}^5_n| = n+1$}. Since $k - 1 = n$, for every vertex $u \in V^5_n$, the set ${\delta}_n(u)$ consists of the $n$ heaviest edges incident to~$u$.

For every vertex $u \in I$, the vertex~$u$ has exactly $n$ incident edges (one to each vertex of~$C$), all of weight~$1$. Therefore, for $u\in I$, 
by definition~\eqref{eq:eta_ub3},
\[
\eta_u = \frac{1}{2} \sum_{e \in {\delta}_n(u)} c_e = \frac{1}{2} \cdot n \cdot 1 = \frac{n}{2}.
\]

For every vertex $u \in C$, the vertex $u$ has $n$ incident edges to the vertices of~$I$, each of weight~$1$, and $n - 1$ incident edges to the other vertices of~$C$, each of weight~$0$. 
Thus, ${\delta}_n(u)$ consists of the $n$ edges connecting~$u$ to each vertex of~$I$, and
\[
\eta_u = \frac{1}{2} \sum_{e \in {\delta}_n(u)} c_e = \frac{1}{2} \cdot n \cdot 1 = \frac{n}{2}.
\]

Since $\eta_u = n/2$ for every $u \in V^5_n$, it follows from definition~\eqref{eq:ub3} that
\[
\HFBcinquen = \sum_{j=1}^{n+1} \eta_{\varphi(j)} = (n+1) \cdot \frac{n}{2} = \frac{n(n+1)}{2}.
\]

We now compute $\JAPcinquen$ for an arbitrary total order $\prec^5_n$. 

We first determine~$\sigma_u$ for every vertex $u \in I$. Suppose that the independent set $I$ occupies position $p$, with $p\in\{1,2,\ldots,n+1\}$, in the total order $\prec^5_n$, so that $I$ is preceded by exactly $p-1$ independent sets, each containing a single vertex of $C$.

Since every vertex $u \in I$ is connected to every vertex of~$C$ by an edge of weight~$1$, each independent set $J\prec^5_n I$ contributes~$1$ to the sum in 
definition~\eqref{eq:jap_weight}. Hence, for every $u \in I$,
\[
\sigma_u = \sum_{\substack{J \in \mathscr{C}^5_n \colon J \prec^5_n I}} \max\{c_e : e \in \delta(u),\, v \in J\} = (p-1) \cdot 1 = p - 1.
\]

Let us now determine $\sigma_u$ for every vertex $u \in C$.
By definition of the coloring $\mathscr{C}^5_n$, each vertex $u \in C$ belongs to a distinct independent set $\{u\} \in \mathscr{C}^5_n$. We distinguish two cases, depending on whether~$\{u\}$ precedes or follows~$I$ under the total order $\prec^5_n$.

If $\{u\}$ precedes $I$ under the total order $\prec^5_n$, then only independent sets containing a single vertex of $C$ precede $\{u\}$, each contributing $0$ to the sum in definition~\eqref{eq:jap_weight}, since all edges within $C$ have weight $0$. Hence, for every $u \in C$ with $\{u\} \prec^5_n I$,
\[
\sigma_u = \sum_{\substack{J \in \mathscr{C}^5_n \colon J \prec^5_n \{u\}}} \max\{c_e : e \in \delta(u),\, v \in J\} = 0.
\]

If instead $I \prec^5_n \{u\}$, then the independent set $I$ contributes $1$ to the sum in definition~\eqref{eq:jap_weight}, since $u$ is connected to every vertex of $I$ by an edge of weight $1$. All other preceding independent sets contain a single vertex of~$C$, each contributing~$0$ as above. Hence, for every $u \in C$ with $I \prec^5_n \{u\}$,
\[
\sigma_u = \sum_{\substack{J \in \mathscr{C}^5_n \colon J \prec^5_n \{u\}}} \max\{c_e : e \in \delta(u),\, v \in J\} = 1.
\]

Since exactly $p-1$ independent sets precede~$I$ and $n-p+1$ follow it in the total order $\prec^5_n$, there are $n-p+1$ vertices $u \in C$ with $I \prec^5_n \{u\}$ and $p-1$ vertices $u \in C$ with $\{u\} \prec^5_n I$. By definition~\eqref{eq:jap}, we obtain
\begin{align*}
\JAPcinquen &= \sum_{J \in \mathscr{C}^5_n} 
  \max\{\sigma_u : u \in J\} \\
&= \max\{\sigma_u : u \in I\} 
  + \sum_{\{u\} \in \mathscr{C}^5_n \setminus \{I\}} \sigma_u \\
&= (p-1) + (n-p+1) \cdot 1 + (p-1) \cdot 0 \\
&= n.
\end{align*}

Hence $\JAPcinquen = n$ for every total order $\prec^5_n$.

Finally, the ratio between the two upper bounds is
\[
\frac{\HFBcinquen}{\JAPcinquen} = \frac{\dfrac{n(n+1)}{2}}{n} = \frac{n+1}{2},
\]
and therefore
\[
\lim_{n \to +\infty} \frac{\HFBcinquen}{\JAPcinquen} = +\infty,
\]
which concludes the proof.}
\end{proof}

\subsection{\Rev{Theoretical comparison between $\HFBbound$ and $\OSBound$}}\label{sec:ub3_vs_ub1}

\Rev{We show that either of the two upper bounds $\HFBbound$ and $\OSBound$ can be arbitrarily worse than the other, by using the families of instances $\{(G^2_n, \boldsymbol{c}^2_n)\}$ and $\{(G^3_n, \boldsymbol{c}^3_n)\}$ introduced in Section~\ref{sec:THEOR_COMP}.}

\Rev{
\begin{proposition}\label{prop:hfb_over_os}
The sequence $\{(G^2_n,\boldsymbol c^2_n)\}$ of EWMCP instances is such that
\begin{equation*}
\frac{\mathrm{UB}_3(G^2_n,\boldsymbol c^2_n,n)}{\OSBduen} \rightarrow + \infty \quad ~~\text{as}~~ \quad n \rightarrow +\infty,
\end{equation*}
where $\mathscr{C}^2_n$ is the coloring of $G^2_n$ with cardinality $n$, defined in Section~\ref{sec:THEOR_COMP}.

\end{proposition}
}

\Rev{
\begin{proof}
By \Cref{prop:example_itaspa_wins}, for any total order $\prec^2_n$ of the independent sets of $\mathscr{C}^2_n$, we have
\[
\frac{\JAPduen}{\OSBduen} \rightarrow +\infty \quad \text{as} \quad n \rightarrow +\infty.
\]
By \Cref{prop:UB2_leq_2UB3}, with $k = |\mathscr{C}^2_n| = n$ it holds that, for every $n \geq 1$,
\[
\mathrm{UB}_3(G^2_n,\boldsymbol c^2_n,n) \geq \frac{1}{2}\,\JAPduen.
\]
Therefore,
\[
\lim_{n \to +\infty} \frac{\mathrm{UB}_3(G^2_n,\boldsymbol c^2_n,n)}{\OSBduen} 
\;\geq\; \lim_{n \to +\infty} \frac{1}{2}\,\frac{\JAPduen}{\OSBduen} 
\;=\; +\infty.
\]
\end{proof}
}

\Rev{
\begin{proposition}\label{prop:os_over_hfb}
The sequence $\{(G^3_n,\boldsymbol c^3_n)\}$ of EWMCP instances is such that
\begin{equation*}
\frac{\OSBtren}{\mathrm{UB}_3(G^3_n,\boldsymbol c^3_n,2)} \rightarrow + \infty \quad ~~\text{as}~~ \quad n \rightarrow +\infty,
\end{equation*}
where $\mathscr{C}^3_n$ is the coloring of $G^3_n$ with cardinality~$2$, defined in Section~\ref{sec:THEOR_COMP}.
\end{proposition}
}

\Rev{
\begin{proof}
As shown in the proof of \Cref{prop:example_japan_wins}, $\OSBtren = n$.

Let us now compute $\mathrm{UB}_3(G^3_n, \boldsymbol c^3_n, 2)$. Let $k = 2 = |\mathscr{C}^3_n|$. Since $k = 2$, for every vertex $u \in V^3_n$, the set $\delta_{k-1}(u) = \delta_1(u)$ contains a single edge, namely the heaviest edge incident to~$u$. 
Since all edge-weights are equal to~$1$, by definition~\eqref{eq:eta_ub3}, we have
\[
\eta_u = \frac{1}{2}\sum_{e \in \delta_1(u)} c_e = \frac{1}{2} \cdot 1 
= \frac{1}{2}
\]
for every $u \in V^3_n$.
It follows from definition~\eqref{eq:ub3} that
\[
\mathrm{UB}_3(G^3_n, \boldsymbol c^3_n, 2) = \sum_{j=1}^{k} \eta_{\varphi(j)} 
= 2 \cdot \frac{1}{2} = 1.
\]
Therefore,
\[
\lim_{n \to +\infty} \frac{\OSBtren}{\mathrm{UB}_3(G^3_n, \boldsymbol c^3_n, 2)} 
= \lim_{n \to +\infty} \frac{n}{1} = +\infty.
\]
\end{proof}
}

\subsection{\Rev{Theoretical comparison between $\HFBbound$ and $\optimum$}}\label{sec:ub3_vs_optimum}

\Rev{Finally, the following proposition shows that the ratio between $\HFBbound$ and the edge-weighted clique number $\optimum$ can be arbitrarily large, by considering the family of instances $\{(G^1_n, \boldsymbol{c}^1_n)\}$ introduced in Section~\ref{sec:WCA}.}

\Rev{
\begin{proposition}\label{prop:example_bad_hfb}
The sequence $\{(G^1_n,\boldsymbol c^1_n)\}$ of EWMCP instances is such that
\begin{equation*}
\frac{\HFBunon}{\optunon} \rightarrow + \infty \quad   ~~\text{as}~~ \quad n \rightarrow +\infty,
\end{equation*}
where $n$ coincides with the clique number of $G^1_n$. 
\end{proposition}
}

\Rev{
\begin{proof}
By Proposition~\ref{prop:example_bad_japan}, for any coloring $\mathscr{C}^1_n$ of $G^1_n$ and any total order $\prec^1_n$ of the independent sets of $\mathscr{C}^1_n$, we have
\[
\frac{\JAPunon}{\optunon} \rightarrow +\infty \quad \text{as} \quad n \rightarrow +\infty.
\]
By construction, the minimum possible number of colors for $G^1_n$ is $n$; let $\mathscr{C}^1_n$ be a coloring with $|\mathscr{C}^1_n| = n$. By Proposition~\ref{prop:UB2_leq_2UB3} with $k = |\mathscr{C}^1_n| = n$, it holds that, for every $n \geq 1$,\[
\HFBunon \geq \frac{1}{2}\,\JAPunon.
\]
Therefore,
\[
\lim_{n \to +\infty} \frac{\HFBunon}{\optunon} \;\geq\; 
\lim_{n \to +\infty} \frac{1}{2}\,\frac{\JAPunon}{\optunon} \;=\; +\infty.
\]
\end{proof}
}

\section{Computational comparison}
\label{sec:COMP_COMP}

In this section, we report the results of our extensive experimental analysis, conducted on both DIMACS instances and randomly generated instances commonly used in the literature for the Edge-Weighted Maximum Clique Problem (EWMCP). This computational campaign provides practical insights into the empirical strength of the upper bounds on~$\optimum$ analyzed in this work: $\OSBound$, proposed by \citet{san2019new} and described in Section~\ref{sec:UB1}, $\JAPbound$, proposed by \citet{shimizu2020maximum} and described in Section~\ref{sec:UB2}, \Rev{and $\HFBbound$, proposed by \citet{hosseinian2020lagrangian} and described in \Cref{sec:UB3}}.

The purpose of our tests is twofold. On the one hand, we aim to determine which of the upper bounds is closer to~$\optimum$ in practice. On the other hand, we seek to assess whether \Rev{their} empirical strength depends on instance characteristics (such as the number of vertices and graph density) or on algorithmic choices (such as the type of coloring and the ordering of its independent sets).

Both $\OSBound$ and $\JAPbound$ depend on a coloring $\mathscr{C}$ of the instance graph $G$. \Rev{In our computational experiments, such coloring is obtained using the DSatur heuristic~\citep{BrelazDSatur}. In the following, we refer to a coloring obtained via the DSatur algorithm as \textit{DSatur coloring}.
Unlike the other two upper bounds, $\HFBbound$ depends on an upper bound $k$ on the clique number of $G$, instead of a coloring of the graph. In order to ensure a fair comparison among the three bounds studied in this paper, we always set $k$ to be equal to the cardinality of the same coloring $\mathscr{C}$ used by $\OSBound$ and $\JAPbound$.}
For the sake of brevity, in the remainder of this section we refer to the upper bounds simply as $\mathrm{UB}_1$, $\mathrm{UB}_2$ \Rev{and $\mathrm{UB}_3$}.

All experiments are conducted on a Linux machine equipped with a 2.30~GHz CPU and 512~GB of RAM. The code is implemented in C$++$, and we use IBM ILOG CPLEX 22.11 to compute~$\UB_1$, as it requires solving the LP Model~\eqref{form:LP-italia-spagna}. 
Extensive preliminary tests show that the most efficient algorithm for this purpose is the barrier method, without performing the crossover phase, since we are only interested in the value of the optimal solution of Model~\eqref{form:LP-italia-spagna}. Indeed, our computational experiments demonstrate that the barrier algorithm solves the problem to optimality in a time that is orders of magnitude lower than that required by alternative LP solvers, such as the primal simplex, dual simplex, or network simplex methods.

Detailed numerical results, beyond those presented in the paper, along with the source code for computing the upper bounds, are available at the following GitHub repository: \url{https://github.com/FabioCiccarelli/EWMCP_Bounds.git}. This material supports precise benchmarking and facilitates comparisons with alternative approaches. We also hope that our findings and the accompanying resources will stimulate further research on upper bounds for~$\optimum$, with the aim of identifying tighter values that can enhance the performance of bounding procedures in exact branch-and-bound algorithms for the EWMCP.

The remainder of this section is organized as follows. In \Cref{sec:benchmark_datasets}, we describe the benchmark datasets used in the experiments. In \Cref{sec:UB2_ordering}, we analyze the impact of the ordering of the independent sets on the quality of $\UB_2$. In \Cref{sec:dimacs_results} and \Cref{sec:random_tests}, we report the results on DIMACS and RANDOM instances, respectively. In \Cref{sec:branching_study}, we assess the behavior of the three bounds in a simulated branch-and-bound setting. Finally, in \Cref{sec:LP_comparison}, we compare $\UB_1$ with other LP-based upper bounds from the literature.

\subsection{Benchmark datasets of instances}\label{sec:benchmark_datasets}
In this work, we consider both DIMACS instances and randomly generated (RANDOM) instances.
\begin{itemize}
    \item The DIMACS instances, see~\cite{dimacs2017}, consist of 80 graphs that are commonly used in the literature to evaluate the performance of exact algorithms for the EWMCP, such as those proposed by \citet{shimizu2020maximum}, \citet{san2019new} \Rev{and \citet{hosseinian2020lagrangian}}. These algorithms are able to determine~$\optimum$ for 48 of the 80 instances, making this benchmark a valuable reference for assessing the strength of upper bounds.

    \item The RANDOM instances are generated according to specific pairs of number of vertices and edge density~$(|V|, \mu)$, following common practice in the literature. For each combination of \(|V| \in \{10, 20, \dots, 100\}\) and \(\mu \in \{0.1, 0.2, \dots, 0.9\}\), we generate 10 instances. Additionally, for the case \(|V| = 100\), we include higher densities in the range \([0.91, 0.99]\) with a step size of~0.01, resulting in a total of 990 graphs. \Rev{Notice that the 10 instances for each combination $(|V|, \mu)$ are generated independently from each other.}
    These instances are designed to evaluate the influence of instance characteristics on upper bound quality. Graphs with up to 100 vertices are considered, as they can be solved to optimality by state-of-the-art exact methods, allowing for a direct comparison with~$\optimum$. Smaller graphs are included to assess the impact of the number of vertices, while high-density graphs (91–99\%) are included due to their well-known difficulty for exact EWMCP algorithms. 
\end{itemize}

Edge-weights are assigned following a standard approach commonly adopted in the literature~\cite{gouveia2015solving, hosseinian2018nonconvex, shimizu2019branch,san2019new}. Specifically, for each edge \( e  \in E \), the edge-weight is defined as \( \pe_e := ((u + v) \bmod 200) + 1 \).

\Rev{\subsection{Sensitivity of $\JAPbound$ to the ordering of the independent sets}\label{sec:UB2_ordering}}

\Rev{As discussed in \Cref{sec:UB2}, $\JAPbound$ depends on the total order $\prec$ of the independent sets of $\mathscr{C}$: different orderings yield different vertex weights, as defined in~\eqref{eq:jap_weight} and, consequently, different bound values.} 

\Rev{
We compare five ordering criteria to determine whether any of them consistently minimizes the value of $\UB_2$: \textit{natural ordering} (NAT), in which the independent sets are kept in the order returned by the coloring algorithm (e.g., DSatur); \textit{non-decreasing size} (SZ$\uparrow$) and \textit{non-increasing size} (SZ$\downarrow$), in which independent sets are sorted by their cardinality; \textit{non-decreasing star-weight score} (SW$\uparrow$) and \textit{non-increasing star-weight score} (SW$\downarrow$), in which independent sets are sorted by the average, over their vertices, of the total weight of the $(k{-}1)$ heaviest edges incident to each vertex $u \in V$, i.e., the total weight of the edges in $\delta_{k-1}(u)$ as defined in \Cref{sec:UB3}.}

\Rev{The size-based orderings are motivated by the observation that the cardinality of each independent set determines the number of candidate vertices competing for the maximum in~\eqref{eq:jap_weight}, so placing larger or smaller sets earlier may affect the bound. The star-weight score provides instead a proxy for the potential contribution of each independent set to the overall bound.}

\Rev{Table~\ref{table:sorting_pctgap_aggregate} compares the five strategies on the DIMACS instances, reporting, for each strategy and instance family, the number of instances for which the strategy yields the tightest upper bound (\#best) and the average percentage increase of the bound relative to the minimum over all five orderings (\% diff.). All the instances have been tested using DSatur as a coloring algorithm. Sorting the independent sets by non-decreasing cardinality (SZ$\uparrow$) dominates the other strategies: it achieves the lowest overall average deviation from the minimum (0.12\%) and yields the tightest bound on 72 of the 80 instances. The ordering by non-decreasing star weight score (SW$\uparrow$) is the second-best alternative, with an average deviation of 3.68\% from the minimum bound, though it is preferable to SZ$\uparrow$ only for the \texttt{MANN} family. The remaining orders (NAT, SZ$\downarrow$, SW$\downarrow$) yield weaker bounds, with average deviations ranging from 4.49\% to 6.83\%. Based on these results, throughout the remainder of the computational section we compute $\UB_2$ by sorting the independent sets of the coloring in non-decreasing order of cardinality.}

\begin{table}[ht]
\centering
\footnotesize
\renewcommand\arraystretch{1.3}
\tabcolsep=3.2pt
\caption{\Rev{Average percentage increase over the best ordering (\% diff.) and number of best results (\#best) per sorting strategy, on DIMACS instances, grouped by instance family. The upper bound is computed using DSatur coloring.}}
\label{table:sorting_pctgap_aggregate}
\Rev{\begin{tabular}{Trrrrrrrrrrrrrrrrr}
\toprule
\multicolumn{2}{l}{Instances}                 &  &  & \multicolumn{2}{r}{NAT} &  & \multicolumn{2}{r}{SZ$\uparrow$} &  & \multicolumn{2}{r}{SZ$\downarrow$} &  & \multicolumn{2}{r}{SW$\uparrow$} &  & \multicolumn{2}{r}{SW$\downarrow$} \\ \cline{1-3} \cline{5-6} \cline{8-9} \cline{11-12} \cline{14-15} \cline{17-18} 
\rmfamily{family}  & \# &  &  & \% diff.       & \#best      &  & \% diff.            & \#best          &  & \% diff.            & \#best            &  & \% diff.           & \#best           &  & \% diff.             & \#best           \\ \cline{1-2} \cline{5-6} \cline{8-9} \cline{11-12} \cline{14-15} \cline{17-18} 
\\[-3ex]
brock                           & 12 &  &  & 8.63  & 0 &  & \textbf{0.00} & \textbf{12} &  & 9.06          & 0          &  & 3.02 & 0 &  & 6.10 & 0 \\
C                               & 7  &  &  & 5.41  & 0 &  & \textbf{0.00} & \textbf{7}  &  & 6.31          & 0          &  & 3.49 & 0 &  & 5.05 & 0 \\
c-fat                           & 7  &  &  & 6.23  & 0 &  & \textbf{0.65} & \textbf{5}  &  & 7.24          & 0          &  & 5.85 & 1 &  & 5.82 & 1 \\
DSJC                            & 2  &  &  & 8.59  & 0 &  & \textbf{0.00} & \textbf{2}  &  & 9.06          & 0          &  & 4.40 & 0 &  & 5.81 & 0 \\
gen                             & 5  &  &  & 8.21  & 0 &  & \textbf{0.00} & \textbf{5}  &  & 10.38         & 0          &  & 5.08 & 0 &  & 6.00 & 0 \\
hamming                         & 6  &  &  & 3.60  & 1 &  & \textbf{0.00} & \textbf{6}  &  & 4.70          & 1          &  & 3.94 & 1 &  & 3.26 & 1 \\
johnson                         & 4  &  &  & 11.25 & 0 &  & \textbf{0.00} & \textbf{4}  &  & 11.35         & 0          &  & 8.88 & 0 &  & 6.68 & 0 \\
keller                          & 3  &  &  & 8.73  & 0 &  & \textbf{0.00} & \textbf{3}  &  & 11.86         & 0          &  & 6.66 & 0 &  & 6.11 & 0 \\
MANN                            & 4  &  &  & 2.42  & 0 &  & 1.04          & 1           &  & \textbf{0.92} & \textbf{3} &  & 1.09 & 0 &  & 1.79 & 0 \\
p\_hat                          & 15 &  &  & 4.25  & 0 &  & \textbf{0.00} & \textbf{15} &  & 5.17          & 0          &  & 3.20 & 0 &  & 2.74 & 0 \\
san                             & 11 &  &  & 2.72  & 1 &  & \textbf{0.09} & \textbf{8}  &  & 3.68          & 0          &  & 1.68 & 0 &  & 2.04 & 2 \\
sanr                            & 4  &  &  & 10.11 & 0 &  & \textbf{0.00} & \textbf{4}  &  & 10.52         & 0          &  & 2.18 & 0 &  & 8.30 & 0     \\[1.2ex] \cline{1-2} \cline{5-6} \cline{8-9} \cline{11-12} \cline{14-15} \cline{17-18} 
Tot/Avg          & 80         &  &  & 6.00          & 2         &  & \textbf{0.12}             & \textbf{72}             &  & 6.83     & 4    &  & 3.68             & 2             &  & 4.49              & 4\\
\bottomrule
\end{tabular}}
\end{table}

\subsection{Results for DIMACS  instances}\label{sec:dimacs_results}

In this section, we present the results of the tests conducted on the DIMACS instances. \Rev{Throughout this section, all bounds are computed using a DSatur coloring $\mathscr{C}$.
Specifically, $\UB_1$ and $\UB_2$ use $\mathscr{C}$ directly, with the independent sets of $\mathscr{C}$ sorted in non-decreasing order of cardinality for $\UB_2$; $\UB_3$ uses instead
$k = |\mathscr{C}|$.} 

Table~\ref{table:results_dimacs_DSatur} reports detailed results for each individual DIMACS instance. For each instance, the table lists its name, number of vertices~$|V|$, number of edges~$|E|$, \Rev{the clique number $\omega(G)$ of the instance graph (when known),} and the value of~$\optimum$ (when known). \Rev{The cardinality of the DSatur coloring of the instance graph, as well as the average size of its independent sets, are reported in the following two columns. Finally, the last columns show the value of each of the three upper bounds}. For each instance, the smaller of the upper bound values is highlighted in bold.
The results reported in the table clearly show that~$\UB_2$ outperforms~$\UB_1$ \Rev{and $\UB_3$} on nearly all DIMACS instances, with the only exceptions being {\tt c-fat200-1} and {\tt c-fat200-5}, where $\UB_1$ achieves the smallest value among the three bounds. 
\Rev{Furthermore, we observe a correlation between the number of vertices $|V|$ and the relative tightness of the bounds. Specifically, for instances with a similar edge density, the relative gap between $\UB_1$ and $\UB_2$ widens as the graph size increases. This behavior is evident within the {\tt C} instance family (with a constant density of roughly~$0.9$), where the ratio $\UB_1 / \UB_2$ grows from approximately $1.45$ for {\tt C125.9} to nearly $1.90$ for {\tt C2000.9}, indicating that~$\UB_2$ scales better on larger topologies.
Graph density also plays a crucial role in the relative performance of the bounds. Observing the {\tt san} instances, it can be noted that the advantage of $\UB_2$ over $\UB_1$ is remarkably pronounced on medium-density graphs. For example, on {\tt san400\_0.5\_1} (density $0.5$), the value of $\UB_1$ is more than seven times larger than that of $\UB_2$. However, as density increases towards $0.9$ (e.g., {\tt san400\_0.9\_1}), this divergence narrows considerably, although $\UB_2$ still maintains a clear superiority.
The results shown in \Cref{table:results_dimacs_DSatur} also highlight a correlation between the performance of $\UB_2$ and that of $\UB_3$. In the vast majority of the test set, $\UB_3$ is slightly higher than $\UB_2$, consistently yielding tighter bound values than $\UB_1$. Noticeable deviations between $\UB_2$ and $\UB_3$ emerge primarily on exceptionally dense or massive graphs, such as those belonging to the {\tt MANN} family, where the differences between the two bounds become more apparent.}
Nevertheless, despite the dominance of~$\UB_2$ over~$\UB_1$ and $\UB_3$, it is important to note that all the upper bounds are generally weak and often far from~$\optimum$. On average, the ratio between bound value and the optimal value exceeds \Rev{the value of 10 for all the considered upper bounds}, highlighting their limited tightness. \Rev{Nevertheless, $\UB_2$ achieves an average ratio slightly above $11$, less than half than that of $\UB_1$, and around $15\%$ smaller than that of $\UB_3$.}

\scriptsize
\renewcommand
\arraystretch{1.2}
\tabcolsep=3pt
\begin{longtable}[c]{Trrrrrrrrrrrrrrrrrrr}
\caption{Values of the three upper bounds on DIMACS instances. A DSatur coloring $\mathscr{C}$ is used for computing $\UB_1$ and $\UB_2$. For $\UB_2$, the independent sets of $\mathscr{C}$ are sorted in non-decreasing order of cardinality. $\UB_3$ uses $k = |\mathscr{C}|$.
}
\label{table:results_dimacs_DSatur}\\
\toprule
\multicolumn{5}{l}{Instances}                                   &  & \multicolumn{2}{r}{DSatur coloring} &  &  &                  &  &  &                     &  &  &             \\ \cline{1-5} \cline{7-8}
name             & $|V|$ & $|E|$     & $\omega(G)$ & $\optimum$ &  & card           & avg size           &  &  & $\UB_1$          &  &  & $\UB_2$             &  &  & $\UB_3$     \\ \cline{1-5} \cline{7-8} \cline{11-11} \cline{14-14} \cline{17-17} 
\endfirsthead
\endhead
\endfoot
\endlastfoot
\\[-2.5ex]
brock200\_1      & 200   & 14,834    & 21          & 21,230     &  & 55             & 3.64               &  &  & 346,395          &  &  & \textbf{206,229}    &  &  & 248,468     \\
brock200\_2      & 200   & 9,876     & 12          & 6,542      &  & 34             & 5.88               &  &  & 139,041          &  &  & \textbf{80,625}     &  &  & 96,843      \\
brock200\_3      & 200   & 12,048    & 15          & 10,303     &  & 41             & 4.88               &  &  & 209,408          &  &  & \textbf{121,862}    &  &  & 141,200     \\
brock200\_4      & 200   & 13,089    & 17          & 13,967     &  & 50             & 4.00               &  &  & 271,661          &  &  & \textbf{163,015}    &  &  & 204,973     \\
brock400\_1      & 400   & 59,723    & 27          & --         &  & 98             & 4.08               &  &  & 1,277,900        &  &  & \textbf{679,356}    &  &  & 808,989     \\
brock400\_2      & 400   & 59,786    & 29          & --         &  & 99             & 4.04               &  &  & 1,276,930        &  &  & \textbf{689,958}    &  &  & 824,164     \\
brock400\_3      & 400   & 59,681    & 31          & 46,785     &  & 98             & 4.08               &  &  & 1,269,990        &  &  & \textbf{692,312}    &  &  & 809,756     \\
brock400\_4      & 400   & 59,765    & 33          & 54,304     &  & 100            & 4.00               &  &  & 1,286,460        &  &  & \textbf{717,069}    &  &  & 839,816     \\
brock800\_1      & 800   & 207,505   & 23          & --         &  & 143            & 5.59               &  &  & 3,218,480        &  &  & \textbf{1,524,480}  &  &  & 1,779,610   \\
brock800\_2      & 800   & 208,166   & 24          & --         &  & 146            & 5.48               &  &  & 3,351,520        &  &  & \textbf{1,591,550}  &  &  & 1,849,540   \\
brock800\_3      & 800   & 207,333   & 25          & --         &  & 142            & 5.63               &  &  & 3,183,140        &  &  & \textbf{1,498,410}  &  &  & 1,755,070   \\
brock800\_4      & 800   & 207,643   & 26          & --         &  & 142            & 5.63               &  &  & 3,264,490        &  &  & \textbf{1,513,150}  &  &  & 1,758,500   \\[1.2ex]
C125.9           & 125   & 6,963     & 34          & 66,248     &  & 52             & 2.40               &  &  & 279,002          &  &  & \textbf{192,784}    &  &  & 227,809     \\
C250.9           & 250   & 27,984    & 44          & --         &  & 98             & 2.55               &  &  & 928,441          &  &  & \textbf{637,829}    &  &  & 778,450     \\
C500.9           & 500   & 112,332   & --          & --         &  & 175            & 2.86               &  &  & 3,458,980        &  &  & \textbf{2,159,570}  &  &  & 2,547,210   \\
C1000.9          & 1,000 & 450,079   & --          & --         &  & 316            & 3.16               &  &  & 12,581,900       &  &  & \textbf{7,058,440}  &  &  & 8,277,240   \\
C2000.5          & 2,000 & 999,836   & 16          & --         &  & 223            & 8.97               &  &  & 9,876,680        &  &  & \textbf{3,894,920}  &  &  & 4,461,250   \\
C2000.9          & 2,000 & 1,799,532 & --          & --         &  & 586            & 3.41               &  &  & 47,423,100       &  &  & \textbf{24,977,100} &  &  & 28,891,500  \\
C4000.5          & 4,000 & 4,000,268 & 18          & --         &  & 399            & 10.03              &  &  & 35,853,700       &  &  & \textbf{12,836,300} &  &  & 14,447,900  \\[1.2ex]
c-fat200-1       & 200   & 1,534     & 12          & 7,734      &  & 15             & 13.33              &  &  & \textbf{10,677}  &  &  & 11,226              &  &  & 14,115      \\
c-fat200-2       & 200   & 3,235     & 24          & 26,389     &  & 24             & 8.33               &  &  & 36,326           &  &  & \textbf{31,203}     &  &  & 37,990      \\
c-fat200-5       & 200   & 8,473     & 58          & 168,200    &  & 84             & 2.38               &  &  & \textbf{304,982} &  &  & 347,871             &  &  & 363,583     \\
c-fat500-1       & 500   & 4,459     & 14          & 10,738     &  & 14             & 35.71              &  &  & 14,422           &  &  & \textbf{11,982}     &  &  & 14,441      \\
c-fat500-2       & 500   & 9,139     & 26          & 38,350     &  & 26             & 19.23              &  &  & 50,942           &  &  & \textbf{40,522}     &  &  & 51,108      \\
c-fat500-5       & 500   & 23,191    & 64          & 205,864    &  & 64             & 7.81               &  &  & 292,208          &  &  & \textbf{224,323}    &  &  & 289,644     \\
c-fat500-10      & 500   & 46,627    & 126         & 804,000    &  & 126            & 3.97               &  &  & 1,174,240        &  &  & \textbf{846,003}    &  &  & 1,122,240   \\[1.2ex]
DSJC500\_5       & 500   & 62,624    & 13          & 9,626      &  & 71             & 7.04               &  &  & 760,032          &  &  & \textbf{371,513}    &  &  & 445,446     \\
DSJC1000\_5      & 1,000 & 249,826   & 15          & 12,054     &  & 122            & 8.20               &  &  & 2,695,560        &  &  & \textbf{1,144,330}  &  &  & 1,321,400   \\[1.2ex]
gen200\_p0.9\_44 & 200   & 17,910    & 44          & 94,362     &  & 63             & 3.17               &  &  & 441,099          &  &  & \textbf{262,929}    &  &  & 327,415     \\
gen200\_p0.9\_55 & 200   & 17,910    & 55          & 150,839    &  & 77             & 2.60               &  &  & 590,043          &  &  & \textbf{387,808}    &  &  & 467,936     \\
gen400\_p0.9\_55 & 400   & 71,820    & 55          & --         &  & 77             & 5.19               &  &  & 1,107,870        &  &  & \textbf{458,318}    &  &  & 528,254     \\
gen400\_p0.9\_65 & 400   & 71,820    & 65          & --         &  & 102            & 3.92               &  &  & 1,528,460        &  &  & \textbf{780,508}    &  &  & 894,911     \\
gen400\_p0.9\_75 & 400   & 71,820    & 75          & --         &  & 103            & 3.88               &  &  & 1,521,530        &  &  & \textbf{784,813}    &  &  & 910,735     \\[1.2ex]
hamming6-2       & 64    & 1,824     & 32          & 32,736     &  & 32             & 2.00               &  &  & 60,192           &  &  & \textbf{33,152}     &  &  & 47,120      \\
hamming6-4       & 64    & 704       & 4           & 396        &  & 7              & 9.14               &  &  & 4,264            &  &  & \textbf{1,907}      &  &  & 2,145       \\
hamming8-2       & 256   & 31,616    & 128         & 800,624    &  & 130            & 1.97               &  &  & 1,558,810        &  &  & \textbf{945,680}    &  &  & 1,277,780   \\
hamming8-4       & 256   & 20,864    & 16          & 12,360     &  & 24             & 10.67              &  &  & 146,767          &  &  & \textbf{42,399}     &  &  & 53,533      \\
hamming10-2      & 1,024 & 518,656   & 512         & --         &  & 537            & 1.91               &  &  & 26,113,600       &  &  & \textbf{15,730,300} &  &  & 21,388,200  \\
hamming10-4      & 1,024 & 434,176   & --          & --         &  & 90             & 11.38              &  &  & 2,798,370        &  &  & \textbf{613,716}    &  &  & 767,997     \\[1.2ex]
johnson8-2-4     & 28    & 210       & 4           & 192        &  & 6              & 4.67               &  &  & 1,079            &  &  & \textbf{470}        &  &  & 678         \\
johnson8-4-4     & 70    & 1,855     & 14          & 6,552      &  & 19             & 3.68               &  &  & 29,569           &  &  & \textbf{15,934}     &  &  & 19,786      \\
johnson16-2-4    & 120   & 5,460     & 8           & 3,808      &  & 15             & 8.00               &  &  & 51,783           &  &  & \textbf{14,375}     &  &  & 20,263      \\
johnson32-2-4    & 496   & 107,880   & 16          & --         &  & 31             & 16.00              &  &  & 454,263          &  &  & \textbf{79,785}     &  &  & 90,891      \\[1.2ex]
keller4          & 171   & 9,435     & 11          & 6,745      &  & 33             & 5.18               &  &  & 143,426          &  &  & \textbf{71,782}     &  &  & 96,508      \\
keller5          & 776   & 225,990   & 27          & --         &  & 123            & 6.31               &  &  & 2,693,370        &  &  & \textbf{1,058,210}  &  &  & 1,381,980   \\
keller6          & 3,361 & 4,619,898 & --          & --         &  & 282            & 11.92              &  &  & 24,531,900       &  &  & \textbf{6,110,990}  &  &  & 7,601,260   \\[1.2ex]
MANN\_a9         & 45    & 918       & 16          & 5,460      &  & 21             & 2.14               &  &  & 15,019           &  &  & \textbf{8,546}      &  &  & 14,668      \\
MANN\_a27        & 378   & 70,551    & 126         & --         &  & 144            & 2.63               &  &  & 2,390,210        &  &  & \textbf{1,100,410}  &  &  & 1,696,000   \\
MANN\_a45        & 1,035 & 533,115   & 345         & --         &  & 375            & 2.76               &  &  & 17,904,100       &  &  & \textbf{7,256,760}  &  &  & 11,639,900  \\
MANN\_a81        & 3,321 & 5,506,380 & 1100        & --         &  & 1160           & 2.86               &  &  & 184,688,000      &  &  & \textbf{69,376,700} &  &  & 111,821,000 \\[1.2ex]
p\_hat300-1      & 300   & 10,933    & 8           & 3,321      &  & 24             & 12.50              &  &  & 77,256           &  &  & \textbf{42,970}     &  &  & 50,724      \\
p\_hat300-2      & 300   & 21,928    & 25          & 31,564     &  & 46             & 6.52               &  &  & 291,593          &  &  & \textbf{160,414}    &  &  & 187,784     \\
p\_hat300-3      & 300   & 33,390    & 36          & 63,390     &  & 72             & 4.17               &  &  & 719,292          &  &  & \textbf{386,513}    &  &  & 452,424     \\
p\_hat500-1      & 500   & 31,569    & 9           & 4,764      &  & 36             & 13.89              &  &  & 200,341          &  &  & \textbf{97,409}     &  &  & 116,933     \\
p\_hat500-2      & 500   & 62,946    & 36          & 63,870     &  & 69             & 7.25               &  &  & 797,997          &  &  & \textbf{385,609}    &  &  & 430,757     \\
p\_hat500-3      & 500   & 93,800    & 50          & --         &  & 112            & 4.46               &  &  & 1,901,550        &  &  & \textbf{950,409}    &  &  & 1,102,070   \\
p\_hat700-1      & 700   & 60,999    & 11          & 5,185      &  & 44             & 15.91              &  &  & 344,530          &  &  & \textbf{154,795}    &  &  & 176,541     \\
p\_hat700-2      & 700   & 121,728   & 44          & --         &  & 89             & 7.87               &  &  & 1,399,730        &  &  & \textbf{655,765}    &  &  & 721,643     \\
p\_hat700-3      & 700   & 183,010   & 62          & --         &  & 148            & 4.73               &  &  & 3,427,090        &  &  & \textbf{1,675,380}  &  &  & 1,932,600   \\
p\_hat1000-1     & 1,000 & 122,253   & 10          & 5,436      &  & 56             & 17.86              &  &  & 637,294          &  &  & \textbf{260,576}    &  &  & 287,901     \\
p\_hat1000-2     & 1,000 & 244,799   & 46          & --         &  & 116            & 8.62               &  &  & 2,607,780        &  &  & \textbf{1,100,310}  &  &  & 1,231,560   \\
p\_hat1000-3     & 1,000 & 371,746   & 68          & --         &  & 198            & 5.05               &  &  & 6,770,190        &  &  & \textbf{3,042,040}  &  &  & 3,459,520   \\
p\_hat1500-1     & 1,500 & 284,923   & 12          & 7,135      &  & 78             & 19.23              &  &  & 1,377,910        &  &  & \textbf{510,199}    &  &  & 564,782     \\
p\_hat1500-2     & 1,500 & 568,960   & 65          & --         &  & 169            & 8.88               &  &  & 5,751,650        &  &  & \textbf{2,358,580}  &  &  & 2,633,230   \\
p\_hat1500-3     & 1,500 & 847,244   & 94          & --         &  & 284            & 5.28               &  &  & 14,909,000       &  &  & \textbf{6,432,260}  &  &  & 7,197,470   \\[1.2ex]
san200\_0.7\_1   & 200   & 13,930    & 30          & 45,295     &  & 31             & 6.45               &  &  & 191,976          &  &  & \textbf{76,677}     &  &  & 84,637      \\
san200\_0.7\_2   & 200   & 13,930    & 18          & 15,073     &  & 18             & 11.11              &  &  & 122,916          &  &  & \textbf{28,141}     &  &  & 29,265      \\
san200\_0.9\_1   & 200   & 17,910    & 70          & 242,710    &  & 74             & 2.70               &  &  & 582,448          &  &  & \textbf{367,724}    &  &  & 436,794     \\
san200\_0.9\_2   & 200   & 17,910    & 60          & 178,468    &  & 75             & 2.67               &  &  & 578,346          &  &  & \textbf{370,206}    &  &  & 446,774     \\
san200\_0.9\_3   & 200   & 17,910    & 44          & 96,764     &  & 56             & 3.57               &  &  & 430,170          &  &  & \textbf{229,653}    &  &  & 263,921     \\
san400\_0.5\_1   & 400   & 39,900    & 13          & 7,442      &  & 14             & 28.57              &  &  & 129,393          &  &  & \textbf{17,229}     &  &  & 17,855      \\
san400\_0.7\_1   & 400   & 55,860    & 40          & 77,719     &  & 41             & 9.76               &  &  & 544,666          &  &  & \textbf{143,893}    &  &  & 154,432     \\
san400\_0.7\_2   & 400   & 55,860    & 30          & 44,155     &  & 30             & 13.33              &  &  & 413,269          &  &  & \textbf{80,724}     &  &  & 83,551      \\
san400\_0.7\_3   & 400   & 55,860    & 22          & 24,727     &  & 29             & 13.79              &  &  & 338,061          &  &  & \textbf{74,233}     &  &  & 78,142      \\
san400\_0.9\_1   & 400   & 71,820    & 100         & --         &  & 107            & 3.74               &  &  & 1,766,750        &  &  & \textbf{876,509}    &  &  & 976,900     \\
san1000          & 1,000 & 250,500   & 15          & 10,661     &  & 16             & 62.50              &  &  & 376,843          &  &  & \textbf{23,456}     &  &  & 23,840      \\[1.2ex]
sanr200\_0.7     & 200   & 13,868    & 18          & 16,398     &  & 51             & 3.92               &  &  & 298,303          &  &  & \textbf{175,001}    &  &  & 214,149     \\
sanr200\_0.9     & 200   & 17,863    & 42          & --         &  & 81             & 2.47               &  &  & 632,063          &  &  & \textbf{433,622}    &  &  & 509,992     \\
sanr400\_0.5     & 400   & 39,984    & 13          & 8,298      &  & 61             & 6.56               &  &  & 507,447          &  &  & \textbf{259,181}    &  &  & 319,893     \\
sanr400\_0.7     & 400   & 55,869    & 21          & 22,791     &  & 90             & 4.44               &  &  & 1,082,260        &  &  & \textbf{577,693}    &  &  & 686,163     \\[1ex]
\bottomrule
\end{longtable}
\normalsize

In Table~\ref{table:results_families_dimacs}, the DIMACS instances are instead grouped into families based on the similarity of their names. For each family, the table reports its name \Rev{and} the number of instances it contains (column $\#$).
The next columns provide the minimum, maximum, and average \Rev{\textit{relative ratio} obtained by $\UB_1$ across the instances in each family, computed for each instance and each upper bound as the ratio between its value and that of the bound achieving the minimum value (i.e., $\min \, \{\UB_1, \, \UB_2, \, \UB_3\}$) on the same instance.} These are followed by the corresponding statistics for~$\UB_3$. \Rev{The value of $\UB_2$ is consistently lower than those of $\UB_1$ and $\UB_3$ across the entire testbed of instances. On the sole {\tt c-fat200-1} and {\tt c-fat200-5} instances, $\UB_2$ has a value greater than that of $\UB_1$. For this reason, \Cref{table:results_families_dimacs} does not contain columns regarding the relative ratio achieved by $\UB_2$, which is equal to $1$ for all instances but two. Specifically, on the {\tt c-fat200-1} instance, $\UB_2$ achieves a relative ratio of $1.05$, while on the {\tt c-fat200-5} instance it equals $1.14$.}

\begin{table}[h]
\footnotesize
\centering
\renewcommand
\arraystretch{1.2}
\tabcolsep=8.0pt
\caption{\Rev{Relative ratio achieved by $\UB_1$ and $\UB_3$ on DIMACS instances, grouped by family. For each bound and family: minimum, maximum, and average relative ratio (ratio between the bound value and $\min \, \{\UB_1, \, \UB_2, \, \UB_3\}$). A DSatur coloring $\mathscr{C}$ is used for computing $\UB_1$ and $\UB_2$. For $\UB_2$, the independent sets of $\mathscr{C}$ are sorted in non-decreasing order of cardinality. $\UB_3$ uses $k = |\mathscr{C}|$. Since $\UB_2$ achieves a relative ratio of $1.00$ on all instances but two ({\tt c-fat200-1} and {\tt c-fat200-5}), its data are omitted from the table.}
}
\label{table:results_families_dimacs}
\begin{tabular}{Trrrrrrrrrrr}
\toprule
\multicolumn{2}{l}{}              &    &  &  & \multicolumn{3}{r}{relative ratio of $\UB_1$} &  & \multicolumn{3}{r}{relative ratio of $\UB_3$} \\ \cline{1-3} \cline{6-8} \cline{10-12} 
\rmfamily{Families of instances} &  & \# &  &  & min        & max        & avg      &  & min        & max       & avg       \\ \cline{1-3} \cline{6-8} \cline{10-12} 
                               &  &    &  &  &            &            &          &  &            &           &           \\[-2.5ex]
brock                          &  & 12 &  &  & 1.67       & 2.16      & 1.89     &  & 1.16       & 1.26      & 1.18     \\
C                              &  & 7  &  &  & 1.45       & 2.79      & 1.93     &  & 1.13       & 1.22      & 1.17     \\
c-fat                          &  & 7  &  &  & 1.00       & 1.39      & 1.19     &  & 1.19       & 1.33      & 1.26     \\
DSJC                           &  & 2  &  &  & 2.05       & 2.36      & 2.20     &  & 1.15       & 1.20      & 1.18     \\
gen                            &  & 5  &  &  & 1.52       & 2.42      & 1.90     &  & 1.15       & 1.25      & 1.18     \\
hamming                        &  & 6  &  &  & 1.65       & 4.56      & 2.56     &  & 1.12       & 1.42      & 1.30     \\
johnson                        &  & 4  &  &  & 1.86       & 5.69      & 3.36     &  & 1.14       & 1.44      & 1.31     \\
keller                         &  & 3  &  &  & 2.00       & 4.01      & 2.85     &  & 1.24       & 1.34      & 1.30     \\
MANN                           &  & 4  &  &  & 1.76       & 2.66      & 2.26     &  & 1.54       & 1.72      & 1.62     \\
p\_hat                         &  & 15 &  &  & 1.80       & 2.70      & 2.17     &  & 1.10       & 1.20      & 1.14     \\
san                            &  & 11 &  &  & 1.56       & 16.07     & 4.63     &  & 1.02       & 1.21      & 1.09     \\
sanr                           &  & 4  &  &  & 1.46       & 1.96      & 1.75     &  & 1.18       & 1.23      & 1.21  \\[1.2ex] \cline{1-3} \cline{6-8} \cline{10-12} 
\rmfamily{Tot/Avg}  &  & 80 &  &  &            &            & 2.44     &  &            &           & 1.21      \\ \bottomrule
\end{tabular}
\end{table}

\Rev{The results reported in \Cref{table:results_families_dimacs} provide insights on the relative strength of the bounds: it emerges, for example, that $\UB_1$ is the upper bound which tends to have the worst performance. Indeed, its average relative ratio is more than $2.4$, and for some families of instances (like {\tt johnson} and {\tt san}) it even reaches values higher than $3$. On the other hand, $\UB_3$ seems to ensure a more stable practical performance, being on average $1.2$ times larger than the tightest upper bound among the three.}

It is also worth noting that computing~$\UB_1$ is considerably more time-consuming than~$\UB_2$ \Rev{and $\UB_3$}, as it requires solving Model~\eqref{form:LP-italia-spagna}. For instances with many edges, and consequently a large number of variables, solving the model can take several tens of seconds and, in some cases, even exceed a few hundred seconds. In contrast, the evaluation of $\UB_2$ \Rev{and $\UB_3$} is extremely fast and incurs negligible computational time.

\Rev{In order to assess the robustness of the upper bounds to the specific coloring used, we also tested each instance with five additional colorings computed via a randomized algorithm. This algorithm first generates a random ordering of the vertices and then assigns each vertex, in that order, to the first available color that does not contain any of its neighbors. We refer to a coloring obtained via this randomized coloring algorithm as \textit{random coloring}. Results show that all the upper bounds generally exhibit weaker performance when computed using random colorings, compared to those obtained with the DSatur coloring. This is likely due to DSatur's tendency to produce colorings with fewer independent sets. The value of $\UB_3$ is indeed directly related to the value of $k$, which in our tests coincides with the cardinality of the coloring. As far as the other two bounds are concerned, our computational experiments show that, as the cardinality of coloring increases, the values of both $\UB_1$ and $\UB_2$ tend to increase accordingly.}

\subsection{Results for RANDOM instances}
\label{sec:random_tests}

In this section, we present the results of the tests conducted on the RANDOM instances. \Rev{Throughout this section, all bounds are computed using a DSatur coloring $\mathscr{C}$.
Specifically, $\UB_1$ and $\UB_2$ use $\mathscr{C}$ directly, with the independent sets of $\mathscr{C}$ sorted in non-decreasing order of cardinality for $\UB_2$; $\UB_3$ uses instead $k = |\mathscr{C}|$.} 

To compare the quality of the three bounds independently of the absolute weight scale of each instance, we measure their strength via the \emph{strength ratio}, defined as the ratio between the bound value and the optimal value of the instance. A strength ratio of 1 indicates a bound whose value coincides with the optimal value of the instance, while values greater than 1 indicate overestimation: a bound with a lower strength ratio is therefore stronger.
Figure~\ref{fig:boxplot_comparison_ratio_opt_nodes} displays the strength ratio of the three bounds across different instance sizes, ranging from 10 to 90 vertices, using a box plot representation. Each box summarizes the distribution of \Rev{the strength ratio achieved by one of the bounds} for all instances of the same size, showing the median, interquartile range, and potential outliers. Specifically, each box plot is based on 90 instances with densities ranging from 10\% to 90\%, resulting in a total of 90 upper bound evaluations.
The results reveal that the median \Rev{strength ratio} for \Rev{all the bounds} is significantly large even for relatively small instances (i.e., for $|V| \geq 30$), confirming the general weakness of these upper bounds, as previously observed on the DIMACS benchmarks. \Rev{On these instances, the average strength ratio} is $3.5$ for~$\UB_1$, $2.7$ for~$\UB_2$, and $3.2$ for $\UB_3$, with extreme cases \Rev{exceeding $8$ for $\UB_1$ and $\UB_3$, while the maximum ratio achieved by $\UB_2$ is around $6.6$}. 
Moreover, the \Rev{strength ratio} shows a clear upward trend with increasing graph size, highlighting the growing difficulty of obtaining tight upper bounds as~$|V|$ increases. \Rev{Similarly}, the variability of the ratio values increases for larger instances, as indicated by the \Rev{growing} interquartile range. This suggests that the performance of the upper bounds, in terms of deviation from~$\optimum$, becomes less consistent as the graph size grows.

\begin{figure}[H]
\centering
\caption{\Rev{Box plot representation of the strength ratio of the three upper bounds on RANDOM instances, grouped by number of vertices from 10 to 90. Box plots for~$\UB_1$ are filled with a blue diagonal line pattern, those for~$\UB_2$ with a red dotted pattern, and those for $\UB_3$ with a green crossed-line pattern. A DSatur coloring $\mathscr{C}$ is used for computing $\UB_1$ and $\UB_2$. For $\UB_2$, the independent sets of $\mathscr{C}$ are sorted in non-decreasing order of cardinality. $\UB_3$ uses $k = |\mathscr{C}|$.}}
\label{fig:boxplot_comparison_ratio_opt_nodes}
    \includegraphics[width=0.85\linewidth]{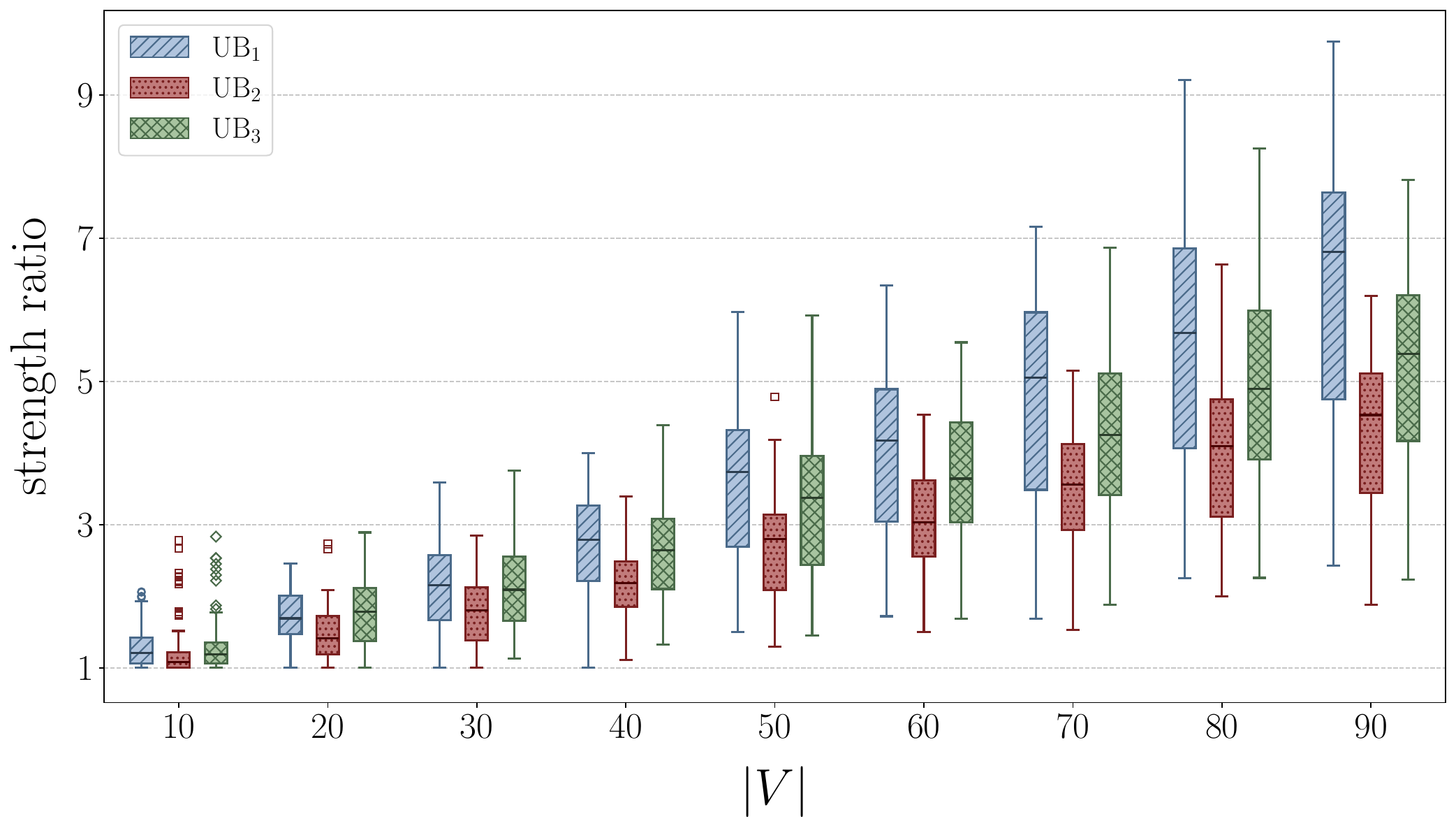}
\end{figure}

Figure~\ref{fig:comp_opt_ratio_density_100} illustrates the influence of graph density on the \Rev{strength ratio} of the \Rev{three} upper bounds, focusing specifically on instances with $|V| = 100$, using a box plot representation. 
Hence, each group consists of $10$ instances. Subfigure~\subref{subfig:comp_opt_ratio_density_1} shows results for densities from 10\% to 90\%, while Subfigure~\subref{subfig:comp_opt_ratio_density_2} focuses on high-density values from 91\% to 99\%.
The results show that the three upper bounds perform worst at low to moderate densities (below 70\%). Conversely, for denser graphs, a monotonic decrease in the median \Rev{strength ratio} is observed, with the lowest values—around $1.02$—attained for highly dense instances ($\mu = 0.99$).
As the box plots reveal, \Rev{all the upper bounds} exhibit weak performance across most density levels \Rev{for the instances with $|V|=100$}, with considerable median \Rev{strength ratios} for nearly all density values. Notably, sparse graphs appear to be more challenging than dense ones: for 10\%-dense instances, \Rev{both $\UB_1$ and $\UB_2$ have a median ratio of more than $3$, while that of $\UB_3$ is approximately $4.6$}.
Finally, and in line with the results observed for the DIMACS instances, the box plots confirm that~$\UB_2$ consistently outperforms~$\UB_1$ \Rev{and $\UB_3$} in terms of \Rev{strength ratio}, regardless of the edge density or size of the graph.

\Rev{Similarly to what has been done for DIMACS instances, we have also tested each RANDOM instance using five random colorings, in addition to the DSatur one. Results confirm that, on average, using DSatur as a coloring algorithm guarantees a better performance than using random colorings, for all three upper bounds. Indeed, the latter tend to produce colorings with one more independent set, on average, than those obtained via DSatur. The average advantage of DSatur is moderate for $\UB_1$, but grows substantially for $\UB_2$ and for $\UB_3$, and it is more pronounced on high density instances.}

\begin{figure}[h!]
\caption{\Rev{Box plot representation of the strength ratio of the three upper bounds on RANDOM instances with $|V| = 100$, grouped by edge density. Box plots for~$\UB_1$ are filled with a blue diagonal line pattern, those for~$\UB_2$ with a red dotted pattern, and those for $\UB_3$ with a green crossed-line pattern. A DSatur coloring $\mathscr{C}$ is used for computing $\UB_1$ and $\UB_2$. For $\UB_2$, the independent sets of $\mathscr{C}$ are sorted in non-decreasing order of cardinality. $\UB_3$ uses $k = |\mathscr{C}|$.}}

\label{fig:comp_opt_ratio_density_100}
\medskip
\begin{subfigure}[b]{\textwidth}
    \centering
    \includegraphics[width=0.8\linewidth]{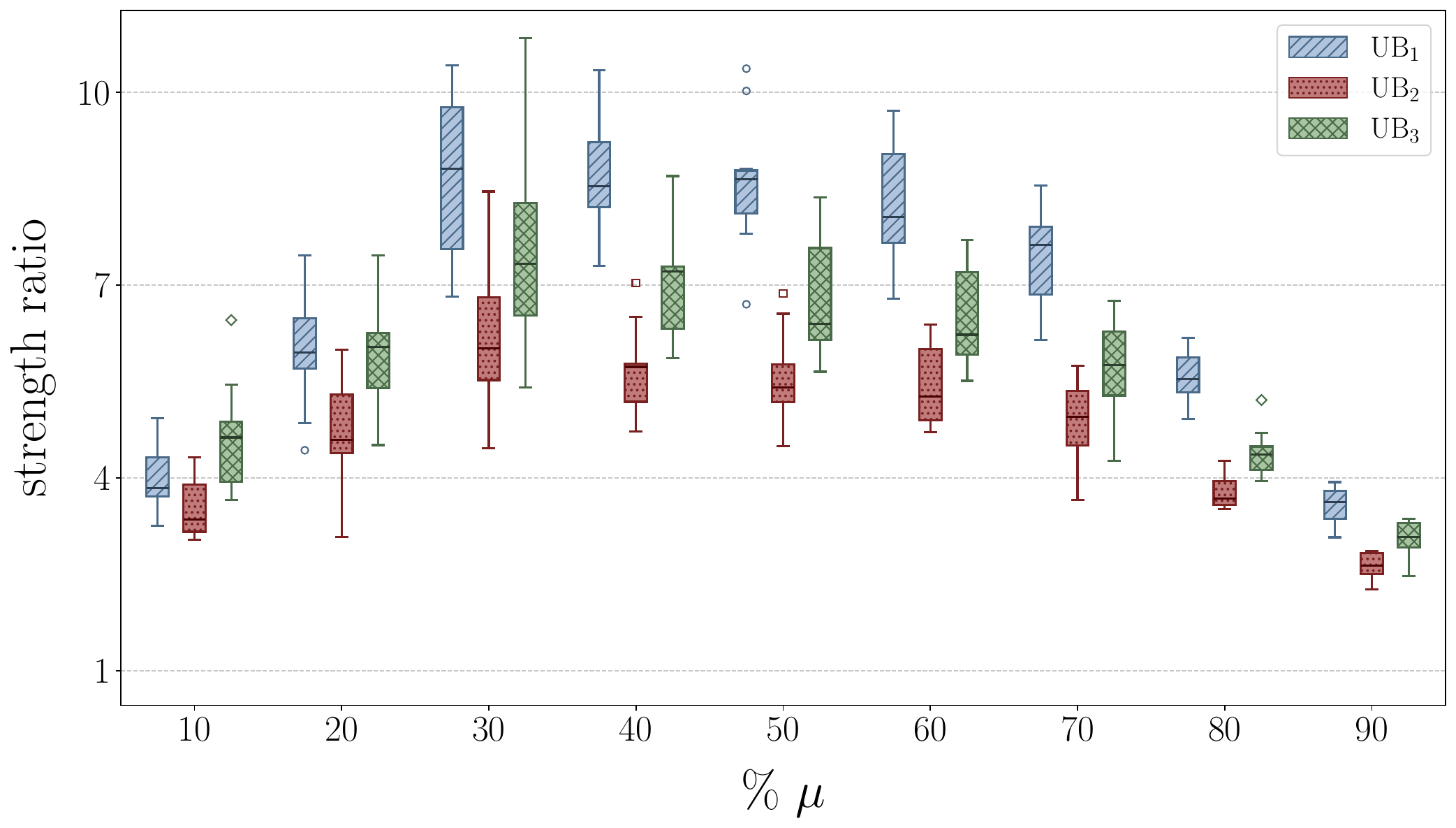}
    \subcaption{}
    \label{subfig:comp_opt_ratio_density_1}
\end{subfigure}
\medskip
\begin{subfigure}[b]{\textwidth}
    \centering
    \includegraphics[width=0.8\linewidth]{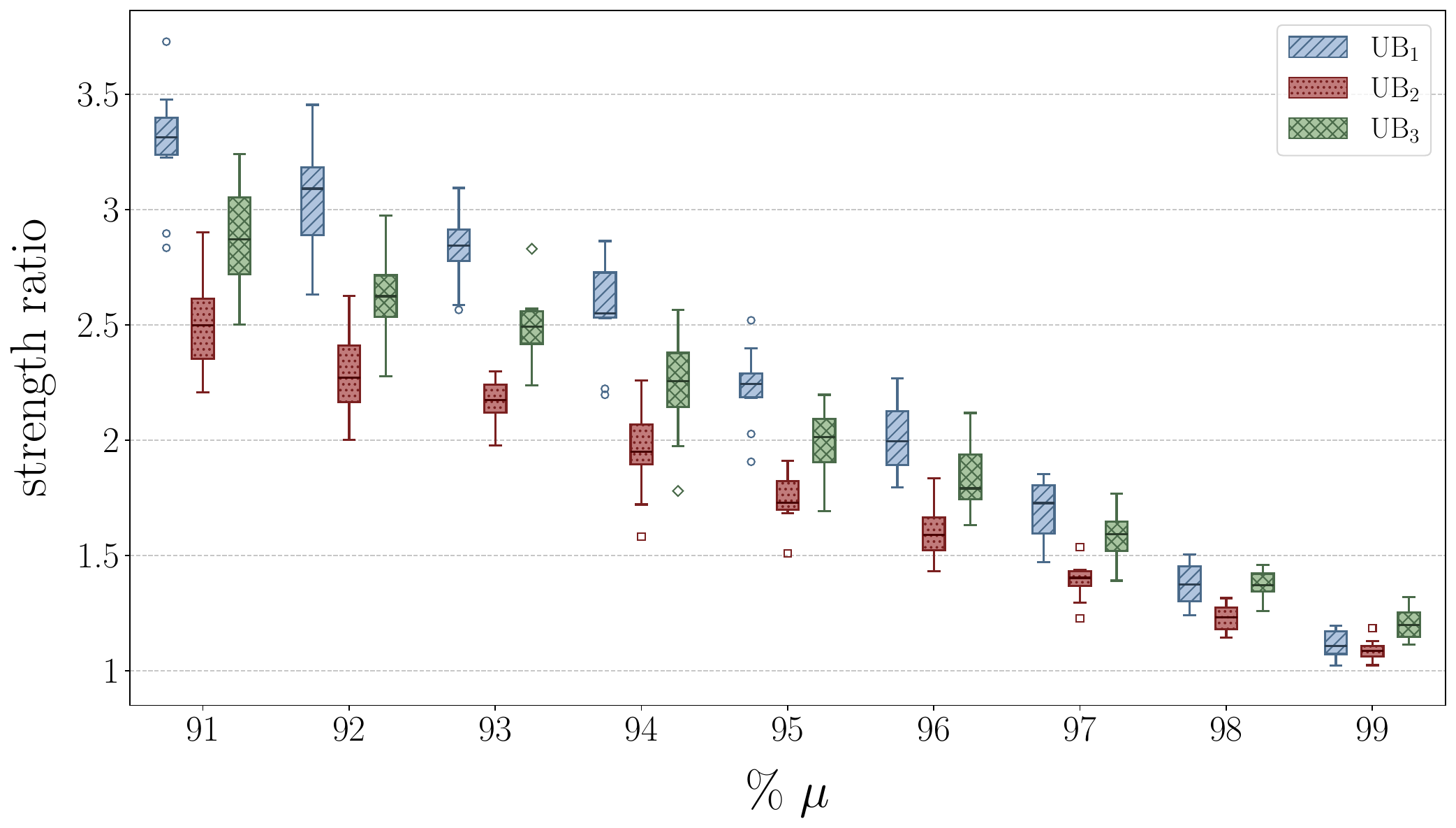}
    \subcaption{}
    \label{subfig:comp_opt_ratio_density_2}
\end{subfigure}
\end{figure}

\Rev{\subsection{Strength of the upper bounds in a branch-and-bound context} \label{sec:branching_study}}

 \Rev{
The results shown in the previous sections have been obtained computing each upper bound on the full original graph $G$. However, in the state-of-the-art branch-and-bound algorithms for the EWMCP, the bounds are computed multiple times: at each node of the search tree, a partial clique $C \subset V$ has already been fixed in a partial solution, and the upper bound is computed only on the subgraph $G[L]$ induced by the remaining candidate vertices $L \subset V$, namely the vertices adjacent to all those in $C$. The total upper bound at that node is then $W(C) + \UB(C, L)$, where $W(C)$ denotes the total edge-weight of the current clique, and $\UB(C, L)$ is an upper bound on the maximum contribution that the vertices of $L$ can give to the current partial solution value. The idea for computing $\UB(C,L)$ is common to all the bounds studied in this article: each candidate vertex $v \in L$ is adjacent, by definition, to every vertex currently in $C$, so it \textit{absorbs} the weights of the edges connecting it to $C$. Formally, for every $v \in L$, define $\vartheta_v := \sum_{u \in C} c_{\{u,v\}}$ as the total weight that $v$ would contribute to the clique $C$ if added to it. The total upper bound at the current node then decomposes as $W(C) + \UB(C,L)$, where $\UB(C,L)$ is computed on $G[L]$ with $\vartheta_v$ playing the role of an additional offset added to each vertex $v \in L$. For $\UB_2$ and $\UB_3$, this offset is incorporated directly into the vertex weights $\sigma_v$ and $\eta_v$, as described in \citet{shimizu2020maximum} and \citet{hosseinian2020lagrangian}, respectively. For $\UB_1$, it is instead taken into account by increasing the right-hand side of Constraints~\eqref{con:LP-italia-spagna_max-w} by $\vartheta_v$ for vertices $v \in L$, and solving the LP on the subgraph $G[L]$ with this modified formulation.
To assess how quickly each bound is able to prune nodes in practice, we simulate a single random path through the branching tree. The simulation proceeds as follows. Starting from an empty clique $C = \emptyset$ and candidate set $L = V$, at each step a random vertex $v$ is selected from $L$ and added to $C$; then $L$ is updated to $L \cap N(v)$, where $N(v)$ denotes the neighborhood of $v$, thus retaining only the vertices adjacent to all current members of $C$. After each update, $G[L]$ is re-colored using the DSatur heuristic. The independent sets of the resulting coloring are sorted in non-decreasing order of cardinality; for $\UB_3$, we set $k = |\mathscr{C}|$, where $\mathscr{C}$ denotes the coloring computed at the current node. The simulation terminates either when $W(C) + \mathrm{UB}(C,L) \leq \optimum$ (the node is pruned by the bound) or when $L$ becomes empty (a maximal clique has been reached without pruning). We refer to the number of vertices in $C$ at termination as the \emph{pruning depth}. Observe that we use the optimal EWMCP value $\optimum$ as the incumbent value throughout the simulation. This constitutes an idealized setting: in a real branch-and-bound algorithm, the incumbent is a lower bound to $\optimum$, and its value may significantly affect the depth reached before being able to prune a node. \RevTwo{Consequently, this experiment is designed solely to compare the relative node-dependent pruning depth of the bounds under a common idealized protocol; it does not evaluate overall branch-and-bound running time or total search-tree size.}
The simulation is run on the 48 DIMACS instances for which $\optimum$ is known, exploring 10 independent random branchings per instance, yielding 480 runs per bound. 

Table~\ref{table:branching_perinstance} reports the minimum, maximum and average pruning depth over the 10 seeds for each instance and each upper bound. \RevTwo{Under this protocol,} $\UB_2$ achieves the lowest average depth on 42 of the 48 instances. Such dominance is confirmed by the average depth values: $\UB_2$ prunes at a shallower depth ($4.1$ on average) than both $\UB_1$ and $\UB_3$ ($4.9$ and $4.8$, respectively). For some instances, the performance gap between $\UB_2$ and the other two bounds is particularly relevant: as an example, for the $\tt hamming8-2$ instance the average pruning depth obtained using $\UB_2$ as an upper bound is approximately half of those obtained  with $\UB_1$ and $\UB_3$. Conversely, on all the instances belonging to the {\tt c-fat} family, $\UB_1$ proves to be the strongest upper bound: indeed, when using $\UB_1$ as an upper bound, pruning occurs at depth $1$ in all the branching simulations performed on these instances. \RevTwo{These results concern relative node-dependent pruning depth only and should not be interpreted as evidence about the running time or total search-tree size of a complete branch-and-bound algorithm.}
}

\footnotesize
\renewcommand\arraystretch{1}
\tabcolsep=2.5pt
\begin{longtable}[c]{Trrrrrrrrrrrrr}
\caption{
\Rev{Minimum, maximum, and average pruning depth over 10 simulated branchings,
for each of the 48 DIMACS instances with known optimum. At each node of the 
simulation, all bounds are recomputed on the current subgraph using a DSatur 
coloring $\mathscr{C}$. For $\UB_2$, the independent sets of each coloring are sorted in non-decreasing order of cardinality. $\UB_3$ uses $k = |\mathscr{C}|$.}}
\label{table:branching_perinstance}\\
\toprule
                                 &  &  & \multicolumn{3}{r}{pruning depth with $\UB_1$} &  & \multicolumn{3}{r}{pruning depth with $\UB_2$} &  & \multicolumn{3}{r}{pruning depth with $\UB_3$} \\ \cline{4-6} \cline{8-10} \cline{12-14} 
\rmfamily{Instance} &  &  & min  & max  & avg          &  & min  & max & avg           &  & min  & max  & avg          \\
\cline{1-1} \cline{4-6} \cline{8-10} \cline{12-14} 
\endfirsthead
\endhead
\endfoot
\endlastfoot
&  &  &      &      &              &  &      &     &               &  &      &      &              \\[-1.5ex]
brock200\_1      &  &  & 5          & 8   & 6.8          &  & 5          & 7   & \textbf{6.4}  &  & 5    & 8    & 7.0          \\
brock200\_2      &  &  & 3          & 4   & 3.4          &  & 3          & 4   & \textbf{3.2}  &  & 3    & 4    & 3.4          \\
brock200\_3      &  &  & 4          & 5   & 4.7          &  & 4          & 5   & \textbf{4.5}  &  & 5    & 6    & 5.1          \\
brock200\_4      &  &  & 4          & 6   & 4.8          &  & 3          & 6   & \textbf{4.7}  &  & 4    & 6    & 4.9          \\
brock400\_3      &  &  & 6          & 7   & 6.9          &  & 6          & 7   & \textbf{6.7}  &  & 6    & 7    & 6.9          \\
brock400\_4      &  &  & 6          & 8   & 7.1          &  & 6          & 7   & \textbf{6.5}  &  & 6    & 8    & 7.0          \\[1.2ex]
C125.9           &  &  & 7          & 11  & 9.3          &  & 6          & 11  & \textbf{8.6}  &  & 7    & 13   & 9.6          \\[1.2ex]
c-fat200-1       &  &  & 1          & 1   & \textbf{1.0} &  & 1          & 1   & \textbf{1.0}  &  & 1    & 4    & 1.4          \\
c-fat200-2       &  &  & 1          & 1   & \textbf{1.0} &  & 1          & 1   & \textbf{1.0}  &  & 1    & 4    & 2.4          \\
c-fat200-5       &  &  & 1          & 1   & \textbf{1.0} &  & 1          & 2   & 1.2           &           & 2    & 4    & 2.2          \\
c-fat500-1       &  &  & 1          & 1   & \textbf{1.0} &  & 1          & 1   & \textbf{1.0}  &  & 1    & 1    & \textbf{1.0} \\
c-fat500-2       &  &  & 1          & 1   & \textbf{1.0} &  & 1          & 3   & 1.2           &           & 1    & 3    & 1.2          \\
c-fat500-5       &  &  & 1          & 1   & \textbf{1.0} &  & 1          & 4   & 1.3           &           & 2    & 5    & 2.7          \\
c-fat500-10      &  &  & 1          & 1   & \textbf{1.0} &  & 1          & 4   & 1.4           &           & 2    & 4    & 2.7          \\[1.2ex]
DSJC500\_5       &  &  & 4          & 5   & 4.7          &  & 4          & 5   & \textbf{4.3}  &  & 4    & 5    & 4.7          \\
DSJC1000\_5      &  &  & 5          & 6   & 5.2          &  & 5          & 5   & \textbf{5.0}  &  & 5    & 6    & 5.1          \\[1.2ex]
gen200\_p0.9\_44 &  &  & 9          & 14  & 11.4         &  & 8          & 13  & \textbf{10.1} &  & 9    & 14   & 11.1         \\
gen200\_p0.9\_55 &  &  & 6          & 8   & 7.3          &  & 5          & 7   & \textbf{6.1}  &  & 6    & 8    & 7.3          \\[1.2ex]
hamming6-2       &  &  & 3          & 6   & 3.9          &  & 2          & 6   & \textbf{2.5}  &  & 2    & 6    & 3.5          \\
hamming6-4       &  &  & 2          & 3   & \textbf{2.3} &  & 2          & 3   & 2.9           &           & 2    & 3    & 2.9          \\
hamming8-2       &  &  & 11         & 16  & 12.1         &  & 5          & 12  & \textbf{6.2}  &  & 10   & 16   & 11.7         \\
hamming8-4       &  &  & 3          & 5   & 3.9          &  & 3          & 5   & \textbf{3.6}  &  & 3    & 5    & 4.0          \\[1.2ex]
johnson8-2-4     &  &  & 2 & 2   & \textbf{2.0} &  & 2          & 3   & 2.6           &           & 2    & 3    & 2.6          \\
johnson8-4-4     &  &  & 3          & 5   & 3.8          &  & 3          & 7   & \textbf{3.5}  &  & 3    & 7    & 4.0          \\
johnson16-2-4    &  &  & 6          & 6   & \textbf{6.0} &  & 6          & 6   & \textbf{6.0}  &  & 6    & 6    & \textbf{6.0} \\[1.2ex]
keller4          &  &  & 4          & 5   & 4.2          &  & 3 & 4   & \textbf{3.7}  &  & 3    & 5    & 4.0          \\[1.2ex]
MANN\_a9         &  &  & 9          & 12  & 10.1         &  & 5          & 12  & \textbf{8.7}  &  & 9    & 15   & 11.6         \\[1.2ex]
p\_hat300-1      &  &  & 2          & 3   & \textbf{2.3} &  & 2          & 3   & \textbf{2.3}  &  & 2    & 4    & 2.5          \\
p\_hat300-2      &  &  & 2          & 7   & 4.3          &  & 2          & 7   & \textbf{3.7}  &  & 2    & 8    & 4.3          \\
p\_hat300-3      &  &  & 5          & 9   & 6.4          &  & 4          & 7   & \textbf{5.5}  &  & 5    & 7    & 6.1          \\
p\_hat500-1      &  &  & 2          & 4   & 2.5          &  & 2          & 3   & \textbf{2.4}  &  & 2    & 4    & 2.5          \\
p\_hat500-2      &  &  & 2          & 8   & 4.3          &  & 2          & 6   & \textbf{3.7}  &  & 2    & 8    & 4.1          \\
p\_hat700-1      &  &  & 2          & 4   & \textbf{3.1} &  & 2          & 4   & \textbf{3.1}  &  & 3    & 4    & 3.3          \\
p\_hat1000-1     &  &  & 3          & 4   & \textbf{3.2} &  & 3          & 4   & \textbf{3.2}  &  & 3    & 4    & 3.3          \\
p\_hat1500-1     &  &  & 3          & 4   & 3.3          &  & 3          & 4   & \textbf{3.2}  &  & 3    & 4    & 3.5          \\[1.2ex]
san200\_0.7\_1   &  &  & 3          & 5   & 3.5          &  & 2          & 3   & \textbf{2.6}  &  & 3    & 4    & 3.2          \\
san200\_0.7\_2   &  &  & 3          & 8   & 6.5          &  & 2          & 5   & \textbf{3.8}  &  & 3    & 6    & 4.2          \\
san200\_0.9\_1   &  &  & 4          & 8   & 5.9          &  & 3          & 5   & \textbf{4.0}  &  & 4    & 7    & 5.5          \\
san200\_0.9\_2   &  &  & 6          & 8   & 7.1          &  & 5          & 6   & \textbf{5.7}  &  & 6    & 8    & 6.7          \\
san200\_0.9\_3   &  &  & 9          & 11  & 10.3         &  & 7          & 10  & \textbf{8.7}  &  & 9    & 11   & 10.2         \\
san400\_0.5\_1   &  &  & 3          & 5   & 4.0          &  & 2          & 2   & \textbf{2.0}  &  & 2    & 3    & 2.3          \\
san400\_0.7\_1   &  &  & 3          & 9   & 5.5          &  & 3          & 3   & \textbf{3.0}  &  & 3    & 4    & 3.9          \\
san400\_0.7\_2   &  &  & 5          & 8   & 6.3          &  & 3          & 5   & \textbf{3.8}  &  & 3    & 5    & 4.7          \\
san400\_0.7\_3   &  &  & 6          & 8   & 6.8          &  & 3          & 6   & \textbf{4.4}  &  & 4    & 7    & 5.1          \\
san1000          &  &  & 5          & 6   & 5.1          &  & 2          & 2   & \textbf{2.0}  &  & 2    & 3    & 2.8          \\[1.2ex]
sanr200\_0.7     &  &  & 5          & 6   & 5.5          &  & 5          & 6   & \textbf{5.3}  &  & 5    & 6    & 5.6          \\
sanr400\_0.5     &  &  & 4          & 5   & \textbf{4.1} &  & 4          & 5   & \textbf{4.1}  &  & 4    & 5    & \textbf{4.1} \\
sanr400\_0.7     &  &  & 6          & 8   & 7.1          &  & 6          & 8   & \textbf{7.0}  &  & 6    & 8    & 7.2          \\[1.2ex]  \cline{1-1} \cline{4-6} \cline{8-10} \cline{12-14} 
\\[-2ex]
\rmfamily{Avg}      &  &  &      &      & 4.9          &  &      &     & \textbf{4.1}           &  &      &      & 4.8          \\ \bottomrule
\end{longtable}

\normalsize
\Rev{\subsection{A comparative analysis of LP-based upper bounds}\label{sec:LP_comparison}}

\Rev{In this subsection, we compare five LP-based upper bounds to $\optimum$. The first three are variants of the bound $\OSBoundGeneral$, differing in the family $\tilde{\mathscr{I}}$ of independent sets used in
Model~\eqref{form:LP-italia-spagna}. The remaining two are obtained by solving the linear relaxation of two ILP formulations for the EWMCP.}

\Rev{The optimal value $\OSBoundGeneral$ of Model~\eqref{form:LP-italia-spagna} depends on
the family $\tilde{\mathscr{I}}$ of independent sets over which variables $\boldsymbol{\pi}$
are defined. As introduced in \Cref{sec:UB1}, \citet{san2019new} restrict this family to a
coloring $\mathscr{C}$ (i.e., a partition) of the vertex set $V$. It is therefore natural to
ask whether enlarging~$\tilde{\mathscr{I}}$ yields a tighter bound, and at what computational
cost. We consider two extensions of the baseline bound $\UB_1 = \OSBound$, here defined
over a single DSatur coloring. The first, denoted~$\UB_1^{\cup}$, takes $\tilde{\mathscr{I}}$
to be the union of the six colorings (one obtained via DSatur and 5 random colorings) described in \Cref{sec:dimacs_results}, thus allowing vertices to appear in more than one independent set. The second, denoted~$\UB_1^{\star}$, sets $\tilde{\mathscr{I}} = \mathscr{I}$, i.e., the collection of all independent sets of~$G$. Since every independent set is contained in a maximal independent set, it is sufficient computationally to consider only maximal independent sets. In this case, the value of $\OSBall$ coincides with the optimal value of the linear relaxation of Model~\eqref{form:EWMCP_ILP}. The computation of $\OSBall$ requires solving a formulation with an exponential number of variables, which we handle via column generation. The corresponding pricing subproblems are Maximum Weighted Independent Set problems with suitably defined vertex weights, which we solve by finding a Maximum Vertex Weighted Clique on the complement graph of $G$, using the TSM algorithm \citep{Jiang_Li_Liu_Manya_2018}.

We place these variants within a broader landscape of LP-based bounds by introducing two
further bounds derived from solving the linear relaxation of two ILP formulations for the EWMCP. The first formulation is denoted as $\Fone$, and it is the natural ILP formulation of the problem, adapted from the one originally proposed by \citet{park1996extended}. This formulation corresponds to Model \eqref{form:EWMCP_ILP}, where Constraints \eqref{eq:form:EWMCP_ILP:constrIndependents} are defined over the set of all non-edges of $G$, i.e., over $\bar{E} = \bigl\{\{u,v\} : \{u,v\} \notin E\bigr\}$. Notice that the dual of the linear relaxation of F1 coincides with Model \eqref{form:LP-italia-spagna}, having $\tilde{\mathscr{I}} = \bar{E}$.

The second formulation, detailed in \citet{gouveia2015solving}, is denoted as $\Feleven$. It is obtained from ${\Fone}$ by adding the following two classes of valid inequalities to the model:}

\begin{equation} \label{eq:valid_ineq_EWMCP}
\sum_{e \in \delta(u)} y_e \le (k - 1) \, x_u, \quad u \in V, \qquad \text{ and } \qquad \sum_{u \in V} x_u \le k,
\end{equation}

\Rev{where $k$ is any upper bound on the clique number of $G$. The first family of inequalities ensures that an edge can be selected only if its endpoints are both in the clique, and limits the number of selected edges incident to each node to at most $k-1$. The second constraint imposes instead a global upper bound of $k$ on the number of selected vertices. \citet{gouveia2015solving} showed that the linear relaxation of $\Feleven$ is not dominated by that of any other compact ILP formulation for the EWMCP they consider, in the sense that neither provides a tighter bound on every instance.

The upper bounds obtained by solving the linear relaxations of formulations $\Fone$ and $\Feleven$ are henceforth referred to as $\UB_{\Fone}$ and $\UB_{\Feleven}$, respectively. In the following computational experiments, the value of $k$ for Constraints \eqref{eq:valid_ineq_EWMCP} is set to the cardinality of the coloring obtained via the DSatur algorithm, for all the tested instances.}

\Rev{\Cref{table:LP-based_stats} reports the performance of all five upper bounds on 43 DIMACS
instances with known optimal values. Instances are grouped by family; for each family, the table reports the number of instances, the average relative ratio (the ratio of each bound's value to the minimum one across all five bounds on the same instance), and the average time ratio (the ratio of each bound's computing time to the minimum one across all five bounds on the same instance). A value of 1.00 in either of these ratios indicates that the corresponding bound achieves the best value, respectively the shortest computing
time, among the five upper bounds on that instance. Four instances of the \texttt{p\_hat} family, as well as the \texttt{DSJC1000\_5} instance, are excluded from the analysis since the column generation procedure for $\UB_1^\star$ did not converge within a 20-hour time limit.}

\Rev{The results reveal a clear hierarchy. $\UB_{\Fone}$ is the weakest of the five upper bounds by a wide margin: its average relative ratio of 22.62 reflects that, on most instances, it
is far from the tightest among the 5 LP-based upper bounds considered. This is most pronounced on the \texttt{san} family,
where $\UB_{\Fone}$ is on average more than 77 times larger than the best bound, confirming that
restricting $\tilde{\mathscr{I}}$ to pairs is a poor choice in practice. Its computing time
is comparable to that of $\UB_1$ (average time ratio 2.65), so the bound weakness is not offset
by any computational advantage.
Within the $\OSBoundGeneral$ family, enlarging $\tilde{\mathscr{I}}$ yields a diminishing improvement in bound tightness. Moving from $\UB_1$ (average relative ratio 2.31) to $\UB_1^{\cup}$ (2.17) slightly reduces
the gap, at a moderate computational cost of 2.64 times in terms of computing time. Extending to
$\UB_1^{\star}$ (having an average relative ratio of 2.02) achieves a further modest improvement, but with an average time ratio of around 1330 -- dominated by the \texttt{p\_hat} family, where column generation requires on average over 13000 times the minimum computing time. The gain in bound tightness from $\UB_1^{\cup}$ to
$\UB_1^{\star}$ is therefore not justified by the additional computational effort on almost all the instance families.
The bound $\UB_{\Feleven}$ stands apart from all the other LP-based bounds considered above.
With an average relative ratio of 1.02, it is the tightest bound on basically every instance, achieving an average ratio of exactly 1.00 on eight of the twelve families. It reaches this quality at an average time ratio of 10.77, substantially less than $\UB_1^{\star}$ and comparable to $\UB_1^{\cup}$ on most families, with the exception of \texttt{c-fat}, \texttt{hamming}, \texttt{johnson} and \texttt{MANN}, where its time ratio reaches high values. The only family where
$\UB_{\Feleven}$ shows a significant performance gap with respect to the best bound is \texttt{keller}, where it achieves a relative ratio of 1.28, while $\UB_1^{\star}$ reaches 1.00.}

\Rev{We observe that three variants of $\OSBoundGeneral$ ($\UB_1$, $\UB_1^\cup$ and $\UB_1^\star$) differ substantially in the number of independent sets they employ (i.e., on the value of $|\tilde{\mathscr{I}}|$): on the 43 considered instances, $\UB_1$ uses on average 48.3 independent sets, corresponding to a single DSatur coloring.
$\UB_1^{\cup}$ enlarges this pool adding the independent sets from the five random colorings and reaching an average of 314.9 sets (approximately, a 6x increase that is stable across families). $\UB_1^{\star}$, by contrast, requires substantially fewer independent sets: on most families the column generation procedure adds only a moderate number of sets, yielding an average of 133.9, well below $\UB_1^{\cup}$. The main exception is the \texttt{p\_hat}
family, where $\UB_1^{\star}$ reaches an average of 482.5 sets, surpassing $\UB_1^{\cup}$ (369.2), and where the column generation procedure did not terminate within the time limit on four of the eight instances. The high computational cost of $\UB_1^\star$ is therefore not due to a large number of generated columns -- on most instances, the column generation procedure terminates after adding relatively few independent sets -- but rather to the cost of solving the pricing subproblem at each iteration.}

\Rev{Two main conclusions follow from these results. First, the gap between $\UB_1$ and stronger bounds
within the $\OSBoundGeneral$ family seems not to be closed by enlarging $\tilde{\mathscr{I}}$
alone: even $\UB_1^{\star}$, defined over the full independent set family $\mathscr{I}$, retains an average
relative ratio of 2.02 at an impractical computational cost. Second, the valid inequalities
of $\Feleven$ significantly strengthen the associated bound: $\UB_{\Feleven}$ is tighter than
$\UB_1^{\star}$ at a fraction of its computing time. Nevertheless, we observe that even $\UB_{\Feleven}$, the strongest LP-based bound considered here, is dominated by $\UB_2$ on 41 out of 43 instances. On the two exceptions (\texttt{c-fat500-1} and \texttt{hamming6-2}), the difference between the two bounds is
negligible (below 2.5\%). On average, the LP-based bounds examined in this subsection are therefore both slower and weaker than $\UB_2$.}

\begin{table}[H]
\centering
\footnotesize
\renewcommand\arraystretch{1.2}
\tabcolsep=3.2pt
\caption{\Rev{Performance of the five LP-based upper bounds on 43 DIMACS instances with known optimum, grouped by family. $\UB_1$ uses a single DSatur coloring $\mathscr{C}$; $\UB_1^{\cup}$ uses the union of six colorings (one DSatur and 
five random); $\UB_1^{\star}$ uses the family $\mathscr{I}$ of all independent sets of $G$. For $\UB_{\Feleven}$, the parameter $k$ in Constraints~\eqref{eq:valid_ineq_EWMCP} is set to $|\mathscr{C}|$, where $\mathscr{C}$ is the DSatur coloring. For each family, the table reports the number of instances, along with the average relative ratio and the average time ratio of each bound.}}
\label{table:LP-based_stats}
\Rev{\begin{tabular}{Trrrrrrrrrrrrrrrrr}
\hline
\multicolumn{2}{l}{\rmfamily Instances} &  &  & \multicolumn{2}{r}{$\UB_1$} &  & \multicolumn{2}{r}{$\UB_1^{\cup}$} &  & \multicolumn{2}{r}{$\UB_1^{\star}$} &  & \multicolumn{2}{r}{$\UB_{\Fone}$} &  & \multicolumn{2}{r}{$\UB_{\Feleven}$} \\ \cline{5-6} \cline{8-9} \cline{11-12} \cline{14-15} \cline{17-18} 
                 &            &  &  & \multicolumn{2}{r}{ratio}  &  & \multicolumn{2}{r}{ratio}         &  & \multicolumn{2}{r}{ratio}          &  & \multicolumn{2}{r}{ratio}         &  & \multicolumn{2}{r}{ratio}            \\ \cline{1-2} \cline{5-6} \cline{8-9} \cline{11-12} \cline{14-15} \cline{17-18} 
\rmfamily{family}           & \#         &  &  & relative       & time      &  & relative          & time          &  & relative           & time          &  & relative          & time          &  & relative              & time         \\ \cline{1-2} \cline{5-6} \cline{8-9} \cline{11-12} \cline{14-15} \cline{17-18} 
\\[-3.5ex]
brock            & 6          &  &  & 1.51           & 1.00      &  & 1.33              & 3.39          &  & 1.07               & 166.85        &  & 3.96              & 2.10          &  & \textbf{1.01}         & 4.43         \\
C                & 1          &  &  & 1.42           & 1.00      &  & 1.26              & 2.29          &  & 1.21               & 43.89         &  & 1.96              & 2.01          &  & \textbf{1.00}         & 3.40         \\
c-fat            & 7          &  &  & 1.10           & 1.00      &  & 1.08              & 2.65          &  & 1.08               & 30.30         &  & 6.97              & 6.29          &  & \textbf{1.07}         & 20.35        \\
DSJC & *1 & & & 1.80 & 1.08 && 1.63 & 2.37 && 1.20 & 1069.34 && 7.48 & 1.00 && \textbf{1.00} & 2.74 \\
gen              & 2          &  &  & 1.34           & 1.00      &  & 1.19              & 2.80          &  & 1.08               & 41.40         &  & 2.40              & 2.32          &  & \textbf{1.01}         & 5.36         \\
hamming          & 4          &  &  & 2.05           & 1.00      &  & 1.91              & 2.23          &  & 1.90               & 14.94         &  & 8.63              & 2.07          &  & \textbf{1.00}         & 19.81        \\
johnson          & 3          &  &  & 1.97           & 1.00      &  & 1.68              & 1.77          &  & 1.65               & 47.93         &  & 8.06              & 1.60          &  & \textbf{1.00}         & 15.77        \\
keller           & 1          &  &  & 2.04           & 1.00      &  & 1.17              & 2.71          &  & \textbf{1.00}      & 43.90         &  & 6.87              & 1.89          &  & 1.28                  & 4.90         \\
MANN             & 1          &  &  & 1.43           & 1.00      &  & 1.39              & 1.92          &  & 1.39               & 52.03         &  & 2.06              & 1.77          &  & \textbf{1.00}         & 37.22        \\
p\_hat           & *4          &  &  & 1.79           & 1.00      &  & 1.62              & 2.59          &  & 1.15               & 13296.40      &  & 7.64              & 1.77          &  & \textbf{1.00}         & 12.90        \\
san              & 10         &  &  & 4.71           & 1.00      &  & 4.66              & 2.49          &  & 4.55               & 19.80         &  & 77.36             & 1.99          &  & \textbf{1.00}         & 4.44         \\
sanr             & 3          &  &  & 1.56           & 1.00      &  & 1.37              & 3.42          &  & 1.08               & 377.32        &  & 4.65              & 2.01          &  & \textbf{1.00}         & 4.15         \\[1.2ex] \cline{1-2} \cline{5-6} \cline{8-9} \cline{11-12} \cline{14-15} \cline{17-18} 
\rmfamily{Tot/Avg}            & 43         &  &  & 2.31           & 1.00      &  & 2.17              & 2.64          &  & 2.02               & 1330.80       &  & 22.62             & 2.65          &  & \textbf{1.02}         & 10.77        \\ \bottomrule
&&&&&&&&&&& \multicolumn{7}{r}{\scriptsize *subset of instances with known $\UB_1^\star$}
\end{tabular}}
\end{table}

\normalsize

\RevTwo{For completeness, we also considered the best value of $\UB_1$ obtained across the DSatur coloring and the five random colorings. This best-over-six strategy provides only a limited improvement over the single-coloring baseline: its average relative ratio decreases from 2.31 to 2.28, with a tighter bound on just 9 of the 43 instances (in the remaining 34 cases, the tightest bound is achieved by DSatur). This confirms the observation made in the previous sections that DSatur is generally more effective than the random colorings.
In contrast, pooling the independent sets from all six colorings yields a more substantial improvement, reducing the average relative ratio to 2.17 and producing a bound that is tighter than the best-over-six value on 38 instances (and equal on the remaining 5).
The complete best-over-six results are available in the online computational material.}

\section{Conclusions}
\label{sec:CONCL}
In this study, we conducted the first theoretical and computational comparison of the \Rev{three} main upper bounds on~$\optimum$, the optimal value of the EWMCP, proposed independently by \citet{san2019new}, \citet{shimizu2020maximum}\Rev{, and \citet{hosseinian2020lagrangian}}.
Our theoretical analysis shows that none of the three upper bounds admits a performance guarantee: for each, there exist instances where its ratio to the optimal EWMCP value is unbounded.
Moreover, \Rev{for every choice of a ratio between two of the three upper bounds, we identified families of instances for which such ratio achieves the largest possible value. A summary of all the results can be found in \Cref{fig:result_summary}.}

\begin{center}
\captionsetup{type=figure}
\captionof{figure}{\Rev{Summary of the theoretical relationships among the three upper bounds and the edge-weighted clique number~$\optimum$, all evaluated with respect to a given coloring~$\mathscr{C}$ (with $k = |\mathscr{C}|$ for~$\mathrm{UB}_3$).
An arrow from $\optimum$ to an upper bound is labeled with the largest possible value of the ratio of the upper bound to $\optimum$.
An arrow is drawn from an upper bound to the other if the former can dominate the latter.
Moreover, the label near the head of the arrow reports the largest possible ratio of the head to the tail. The proposition proving each result is referenced next to the corresponding label.
}}\label{fig:result_summary}
\begin{tikzpicture}[
    >=stealth,
    thick,
    every node/.style={font=\small, fill=none},
    lbl/.style={font=\normalsize, fill=white, inner sep=2pt, rounded corners=1pt},
    bound node/.style={circle, draw, minimum size=32pt, inner sep=1pt, font=\normalsize, line width=0.8pt},
    myarrow/.style={black, line width=1pt},
  ]



  \node[bound node, draw=black, fill=white, font=\normalsize, minimum size=46pt] (omega) at (0,0) {$\optimum$};

  \node[bound node, draw=black] (UB3) at (90:3.5) {$\mathrm{UB}_3$};
  \node[bound node, draw=black] (UB1) at (210:3.5) {$\mathrm{UB}_1$};
  \node[bound node, draw=black] (UB2) at (330:3.5) {$\mathrm{UB}_2$};

  
  \draw[-{stealth}, myarrow] (omega) -- (UB1) node[lbl, pos=0.5] {$\infty$}
  node[lbl, pos=0.5, xshift = -0.6cm, yshift=0.3cm] {\tiny Prop.~\ref{prop:example_bad_itaspa}};
  
  \draw[-{stealth}, myarrow] (omega) -- (UB2) node[lbl, pos=0.5] {$\infty$}
  node[lbl, pos=0.5, xshift = 0.5cm, yshift=0.3cm] {\tiny Prop.~\ref{prop:example_bad_japan}};

  \draw[-{stealth}, myarrow] (omega) -- (UB3) node[lbl, pos=0.5] {$\infty$}
  node[lbl, pos=0.5, xshift=0.8cm] {\tiny Prop.~\ref{prop:example_bad_hfb}};
 
  \draw[{stealth}-{stealth}, myarrow, bend left=30] (UB1) to 
    node[lbl, pos=0.15] {$\infty$}
    node[lbl, pos=0.85] {$\infty$}
    node[lbl, pos=0.15, xshift=-0.8cm] {\tiny Prop.~\ref{prop:os_over_hfb}}
    node[lbl, pos=0.85, xshift=-0.9cm] {\tiny Prop.~\ref{prop:hfb_over_os}}
    (UB3);

  \draw[{stealth}-{stealth}, myarrow, bend left=30] (UB3) to 
    node[lbl, pos=0.15] {$\infty$}
    node[lbl, pos=0.85] {$2$}
    node[lbl, pos=0.15, xshift=0.9cm] {\tiny Prop.~\ref{prop:hfb_over_jap}}
    node[lbl, pos=0.85, xshift=0.8cm] {\tiny Prop.~\ref{prop:UB2_leq_2UB3},\ref{prop:tightness_UB2_UB3}}
    (UB2);

  \draw[{stealth}-{stealth}, myarrow, bend right=30] (UB1) to 
    node[lbl, pos=0.15] {$\infty$}
    node[lbl, pos=0.85] {$\infty$}
    node[lbl, pos=0.15, yshift=-0.45cm] {\tiny Prop.~\ref{prop:example_japan_wins}}
    node[lbl, pos=0.85, yshift=-0.45cm] {\tiny Prop.~\ref{prop:example_itaspa_wins}}
    (UB2);  
    
  \end{tikzpicture}
\end{center}

In particular, the bound proposed by \citet{shimizu2020maximum} appears to perform better on instances where large independent sets are connected by a few heavy edges, while the one introduced by \citet{san2019new} tends to yield tighter values when the coloring includes a larger number of independent sets.
\Rev{The third bound of \citet{hosseinian2020lagrangian} shows qualitative characteristics similar to those of \citet{shimizu2020maximum}, and can improve it by a factor of at most 2, in certain instances. Still, it can be arbitrarily worse than the other two bounds in specific families of instances.}
These observations were corroborated by our empirical analysis, showing that the \Rev{first} two upper bounds behave differently depending on the graph's structure.
While they perform similarly on small or sparse instances, the upper bound by \citet{shimizu2020maximum} proves tighter on large, moderately dense graphs, especially around 50\% density.
\Rev{In the experiments, the third upper bound tends to perform middle of the road between the other two. It is tighter than the one of \citet{san2019new} for intermediate densities, while it is looser in very sparse and very dense instances.} This pattern emerges consistently across the randomly generated graphs, \Rev{and similar behavior is also observed on the DIMACS benchmark set,}
highlighting the importance of accounting for graph size and density when selecting or designing bounding strategies for the EWMCP. 

\Rev{An additional practical outcome of our study is that the quality of $\UB_2$ can be significantly improved by appropriately ordering the independent sets of the coloring, and sorting them in non-decreasing order of cardinality turns out to be the best overall strategy on the DIMACS benchmark.}

\noindent
These findings open several lines for future research, including the design of hybrid upper bounds that leverage the complementary strengths of the three methods based on instance characteristics, as well as their effective integration into exact algorithms such as branch-and-bound frameworks.

\section*{Acknowledgments}
\RevTwo{The authors thank the anonymous Reviewer for the careful reading and constructive comments provided throughout the review process, which helped improve the overall quality and clarity of the manuscript.}

\bibliographystyle{abbrvnat}
\bibliography{biblio}

\end{document}